\documentclass[11pt]{amsart}

\usepackage{amssymb}
\usepackage{upgreek}
\usepackage{amsmath}
\usepackage[all,cmtip]{xy}
\usepackage{enumitem}

\usepackage{amsfonts}
\usepackage{mathrsfs}
\usepackage{latexsym}
\usepackage{graphicx}
\usepackage{amscd,amssymb,amsmath,amsbsy,amsthm}
\usepackage[colorlinks,plainpages,backref,urlcolor=blue]{hyperref}
\allowdisplaybreaks

\usepackage{tikz}
\usetikzlibrary{arrows,calc}
\tikzset{
>=stealth',
help lines/.style={dashed, thick},
axis/.style={<->},
important line/.style={thick},
connection/.style={thick, dotted},
}

\newcommand{\nc}{\newcommand}
\nc{\rnc}{\renewcommand}
\nc{\bb}[1]{{\mathbb #1}}
\nc{\bbA}{\bb{A}}\nc{\bbB}{\bb{B}}\nc{\bbC}{\bb{C}}\nc{\bbD}{\bb{D}}
\nc{\bbE}{\bb{E}}\nc{\bbF}{\bb{F}}\nc{\bbG}{\bb{G}}\nc{\bbH}{\bb{H}}
\nc{\bbI}{\bb{I}}\nc{\bbJ}{\bb{J}}\nc{\bbK}{\bb{K}}\nc{\bbL}{\bb{L}}
\nc{\bbM}{\bb{M}}\nc{\bbN}{\bb{N}}\nc{\bbO}{\bb{O}}\nc{\bbP}{\bb{P}}
\nc{\bbQ}{\bb{Q}}\nc{\bbR}{\bb{R}}\nc{\bbS}{\bb{S}}\nc{\bbT}{\bb{T}}
\nc{\bbU}{\bb{U}}\nc{\bbV}{\bb{V}}\nc{\bbW}{\bb{W}}\nc{\bbX}{\bb{X}}
\nc{\bbY}{\bb{Y}}\nc{\bbZ}{\bb{Z}}
\nc{\mbf}[1]{{\mathbf #1}}
\nc{\bfA}{\mbf{A}}\nc{\bfB}{\mbf{B}}\nc{\bfC}{\mbf{C}}\nc{\bfD}{\mbf{D}}
\nc{\bfE}{\mbf{E}}\nc{\bfF}{\mbf{F}}\nc{\bfG}{\mbf{G}}\nc{\bfH}{\mbf{H}}
\nc{\bfI}{\mbf{I}}\nc{\bfJ}{\mbf{J}}\nc{\bfK}{\mbf{K}}\nc{\bfL}{\mbf{L}}
\nc{\bfM}{\mbf{M}}\nc{\bfN}{\mbf{N}}\nc{\bfO}{\mbf{O}}\nc{\bfP}{\mbf{P}}
\nc{\bfQ}{\mbf{Q}}\nc{\bfR}{\mbf{R}}\nc{\bfS}{\mbf{S}}\nc{\bfT}{\mbf{T}}
\nc{\bfU}{\mbf{U}}\nc{\bfV}{\mbf{V}}\nc{\bfW}{\mbf{W}}\nc{\bfX}{\mbf{X}}
\nc{\bfY}{\mbf{Y}}\nc{\bfZ}{\mbf{Z}}
\nc{\bfa}{\mbf{a}}\nc{\bfb}{\mbf{b}}\nc{\bfc}{\mbf{c}}\nc{\bfd}{\mbf{d}}
\nc{\bfe}{\mbf{e}}\nc{\bff}{\mbf{f}}\nc{\bfg}{\mbf{g}}\nc{\bfh}{\mbf{h}}
\nc{\bfi}{\mbf{i}}\nc{\bfj}{\mbf{j}}\nc{\bfk}{\mbf{k}}\nc{\bfl}{\mbf{l}}
\nc{\bfm}{\mbf{m}}\nc{\bfn}{\mbf{n}}\nc{\bfo}{\mbf{o}}\nc{\bfp}{\mbf{p}}
\nc{\bfq}{\mbf{q}}\nc{\bfr}{\mbf{r}}\nc{\bfs}{\mbf{s}}\nc{\bft}{\mbf{t}}
\nc{\bfu}{\mbf{u}}\nc{\bfv}{\mbf{v}}\nc{\bfw}{\mbf{w}}\nc{\bfx}{\mbf{x}}
\nc{\bfy}{\mbf{y}}\nc{\bfz}{\mbf{z}}

\nc{\mcal}[1]{{\mathcal #1}}
\nc{\calA}{\mcal{A}}\nc{\calB}{\mcal{B}}\nc{\calC}{\mcal{C}}\nc{\calD}{\mcal{D}}
\nc{\calE}{\mcal{E}} \nc{\calF}{\mcal{F}}\nc{\calG}{\mcal{G}}\nc{\calH}{\mcal{H}}
\nc{\calI}{\mcal{I}}\nc{\calJ}{\mcal{J}}\nc{\calK}{\mcal{K}}\nc{\calL}{\mcal{L}}
\nc{\calM}{\mcal{M}}\nc{\calN}{\mcal{N}}\nc{\calO}{\mcal{O}}\nc{\calP}{\mcal{P}}
\nc{\calQ}{\mcal{Q}}\nc{\calR}{\mcal{R}}\nc{\calS}{\mcal{S}}\nc{\calT}{\mcal{T}}
\nc{\calU}{\mcal{U}}\nc{\calV}{\mcal{V}}\nc{\calW}{\mcal{W}}\nc{\calX}{\mcal{X}}
\nc{\calY}{\mcal{Y}}\nc{\calZ}{\mcal{Z}}
\nc{\fA}{\frak{A}}\nc{\fB}{\frak{B}}\nc{\fC}{\frak{C}} \nc{\fD}{\frak{D}}
\nc{\fE}{\frak{E}}\nc{\fF}{\frak{F}}\nc{\fG}{\frak{G}}\nc{\fH}{\frak{H}}
\nc{\fI}{\frak{I}}\nc{\fJ}{\frak{J}}\nc{\fK}{\frak{K}}\nc{\fL}{\frak{L}}
\nc{\fM}{\frak{M}}\nc{\fN}{\frak{N}}\nc{\fO}{\frak{O}}\nc{\fP}{\frak{P}}
\nc{\fQ}{\frak{Q}}\nc{\fR}{\frak{R}}\nc{\fS}{\frak{S}}\nc{\fT}{\frak{T}}
\nc{\fU}{\frak{U}}\nc{\fV}{\frak{V}}\nc{\fW}{\frak{W}}\nc{\fX}{\frak{X}}
\nc{\fY}{\frak{Y}}\nc{\fZ}{\frak{Z}}
\nc{\fa}{\frak{a}}\nc{\fb}{\frak{b}}\nc{\fc}{\frak{c}} \nc{\fd}{\frak{d}}
\nc{\fe}{\frak{e}}\nc{\fFf}{\frak{f}}\nc{\fg}{\frak{g}}\nc{\fh}{\frak{h}}
\nc{\fri}{\frak{i}}\nc{\fj}{\frak{j}}\nc{\fk}{\frak{k}}\nc{\fl}{\frak{l}}
\nc{\fm}{\frak{m}}\nc{\fn}{\frak{n}}\nc{\fo}{\frak{o}}\nc{\fp}{\frak{p}}
\nc{\fq}{\frak{q}}\nc{\fr}{\frak{r}}\nc{\fs}{\frak{s}}\nc{\ft}{\frak{t}}
\nc{\fu}{\frak{u}}\nc{\fv}{\frak{v}}\nc{\fw}{\frak{w}}\nc{\fx}{\frak{x}}
\nc{\fy}{\frak{y}}\nc{\fz}{\frak{z}}

\newtheorem{theorem}{Theorem}[section]
\newtheorem{lemma}[theorem]{Lemma}
\newtheorem{corollary}[theorem]{Corollary}
\newtheorem{prop}[theorem]{Proposition}

\theoremstyle{definition}
\newtheorem{definition}[theorem]{Definition}
\newtheorem{example}[theorem]{Example}
\newtheorem{remark}[theorem]{Remark}

\newtheorem{thm}{Theorem}

\DeclareMathOperator{\im}{im} 
 \DeclareMathOperator{\id}{id}

\DeclareMathOperator{\Hom}{{Hom}} 

\DeclareMathOperator{\sHom}{{\mathscr{H}om}}

 \DeclareMathOperator{\End}{End}

\DeclareMathOperator{\Coh}{Coh}

\DeclareMathOperator{\Ell}{\mathcal{E}}
\DeclareMathOperator{\Pic}{Pic}

\newcommand{\catA}{\mathfrak{A}}

\newcommand{\al}{\alpha}
\newcommand{\la}{\lambda}
\newcommand{\ga}{\gamma}

\newcommand{\de}{\delta}

\newcommand{\be}{\beta}

\newcommand{\dyn}{{\operatorname{d}}}
\newcommand{\vor}{{\operatorname{vor}}}
\DeclareMathOperator{\Kr}{{{Kr}}}

\newcommand{\ep}{\epsilon}

\newcommand{\lag}{\langle}
\newcommand{\rag}{\rangle}

\newcommand{\tensor}{\star}
\newcommand{\stone}{{\tensor'}} %
\newcommand{\sttwo}{{\tensor''}} 
\newcommand{\dynst}{{\dyn*}}

\newcommand{\up}{\Uptheta}

\newcount\cols
{\catcode`,=\active\catcode`|=\active
 \gdef\Young(#1){\hbox{$\vcenter
 {\mathcode`,="8000\mathcode`|="8000
  \def,{\global\advance\cols by 1 &}%
  \def|{\cr
        \multispan{\the\cols}\hrulefill\cr
        &\global\cols=2 }%
  \offinterlineskip\everycr{}\tabskip=0pt
  \dimen0=\ht\strutbox \advance\dimen0 by \dp\strutbox
  \halign
   {\vrule height \ht\strutbox depth \dp\strutbox##
    &&\hbox to \dimen0{\hss$##$\hss}\vrule\cr
    \noalign{\hrule}&\global\cols=2 #1\crcr
    \multispan{\the\cols}\hrulefill\cr%
   }
 }$}}
}

\newcommand{\cha}[1]{\textcolor{blue}{$[$ Changlong: #1 $]$}}

\title[elliptic  classes]
{Elliptic classes via the periodic Hecke module and its Langlands dual}

\author[C.~Lenart]{Cristian~Lenart}
\address{State University of New York at Albany, 1400 Washington Avenue, Albany, NY 12222}
\email{clenart@albany.edu}

\author[G.~Zhao]{Gufang~Zhao}
\address{The University of Melbourne, School of Mathematics and Statistics, 813 Swanston Street, Parkville VIC 3010, Australia}
\email{gufangz@unimelb.edu.au}

\author[C.~Zhong]{Changlong~Zhong}
\address{State University of New York at Albany, 1400 Washington Avenue, Albany, NY 12222}
\email{czhong@albany.edu}

\date{\today}
\begin{document}

\begin{abstract}
This paper explores a construction of the elliptic classes of the Springer resolution using the periodic Hecke module. The module is established by employing the Poincar\'e line bundle over the product of the abelian variety of elliptic cohomology and its dual. Additionally, we introduce the elliptic twisted group algebra, which acts on the periodic module. The construction of the elliptic twisted group algebra is such that the Demazure-Lusztig (DL) operators with dynamical parameters are rational sections. 
We define elliptic classes as rational sections of the periodic module, and  give explicit formulas of the restriction  to fixed points. Our main result shows that a natural assembly of the  DL operators defines a  rational isomorphism between the periodic module and the one associated to the Langlands dual root system. This isomorphism intertwines the elliptic classes with the fixed point basis of the Langlands dual root datum.
\end{abstract}
\maketitle

\section{Introduction}
Elliptic classes are a type of mathematical object that arise in the study of symplectic resolutions and their associated structures. They are essentially rational sections of certain vector bundles on abelian varieties, and they have  many interesting properties and applications.
In recent years, there has been a growing interest in understanding the behavior of elliptic classes in the context of Springer resolutions. Notably, Aganagi\'c and Okounkov \cite{AO21} following the idea of stable envelope, and Rim\'anyi and Weber \cite{RW20, RW22}, following Borisov and Libgober \cite{BL03} on elliptic genus. 
It is worth pointing out that in the latter approach, based on the Bott-Samelson resolution of Schubert varieties,  an iteration formula of the elliptic classes is obtained in terms of elliptic Demazure-Lusztig operators with dynamical parameters.

In this paper, we explore the elliptic classes via a purely algebraic approach using the structures of  
root systems and the Poincar\'e line bundle on an elliptic curve.
We start with an algebraic definition of the (localized) elliptic Hecke algebras and the periodic Hecke modules. The elliptic classes are rational sections of the periodic modules. 
We give explicit formulas of the restriction of elliptic classes on fixed points. 

Moreover, the definition of elliptic Hecke algebra with dynamical parameters is well-behaved with respect to Langlands duality. 
We formulate such relation in terms of the periodic Hecke modules. As a consequence, we deduce some interesting relations between elliptic classes of
 Langlands dual root datum. In the subsequent subsections of the introduction, we present a summary account of our findings, providing  detailed descriptions and explanations.

\subsection{Elliptic Demazure-Lusztig (DL) operators  with dynamical operators}
The idea of elliptic  Hecke algebra and elliptic DL operators goes back to Ginzburg, Kapranov, and Vasserot \cite{GKV97} (see also \cite{GKV95, Ga12}). Rim\'anyi and Weber realized that after adding the dynamical parameters, there are  analogues of the Demazure-Lusztig operators, satisfying the braid relations. 

We reformulate Rim\'anyi-Weber's definition in terms of an elliptic Hecke algebra with dynamical parameters, using only the elliptic curve, the
Poincar\'e line bundle, and the root datum as input. Following directly from this construction  is the  polynomial representation, and more closely related to this paper, the periodic module. 

More precisely, let $E$ be an elliptic curve over $\bbC$. 
Let $X^*$ and $X_*$ be the two dual lattices of  the root datum. Define $\catA=X_*\otimes E$, then the dual of this abelian variety is $\catA^\vee\cong X^*\otimes E$. 
The Weyl group $W$ acts on both $\catA$ and $\catA^\vee$, where the latter action is denoted by $W^\dyn$.

Let $\bbL$ be (a certain $\rho$-shift of) the Poincar\'e line bundle on $\catA\times\catA^\vee$ (Definition \ref{def:bbL}). 
Utilizing the Weyl group actions, we obtain a new line bundle  \[\bbS_{w,v} := \bbL \otimes (w^{-1})^*(v^{-1})^{\dynst}\bbL^{-1}\] for $v, w \in W$.
Define $\bbS$ as the direct sum \[\bbS = \bigoplus_{v, w \in W} \bbS_{w,v},\] referred to as the elliptic twisted group algebra.
The algebra structure is to be understood within a suitable monoidal category of coherent sheaves (see \S~\ref{sec:bbS} for detailed explanations).
Considering the vector space of rational sections, the algebra structure corresponds to the well-known Kostant-Kumar twisted product, augmented by the additional dynamical Weyl group action (refer to \S~\ref{subsec:Sproductrational} for further details).
A key feature of this construction is the existence of a rational section of $\bbS$ associated with each simple root $\alpha$, referred to as the elliptic Demazure-Lusztig operator denoted by $T^{z,\lambda}_\alpha$ (abbreviated as $T_\alpha$ in this introduction).
The appearance of these operators was first noted in \cite{RW20} in an iteration formula for the elliptic classes of Schubert varieties. They satisfies the braid relations, ensuring that $T_w$ is well-defined for any $w \in W$.

By definition, the algebra $\bbS$ acts on a module called the periodic module, defined as \[\bbM = \bigoplus_{w \in W} w^*\bbL.\]
The module $\bbM$ can be viewed as the equivariant elliptic cohomology of the Springer resolution $T^*G/B$.
To establish the fixed point class associated with $w \in W$, we fix a rational section $\fc$ of $\bbL$. We denote the fixed point class as \[\tilde{f}_w = {}^{w^{-1}}\fc,\] which is a section of $\bbM_w = w^*\bbL$.
Furthermore, we define the elliptic class for $w \in W$ as \[\calE_v=T_{vw_0}\tilde f_{w_0}.\]
They coincide with the classes defined by the elliptic stable envelop of Aganagic-Okounkov for suitably chosen parameters (chamber and polarization) \cite{AO21}. 


\subsection{Elliptic classes and Langlands dual}

One of the key properties of the periodic module is that it is equipped with a natural pairing, called the Poincar\'e pairing, which is a bilinear form on the  periodic modules associated with a given root system. 
The precise definition of the Poincar\'e pairings involves a sequence of
duality functors, and  dual algebras and modules, elaborated in \S~\ref{sec:dual}.  They are used to prove  dualities between the elliptic classes.  Moreover,  these functors and modules allow us to relate the periodic modules associated with a given root system to those associated with its Langlands dual, and to derive various properties of these modules and their associated structures.

More precisely, the periodic module associated to the Langlands dual root datum is also a vector bundle on $\catA\times\catA^\vee$, defined by \[\bbM^\dyn :=\bigoplus_{w\in W} w^{\dyn*}\bbL.\]
Indeed, in the definition of $\bbM$ and $\bbM^\dyn$ there is an asymmetry between $\catA$ and $\catA^\vee$. This asymmetry  gives a notion of dynamical periodic module  $\bbM^\dyn$ (\S~\ref{sec:bbS}),   which we identify as the periodic module associated to the Langlands dual root datum (\S~\ref{sec:Langlands}).
The module $\bbM^\dyn $ contains similarly defined fixed-point classes $\tilde f^\dyn _w$ for $w\in W$, and is acted on by the operators $T^\dyn _w$ for $w\in W$, where $T_w^\dyn$ are rational sections of the same algebra $\bbS$. These operators define elliptic classes $\calE^\dyn_w$ of the Langlands dual root datum.

Naturally, the operators $T_w$ assemble into a rational map of vector bundles on $\catA\times\catA^\vee$ (Definition~\ref{def:bbT})
\[\bbT:\bbM\dashrightarrow\bbM^\dyn.\]
Similarly,  the operators $T^\dyn _v$ assemble into a rational map $\bbT^\dyn$.
Now we state the main theorem regarding elliptic classes of Langlands dual root datum. 
\begin{thm}[Theorems \ref{thm:inverse_maps},  \ref{thm:main}]\label{thm:intr}
\begin{enumerate}
\item 
The map $\bbT: \bbM\dashrightarrow\bbM^\dyn$ is a rational isomorphism between vector bundles; 
\item $\bbT(\tilde f_v)=\calE_v^\dyn $, and $\bbT(\calE_v)=\tilde f^\dyn _v$;
\item The inverse map of $\bbT$ is given by $\bbT^\dyn $. 
\end{enumerate}
\end{thm}
The periodic modules $\bbM$ and $\bbM^\dyn$ correspond to the torus equivariant elliptic cohomology of the Springer resolution $T^*G/B$ and its Langlands dual $T^*G^\vee/B^\vee$, respectively. These modules serve as representation-theoretical models for studying the respective cohomology.
Our main theorem establishes a canonical isomorphism between the equivariant elliptic cohomology of $T^*G/B$ and its Langlands dual $T^*G^\vee/B^\vee$. This isomorphism maps the fixed point basis of $T^*G/B$ to the (opposite) elliptic classes of $T^*G^\vee/B^\vee$ and vice versa.

The proof of the theorem relies on the Poincar\'e pairing and the adjunction properties of the elliptic DL operators, which play crucial roles.
In \S~\ref{sec:complete}, we utilize the same adjunction properties to reproduce a result previously established by Rim\'anyi and Weber \cite{RW22}. This result compares restriction formulas of (normalized) elliptic classes and their Langlands dual, and it is linked to the concept of 3d-mirror symmetry, as explained in the aforementioned reference.
The study of 3d-mirror symmetry (also called the symplectic duality) in elliptic cohomology has gained significant attention in recent years, with notable contributions from various researchers (e.g., \cite{BR23, FRV18,H20,KS22, KS23,RSVZ19,RSZV22, RW22,SZ22}). We refer interested readers to \cite[\S~1.3.1]{AO21}, which presents a related statement to that of \cite{RW22}. The main theorem of the present paper offers a distinct perspective on this topic. Another interesting spect of 3d-mirror symmetry is $q$-difference equation, which we hope to investigate in the future. In particular, we expect Theorem~\ref{thm:intr}(1) to be related to the pole cancelation property \cite[Theorem~5]{AO21}.

Elliptic classes are defined as rational sections of specific line bundles (direct sums) on an abelian variety. Consequently, the poles and zeros of an elliptic class are constrained by the properties of the line bundle.
This constraint enables us to analyze the properties of zeros and poles of the elliptic classes. In \S~\ref{sec:restriction}, we provide examples of calculations to illustrate these properties. Moreover, we define the logarithmic degree  of an element from $W$. It can be used to recover the information of zeroes and poles of the restriction coefficients. Meanwhile, Allen Knutson has communicated to the authors that he was able to employ pipe dreams to identify restriction coefficients of cohomology Schubert classes with double Schubert polynomials for the root system of type $A$. Given this result, and since the elliptic classes are the sole classes associated with Schubert varieties in elliptic cohomology, one may regard our restriction coefficients as ``elliptic double Schubert polynomials'' (in type $A$).

In this paper, we study elliptic classes in the elliptic Hecke algebra with dynamical parameters and its periodic module.
We show that a localization of the elliptic Hecke algebra and its periodic module is sufficient to captures the essential properties of the elliptic classes.
An integral structure of the elliptic Hecke algebra with dynamical parameters, that is, the algebra without localization, is interesting for representation theory purpose. 
We refer an interested reader to \cite{ZZ23} for a 
 definition of such an integral structure, and an application in Fourier-Mukai transform of representations.
 
In order to study the restriction formula of the elliptic classes, one can also derive a formula similar to the ones due to Billey~\cite{billey} (for cohomological Schubert classes, cf. also \cite{ajs}) and Graham-Willems~\cite{graham, willems} (for $K$-theory Schubert classes). Such a formula, known as a root polynomial formula, will offer a different perspective compared to the one presented in Theorem~\ref{thm:res}, and has certain advantages. While the formula in the mentioned theorem concerns the $a_{v,w}$-coefficients, the root polynomial formula will focus on the $b_{v,w}$-coefficients. Notably, the exploration of these $b$-coefficients served as the original motivation for the present paper. The root polynomial formula will be presented  in an upcoming paper~\cite{LZZ} soon.

\subsection*{Organization}

In \S~\ref{sec:poin} we introduce  basic concepts of root systems and their associated structures, together with the Poincar\'e line bundle, and in \S~\ref{sec:bbS} we define the elliptic twisted group algebra  $\bbS$, and the product of its rational sections. We define the elliptic DL operators and show that they are rational sections of $\bbS$. We also introduce the periodic module $\bbM$, and give precise statements of Theorem~\ref{thm:intr}(1),(3). In \S~\ref{sec:S_property} we collect basic properties of $\bbS$. In \S~\ref{sec:dual} we introduce the duality functors that are used to construct Poincar\'e pairings. We compare the construction for the Langlands dual system in \S~\ref{sec:Langlands} with the dynamical periodic module. In \S~\ref{sec:DL} we collect other useful versions of the elliptic DL operators. We define the elliptic classes in \S~\ref{sec:ell_class}, study their dual via the Poincar\'e pairings, and prove the main results. For the convenience of the readers, in \S~\ref{sec:complete},  we collect a list of all versions of elliptic classes.  In \S~\ref{sec:mirror}, we deduce a known result due to Rim\'anyi and Weber. In sections \S~\ref{sec:restriction} and \S~\ref{sec:examples}, we give a formula of the restriction coefficients, and define the logarithmic degree. 


\subsection*{Acknowledgements} The authors would like to thank Allen Knutson and Rich\'ard Rim\'anyi for helpful conversations. C.~L. is partially supported by the NSF grant DMS-1855592. 
G.~Z.~is partially supported by the Australian Research Council via grants DE190101222 and DP210103081. Part of this research was conducted when C.~Z.~ was visiting the University of Melbourne,  the Sydney Mathematical Research Institute, University of Ottawa, and the Max Planck Institute of Mathematics. He would like to thank these  institutes for their hospitality and support.

\section{The Poincar\'e line bundle}\label{sec:poin}
In this section we collect some basic facts about the Poincar\'e line bundle $\calP$, root datum, and define the $\rho$-shift of $\calP$, the latter of which serves as the polynomial representation of the DL operators with dynamical parameters.
\subsection{Coherent sheaves and line bundles}We clarify some notation about the Weyl group action on the bundles and their rational sections.
 If $\calF$ is a vector bundle over $X$ on which the isomorphism $w$ acts, and if $U$ is an open subset, then a section $f\in \calF(U)$ defines a section $f'\in (w^*\calF)(w^{-1}{U})$, and we will denote 
\[
f'={}^{w^{-1}}f. 
\]
This is compatible with compositions,   that is, given $v,w$ acting as isomorphisms, then a section $f\in \calF(U)$ defines a section   ${}^{v^{-1}}f\in (v^*\calF)(v^{-1}(U))$, and 
\[{}^{(vw)^{-1}}f={}^{w^{-1}}({}^{v^{-1}}f)\in (w^*(v^*\calF))(w^{-1}(v^{-1}(U)))= ((vw)^*\calF)((vw)^{-1}(U)).\]

 Now let $f:X\to Y$ be a homomorphism of algebraic varieties, then it induces a homomorphism $f^*:\Pic^0(Y)\to \Pic^0(X)$. It has the property that for any $\la\in \Pic^0(Y) $ that defines the line bundle $\calO(\la)$, $f^*\calO(\la)=\calO(f^*\la)$.
We state this property in family and functorially.
\begin{lemma}\label{lem:Pic_Poincare} Let $f:X\to Y$ be a morphism of algebraic varieties, which induces the following maps:
\[
\xymatrix{
Y\times \Pic^0(Y) & X\times \Pic^0(Y)\ar[l]_{f\times 1}\ar[r]^{1\times f^*} & X\times \Pic^0(X) .}
\] Let $\calP_X$ and $\calP_Y$ be the tautological line bundles on $X\times\Pic^0(X)$ and $Y\times\Pic^0(Y)$ respectively \cite[Exercise 9.4.3]{FGA}. Then, there is a canonical isomorphism 
\[
(1\times f^*)\calP_X\cong(f\times 1)^*\calP_Y.
\]
which is furthermore compatible with compositions of morphisms of algebraic varieties.
\end{lemma}
\begin{proof}
This follows from the definition of the Picard functor \cite[\S~9]{FGA}. 
\end{proof}

\subsection{Elliptic curve}
Let $E$ be an elliptic curve over $\bbC$. We have isomorphisms 
\[
E\cong \bbC/(\bbZ+\bbZ\tau)\cong \bbC^*/q^\bbZ, 
\]
where $\tau\in \bbC$ satisfies $\im \tau>0$, and $q=e^{2\pi i\tau}$. Recall the Jacobi-Theta function
\[\vartheta(u)=\frac{1}{2\pi i}(u^{1/2}-u^{-1/2})\prod_{s>0}(1-q^su)(1-q^su^{-1})\cdot \prod_{s>0}(1-q^s)^{-2}, \quad u\in \bbC^*.\]
Assume $|q|<1$, then this series converges and defines a holomorphic function on a double cover of  $\bbC^*$. 
The line bundle  $\calO(\{0\})$ becomes trivial when lifted to the cover $\bbC^*$. Denote the lifting by $\tilde{L}\to \bbC^*$, which is endowed with a lift of the section $\tilde{s}:\bbC^*\to \tilde{L}$.
The trivialization 
\[\xymatrix{
\tilde{L}\ar@/^/[dr]\ar[rr]^{\cong}_{\phi}&&\bbC^*\times \bbC\ar[dl]\\
&\bbC^*\ar@/^/[ul]^{\tilde{s}}&
}\]
 can be uniquely fixed by the property that $\phi$ commutes with multiplication by $q^\bbZ$ and that the derivative of $\phi\circ\tilde{s}$ is 1 at $1\in \bbC^*$ \cite[p.38]{Sie}.
The above formula for the theta function are written so that it is a function on $\bbC^*$, that vanishes of order 1 at $q^\bbZ$, non-zero everywhere else, and the derivative at $1\in \bbC^*$ is equal to 1. In what follows, we denote $\theta(x)=\vartheta(e^{2\pi i x})$.

We identify $\Pic^0(E)$ with $E$ itself, via the map $E\to \Pic^0(E)$, $\lambda\mapsto \{\lambda\}-\{0\}$.
In particular, on $E\times E$, there is a universal line bundle $\calP_E$, called  the Poincar\'e line bundle. 
Writing the coordinate of $E\times E$ as $(z,\lambda)$, the Poincar\'e line bundle is characterized 
by the property that 
for any $\lambda\in E$, we have $\calP_E|_{E\times \lambda}=\calO( \{\lambda\}-\{0\})$, and  $\calP_E|_{0\times E}=\calO$.
Equivalently, with $\theta$ introduced above,  a rational section of $\calP_E$ is \[\frac{\theta(z-\lambda)}{\theta(z)}.\]
 
 The following standard fact is used frequently in determining rational sections of line bundles over the elliptic curve.
\begin{prop}\label{prop:E_zero_pole}
Let $x\in E$ which determines the degree zero line bundle $\calO(x)$. If $f$ is a rational section of  $\calO(x)$ on $E$ with zeros $\{p_i, i=1,...,n\}$ and poles $\{q_j, j=1,...,m\}$ (counting with multiplicities), then 
\[
m=n, \text{ and } \sum p_i-\sum q_j=x.
\]
\end{prop}

\subsection{Root system}
\label{subsec:rootsys}
 We recall the notion of root datum following \cite[Exp. XXI, \S~1.1]{SGA3}. A root datum of rank $n$ consists of two lattices $X^*$ and $X_*$ of rank $n$ with a perfect $\bbZ$-valued pairing $\lag, \rag$,  non-empty finite subsets of roots and coroots $\Phi\subset X^*$, and $\Phi^\vee\subset X_*$ with a bijection $\alpha\mapsto \alpha^\vee$.
If the root datum  has a maximal torus $T$, the lattices above are respectively $\bbX^*(T)$ and $\bbX_*(T)$, the group of characters and co-characters of $T$. 
In the present paper, we fix a set $\Sigma$ of simple roots $\{\alpha_1,\dots,\alpha_n\}$, with simple reflections $s_i:=s_{\al_i}$.  The Weyl group is denoted by $W$. For each $w\in W$, define the inversion set $\Phi(w)=w\Phi_-\cap \Phi_+$,  and the length as $\ell(w)$, with the sign $\ep_w=(-1)^{\ell(w)}$.  
 Let $\rho$, resp. $\rho^\vee$ be the half-sum of all positive roots of $\Phi$ and $\Phi^\vee$.

 Define $\catA=X_*\otimes E$, and $\catA^\vee:=X^*\otimes E$.
The Weyl group $W$ acts on both $\catA$ and $\catA^\vee$, and hence on the product $\catA\times\catA^\vee$, we have the action of product of two Weyl groups. To avoid confusion, we write the factor of $W$ acting on $\catA$ as $W$ and the factor acting on $\catA^\vee$ as $W^\dyn$. Similarly, for any element $w\in W$, we write the element in $W^\dyn$ as $w^\dyn$.
For clarity, we call them respectively the {\it spectral Weyl group action} and the {\it dynamical Weyl group action}. Since they act on different factors of the product, they obviously commute.

The isomorphism $E\cong \Pic^0(E)$ induces an isomorphism 
 \[\catA^\vee\cong \Pic^0(\catA), \]
 where $\Pic^0(\catA)$ is the dual abelian variety of $\catA$. 
In particular, the notation $\catA^\vee$  has no ambiguity. 
We now make this isomorphism explicit. 
We consider a special point $\mu \otimes t\in \catA^\vee=X^*\otimes E$, with $\mu \in  X^*$ and $t\in E$. 
The character $\mu $ defines a map $\chi_\mu :\catA\to E$, and $t\in E=E^\vee$ defines a degree 0 line bundle $\calO(t)$ on $E$. 
Under the above isomorphism, the line bundle on $\catA$ corresponding to $\mu \otimes t$ is $\chi_\mu ^*\calO(t)$. 
In what follows, we write $\mu \otimes t$ simply as $t\mu $, and $\calO(t\mu)=\chi^*_\mu\calO(t)$. Also denote $z_\mu=\chi_\mu(z)$ for $z\in \catA$. 
   Dually, each $\mu^\vee\in X_*$ defines a map $\chi_{\mu^\vee}:\catA^\vee\to E$, and denote $\la_{\mu^\vee}=\chi_{\mu^\vee}(\la), \la\in \catA^\vee$. 
 For $t\in E$, we have similarly $\chi_{\mu ^\vee}^*\calO(t)=\calO(t\mu ^\vee)$.


\subsection{The $\rho$-shift of the Poincar\'e line bundle}
In this section, and what follows we fix an $\hbar\in E$.

On $\catA\times\catA^\vee$ there is a  line bundle, 
the Poincar\'e line bundle, also denoted by $\calP$ \cite[\S~9]{P03}. It satisfies the universal property, that is, for any variety $X$, if  there is any line bundle $\calL$ over $\catA\times S $ for a variety $S$, then there is a unique map $g:\catA\times S\to \catA\times \catA^\vee$ such that $\calL=g^*\calP$.
In particular, it 
satisfies the property that $\calP|_{\catA\times \la}\cong \calO(\la)$ for any $\la\in \catA^\vee$.
Such a line bundle is unique if we impose a normalization condition
$\calP|_{a\times \catA^\vee}\cong \calO_{\catA^\vee}$ for a given point $a\in \catA$.
In what follows, we take $a=\hbar\rho^\vee$ with $\hbar\in E$. 
Here the uniqueness follows from the See-Saw Lemma, which asserts that two line bundles on $\catA\times\catA^\vee$ are isomorphic if the restrictions on  $\catA\times\{\lambda\}$ are isomorphic for all $\lambda\in \catA^\vee$, and on $\{z\}\times\catA^\vee$ for some $z\in \catA$. Indeed, this property will be used below frequently   in showing two line bundles over $\catA\times \catA^\vee$ are isomorphic. 
We note that the See-Saw Lemma only implies the existence of an isomorphism but does not necessarily provide a canonical one. 

Define the dotted action of $W$ on $\catA$ by \[
w_\bullet z=w(z-\hbar\rho^\vee)+\hbar\rho^\vee.
\]
By definition, the map $w_\bullet:\catA\to \catA$ induces a map $w^*:\Pic^0(\catA)\to \Pic^0(\catA)$, which is denoted by $(w^{-1})^\dyn$. That is, $\calO((w^{-1})^\dyn(\la))=w^*\calO(\la)$ as line bundles on $\catA$. Then  Lemma~\ref{lem:Pic_Poincare} gives an isomorphism 
\begin{equation}\label{eqn:Pincare_dot}
(w_\bullet\times 1)^*\calP\cong(1\times (w^{-1})^\dyn)^*\calP,
\end{equation}
which is compatible with compositions of group elements.

\label{subsec:rho-shift}

\begin{definition}\label{def:bbL}Define $\bbL=\vor^*\calP$ on $\catA\times \catA^\vee$, where  $\vor:  \catA^\vee\to  \catA^\vee, ~ \la\mapsto \la+\rho \hbar.$
\end{definition}
The line bundle $\bbL$  will be the polynomial representation of the dynamical elliptic Hecke algebra \cite{ZZ23}. Also see Lemma \ref{lem:poly} below.
  
\section{The elliptic twisted group algebra } \label{sec:bbS}

In this section we  define the elliptic twisted group algebra $\bbS$ and the periodic modules $\bbM$, together with some of their variant forms. We also explain the algebra structure of $\bbS$ and its action on $\bbM$ in terms of rational sections. Note that  rational sections of vector bundles in this paper mean sections on open subsets on which they are defined.

\subsection{The elliptic twisted group algebra}
We define the twisted algebra, which is a vector bundle induced by the Weyl group actions on the line bundle $\bbL$. 
\begin{definition}\label{def:bbS}
We define, for each $w,v\in W$
\[
\bbS_{w,v}=\bbL\otimes (w^{-1})^*(v^{-1})^{\dyn*}\bbL^{-1}. 
\]
Let $\bbS=\bigoplus_{w,v\in W}\bbS_{w,v}$. 
\end{definition}
The goal of the present section is to explain the algebra structure on $\bbS$. In order to do so, we introduce an ambient monoidal category, in which $\bbS$ is an algebra object.

\subsection{A tensor structure} 
We consider the following category $\Coh^{W\times W^\dyn}(\catA\times \catA)$, that is, the  $W\times W^\dyn$-graded category of coherent sheaves on $\catA\times \catA^\vee$. Objects of  this category can be written as $\calF=\bigoplus_{w\in W,v\in W^\dyn}\calF_{w,v}$ with  $\calF_{w,v}$ at degree $(w,v)$.  We then define the tensor product $\star$ by 
\[
\calF_{w,v}\tensor \calF_{w',v'}':=\calF_{w,v}\otimes(w^{-1})^*(v^{-1})^{\dyn*}\calF'_{w',v'},
\]
which is set to belong to degree $(ww',vv')$. Clearly $\calO$ is the identity object. It is trivial to check that this defines a tensor monoidal category structure with unit $\calO$ concentrated at degree $(e,e)$.

\begin{lemma} \label{lem:tensor}
\begin{enumerate}
 \item \label{lem:associat} We have the following  composition properties:
\begin{align*}
\bbS_{w_1w_2, v_1v_2}&=\bbS_{w_1,v_1}\otimes (w_1^{-1})^*(v_1^{-1})^{\dyn*}\bbS_{w_2,v_2},&\quad \bbS_{w,v}^{-1}&=(w^{-1})^*(v^{-1})^{\dyn*}\bbS_{w^{-1}, v^{-1}}.
\end{align*}
\item The objects \[\bbS_{W\times W^\dyn}:=\bigoplus_{w,v\in W}\bbS_{w,v}\hbox{  and }~\bbS_{W\times W^\dyn}^{-1}:=\bigoplus_{w,v\in W}\bbS_{w,v}^{-1}\] are  monoidal objects, where the degrees of $\bbS_{w,v}$ and $\bbS_{w,v}^{-1}$ are both $(w,v)$. For simplicity, we write $\bbS$ and $ \bbS^{-1}$ respectively.
\end{enumerate}
\end{lemma}
\begin{proof}
(1). We have
\begin{align*}
\bbS_{w_1w_2,v_1v_2}&=\bbL\otimes ((w_1w_2)^{-1})^*((v_1v_2)^{-1})^{\dyn*}\bbL^{-1}\\
&=\bbL\otimes (w_1^{-1})^*(w_2^{-1})^*(v_1^{-1})(v_2^{-1})^{\dyn*}\bbL^{-1}\\
&=\bbL\otimes (w_1^{-1})^*(v_1^{-1})^{\dyn*}\left((w_2^{-1})^*(v_2^{-1})^{\dyn*}\bbL^{-1}\right)\\
&=\bbL\otimes (w_1^{-1})^*(v_1^{-1})^{\dyn*}(\bbS_{w_2, v_2}\otimes \bbL^{-1})\\
&=\bbL\otimes (w_1^{-1})^*(v_1^{-1})^{\dyn*}\bbS_{w_2,v_2}\otimes (w_1^{-1})^*(v_1^{-1})^{\dyn*}\bbL^{-1}\\
&=\bbS_{w_1,v_1}\otimes (w_1^{-1})^*(v_1^{-1})^{\dyn*}\bbS_{w_2, v_2}. 
\end{align*}
This proves the first identity. Considering $w_1w_2=e=v_1v_2$ in the first identity, we obtain the second identity.

(2). 
From part (1),  we  have an isomorphism
\begin{align*}\bbS_{w,v}\tensor\bbS_{x,y}&=\bbS_{w,v}\otimes (w^{-1})^*(v^{-1})^{\dyn*}\bbS_{x,y}\overset{\sim}\longrightarrow\bbS_{wx, vy},
\end{align*}
it is then easy to see that this product is associative. The statement for $\bbS^{-1}_{W\times W^\dyn}$ can be proved similarly. 
\end{proof}
The sheaf $\bbS$, together with its variant forms $\bbS', \bbS'', \bbS(-\la), \bbS(-\la)', \bbS(-\la)'', \bbS(-z), \bbS(-z)', \bbS(-z)''$ defined below,  are called the elliptic twisted group algebras.

\begin{lemma}\label{lem:poly}
The object $\bbL\in \Coh(\catA\times \catA^\vee)$ is a left module over the monoidal object $\bbS$.
\end{lemma}
\begin{proof}The actions are defined as follows:
\begin{align*}
\bbS_{w,v}\tensor \bbL&= \bbS_{w,v}\otimes (w^{-1})^*(v^{-1})^{\dyn*}\bbL\overset{\sim}\longrightarrow\bbL,
\end{align*}
it is then straightforward to show that the associativity holds.
\end{proof}

\subsection{Equivalent tensor categories}
We define another tensor product on the category \mbox{$\Coh^{W\times W^\dyn}(\catA\times \catA^\vee)$}: for objects $\calF_{w_1,v_1}$ and $\calG_{w_1,v_2}$ with degrees $(w_1,v_1)$ and $(w_2,v_2)$, respectively, define
\[
\calF_{w_1, v_1}\stone \calG_{w_2,v_2}=w_2^*\calF_{w_1,v_1}\otimes (v_1^{-1})^{\dyn*}\calG_{w_2,v_2},
\]
whose degree is $(w_1w_2, v_1v_2)$.

Similar as the second tensor, we can define the third  tensor product structure on $\Coh^{W\times W^\dyn}(\catA\times \catA^\vee)$ as 
\[
\calF_{w_1,v_1}\sttwo \calG_{w_2,v_2}=v_2^{\dyn*}\calF_{w_1,v_1}\otimes (w_1^{-1})^*\calG_{w_2,v_2}.
\]

\begin{lemma}\label{lem:monoidal_equivalence}
We have the following commutative diagram of equivalences between these different tensor monoidal categories:
\[
\xymatrix{&\Coh^{W\times W^\dyn}(\catA\times \catA^\vee, \tensor)\ar[dl]_-{\calF_{w,v}\mapsto w^*\calF_{w,v}}\ar[dr]^-{\calF_{w,v}\mapsto v^{\dyn*}\calF_{w,v}}&\\
\Coh^{W\times W^\dyn}(\catA\times \catA^\vee, \stone)\ar[rr]^-{\calF_{w,v}\mapsto (w^{-1})^*v^{\dyn*}\calF_{w,v}}&&\Coh^{W\times W^\dyn}(\catA\times \catA^\vee, \sttwo).}
\]
\end{lemma}
\begin{proof}
The proof is straightforward by using the definitions. \end{proof}

Define 
\[\bbS_{w,v}'=w^*\bbS_{w,v}, \quad \bbS'=\bbS'_{W\times W^\dyn}=\bigoplus_{w,v}\bbS'_{w,v},
\]
which by Lemma~\ref{lem:monoidal_equivalence} is a monoidal object in the category $(\Coh^{W\times W^\dyn}(\catA\times \catA^\vee), \stone)$. 
The algebra structure is obtained from  the following isomorphism
\[
\bbS'_{w_1,v_1}\stone \bbS'_{w_2,v_2}=w_2^*\bbS'_{w_1,v_1}\otimes  (v_1^{-1})^{\dyn*}\bbS'_{w_2,v_2}\overset{\sim}\longrightarrow\bbS'_{w_1w_2,v_1v_2}.\]

Similarly,  the object 
\[\bbS''=\bbS''_{W\times W^\dyn}=\bigoplus_{w,v}\bbS''_{w,v}, \quad \text{with }\quad 
\bbS''_{w,v}:=v^{\dyn*}\bbS_{w,v}
\]
defines a monoidal object in $(\Coh^{W\times W^\dyn}(\catA\times \catA^\vee), \sttwo)$.

\subsection{The periodic modules}\label{subsec:act_periodic}

The category $\Coh^W(\catA\times \catA^\vee)$ is a module over the monoidal category $(\Coh^{W\times W^\dyn}(\catA\times \catA^\vee), \star)$, where the action is defined as\[\calF_{w,v}\star\calG_x:=(wx)^*\calF_{w,v}\otimes (v^{-1})^{\dyn*}\calG_x\] sitting in degree $wx$.
The associativity is straightforward to check. 
Similarly,  $\Coh^{W^\dyn}(\catA\times \catA^\vee)$ is also a   module over the monoidal category $(\Coh^{W\times W^\dyn}(\catA\times \catA^\vee), \star)$.

Via the equivalences in Lemma \ref{lem:monoidal_equivalence}, the monoidal categories  $(\Coh^{W\times W^\dyn}(\catA\times \catA^\vee), \stone)$ and $(\Coh^{W\times W^\dyn}(\catA\times \catA^\vee), \sttwo)$ have module categories $\Coh^W(\catA\times \catA^\vee)$ and $\Coh^{W^\dyn}(\catA\times \catA^\vee)$ where we use the  notations $\stone$ and $\sttwo$  for the actions. 
We leave the details to interested readers. 

\begin{definition}
Define the periodic module in $\Coh^W(\catA\times \catA^\vee)$ as follows:
\[
\bbM=\bigoplus _{x\in W}\bbM_x, \quad \bbM_x=x^*\bbL.
\]Denote, for each $w\in W$, $\fp_w:\bbM\to \bbM_w$ the projection and $\fri_w:\bbM_w\to \bbM$ the embedding of each summand. 
\end{definition}
The action of $\bbS$ in $(\Coh^{W\times W^\dyn}(\catA\times \catA^\vee), \star)$  on the object $\bbM$ in $\Coh^W(\catA\times \catA^\vee)$ is denoted by 
\[\bullet:\bbS\star\bbM\to\bbM,\]
via, for each $w,v,x\in W$, the canonical isomorphism
\begin{equation}\label{eqn:action}
\bullet: \bbS_{w,v}\star x^*\bbL= (wx)^{*}\bbS_{w,v}\otimes (v^{-1})^{\dyn*}x^*\bbL=(wx)^{*}\bbL\otimes (wx)^{*}(w^{-1})^*(v^{-1})^{\dyn*}\bbL^{-1}\otimes (v^{-1})^{\dyn*}x^*\bbL=(wx)^{*}\bbL.\end{equation}

Similarly, \[\bbM^\dyn:=\bigoplus_y\bbM_y^\dyn, \quad  \bbM_y^\dyn=y^{\dyn*}\bbL\]
is an object in $\Coh^{W^\dyn}(\catA\times \catA^\vee)$.
We denote 
  $\fp_w^\dyn:\bbM^\dyn\to \bbM_w^\dyn$, and $\fri_w^\dyn:\bbM_w^\dyn \to \bbM$. 
The action of $\bbS$ in $(\Coh^{W\times W^\dyn}(\catA\times \catA^\vee), \star)$ on $\bbM^\dyn$ in $\Coh^{ W^\dyn}(\catA\times \catA^\vee)$ is denoted by  \[\bullet^\dyn:\bbS\star\bbM^\dyn\to\bbM^\dyn.\]

\begin{remark}\label{rmk:ell_coh}
The periodic module of the usual affine Hecke algebra, besides its applications in combinatorics and representation theory, also has a geometric realization as the $K$-theory of the Springer resolutions $T^*G/B$  \cite{L}. 
Therefore, the  periodic modules $\bbM$ and $\bbM^\dyn$ in the present paper should be considered as the equivariant  elliptic cohomology of the  Springer resolutions $T^*G/B$ and $T^*G^\vee/B^\vee$ respectively.
A version of such construction, in the absence of the dynamical parameters, can be found in \cite{ZZ22}. 
We notice another construction by Hikita \cite{H20} in elliptic cohomology related to Lusztig's original \cite{L}. We leave  its relation with the present paper to  future investigations. 
\end{remark}

\subsection{Variant forms of the actions} From the equivalences of monoidal categories in Lemma \ref{lem:monoidal_equivalence}, the algebra object $\bbS'$ then act on module objects $\bbM$ and $\bbM^\dyn$.
For example, the action 
\begin{equation}\label{eq:periodicaction}
\bullet: \bbS'_{w,v}\stone \bbM_x\overset\sim\longrightarrow \bbM_{wx}
\end{equation}
is defined by 
\[
\bbS'_{w,v}\stone \bbM_x:=x^*\bbS'_{w,v}\otimes (v^{-1})^{\dyn*}\bbM_x\overset\sim\to x^*\bbS'_{w,v}\otimes (v^{-1})^{\dyn*}\bbS'_{x,e}\otimes (v^{-1})^{\dyn*}\bbL\overset\sim\to \bbS'_{wx,v}\otimes (v^{-1})^{\dyn*}\bbL\overset\sim \to\bbM_{wx}.
\]

Similarly, equivalences in Lemma \ref{lem:monoidal_equivalence} also induces an action of $\bbS''$ on $\bbM^\dyn$, denoted by 
\begin{equation}\label{eq:periodicactiondyn}
\bullet^\dyn:\bbS''_{w,v}\sttwo \bbM_y^\dyn\to \bbM^\dyn_{vy}.
\end{equation}
More precisely, we have
\[
\bbS''_{w,v}\sttwo\bbM''_{y}:=y^{\dyn*}\bbS''_{w,v}\otimes (w^{-1})^*\bbM''_y\overset\sim\to y^{\dyn*}\bbS''_{w,v}\otimes (w^{-1})^*\bbS''_{e,y}\otimes (w^{-1})^*\bbL\overset\sim\to \bbS''_{w,vy}\otimes (w^{-1})^*\bbL\overset\sim\to \bbM_{vy}^\dyn.
\]

For the convenience of the readers, we summarize the constructions in the following proposition.
\begin{prop}
\begin{enumerate}
\item The category $\Coh^{W\times W^\dyn}(\catA\times \catA)$ with $\star$ is a monoidal category, in which $\bbS$ with the morphism $\bbS\star\bbS\to \bbS$ in Lemma~\ref{lem:associat} is an algebra (monoidal) object. 
\item The category $\Coh^W(\catA\times \catA^\vee)$  with  $\star$ is a module category of $\Coh^{W\times W^\dyn}(\catA\times \catA)$, in which the object $\bbM$ with the morphism $\bbS\star\bbM\to \bbM$ in \eqref{eqn:action} is a  module of $\bbS$.
\item 
The category $\Coh^{W^\dyn}(\catA\times \catA^\vee)$  with  $\star$ is a module category of $\Coh^{W\times W^\dyn}(\catA\times \catA)$, in which the object $\bbM^\dyn$ with the corresponding action is is a  module of $\bbS$.
\item Similar results to (1-3) hold for the category $(\Coh^{W\times W^\dyn}(\catA\times \catA),\star')$ which containing the object $\bbS'$, the category $(\Coh^{W\times W^\dyn}(\catA\times \catA),\star'')$ which containing the object $\bbS''$.
\item There is an equivalence of monoidal categories (Lemma~\ref{lem:monoidal_equivalence}) intertwining $\star$ and $\star'$ (resp. $\star''$), sending  the object $\bbS$ to $\bbS'$ (resp. $\bbS''$).
\end{enumerate}
\end{prop}

\subsection{Twisted product of rational sections}
\label{subsec:Sproductrational}
We will perform calculations of rational sections in $\bbS$ and their variant forms with $'$ and $''$. 
To keep track of the degrees in $W\times W^\dyn$, a rational section $a$ of $\bbS_{w,v}$ is denoted by $a\de_w\de_v^\dyn$. 
The functor sending a coherent sheaf to the vector space of  rational sections is a Lax monoidal functor with respect to the tensor product $\tensor$ in Lemma \ref{lem:tensor}. 
Indeed, one has a stronger version, which we recall for illustration purpose although not used in the present paper.
Let $\pi^2_*: \catA\times\catA^\vee\to \catA/W\times\catA^\vee/W^\dyn$ be the projection. Then, the functor $\pi^2_*$ is a Lax monoidal functor with the target endowed with the usual tensor structure of coherent sheaves \cite[Lemma~2.3]{ZZ23}.
In particular, we have for example \[\pi^2_*\bbS'_{w,v}\otimes \pi^2_*\bbS'_{x,y}\to \pi^2_*\bbS'_{wx,vy},\] and similarly 
\[\pi^2_*\bbS'_{w,v}\to  \bigoplus_x\sHom(\pi^2_*\bbM_x, \pi^2_*\bbM_{wx}).\]

The space of rational sections of $\bbS_{w,v}$  has an algebra structure, which
 can be described using Kostant-Kumar's twisted product:
if  $a\de_{w_1}\de^\dyn_{ v_1}$,  and $b\de_{w_2}\de_{v_2}^\dyn$ are rational sections of  $\bbS_{w_1,v_1}$ and $\bbS_{w_2,v_2}$, respectively, then the tensor product  $\tensor$ gives
\begin{equation}\label{eq:prod}
a\de_{w_1}\de^\dyn_{ v_1}b\de_{w_2}\de_{v_2}^\dyn=a\cdot {}^{w_1v_1^\dyn}b\de_{w_1w_2}\de_{v_1v_2}^\dyn,
\end{equation}
which is a rational section   of $\bbS_{w_1w_2,v_1v_2}$ via $\bbS_{w_1,v_2}\tensor\bbS_{w_2,v_2}\overset\sim\longrightarrow \bbS_{w_1w_2,v_1v_2}$.  

If $a\de_{w}\de_v^\dyn$ is a rational section of $\bbS_{w,v}$, then $\de_{w}\cdot {}^{w^{-1}}a\cdot\de_{v}^\dyn$ can be viewed as a rational section of $w^*\bbS_{w, v}=\bbS'_{w,v}$ (again, $\de_w,\de_v^\dyn$ can be thought of as the bi-degree of this rational section). Then in terms of rational sections, the multiplication of $\bbS'_{w,v}$ under $\stone$ can be written as
 \[
\de_{w_1}{}^{w_1^{-1}}a\de_{v_1}^\dyn\de_{w_2}{}^{w_2^{-1}}b\de_{v_2}^\dyn=\de_{w_1w_2}{}^{w_2^{-1}w_1^{-1}}a\cdot {}^{v_1^\dyn w_2^{-1}}b\de_{v_1v_2}^{\dyn}.
\]
This agrees with \eqref{eq:prod}. 

The same remark holds for $\bbS''$,  that is, a rational section $a\de_{w}\de_{v}^\dyn$ of $\bbS_{w,v}$ defines a rational section $\de_v^\dyn \cdot{}^{(v^{-1})^\dyn}a\cdot\de_w$ of $\bbS''_{w,v}$, and \[
\de_{v_1}^\dyn {}^{(v_1^{-1})^\dyn}a\de_{w_1}\de_{v_2}^\dyn {}^{(v_2^{-1})^\dyn}b\de_{w_2}=\de_{v_1v_2}^\dyn  {}^{(v_2^{-1}v_1^{-1})^\dyn} a \cdot {}^{w_1(v_2^{-1})^\dyn}b\de_{w_1w_2},
\]
which again coincides with the result from \eqref{eq:prod}. 

\label{subsec:rat_sec}
In the same spirit, we spell out the action of $\bbS'$ on $\bbM$ from \eqref{eq:periodicaction} on the space of rational sections. Rational sections of $\bbM$ are finite sums of $a f_x$ where $a$ is a rational section of $\bbM_x=x^*\bbL$. Again, $f_x$ is here to keep track of the degree of $a$.  Then the action  can be written as 
\begin{equation}
\de_wb\de_v^\dyn \bullet af_x={}^{v^\dyn }a\cdot {}^{x^{-1}}bf_{wx}. 
\end{equation}
One can check that this  provides the same formula if we start from the variant forms of $\bbS'$, that is, we can also start from $\bbS$ or $\bbS''$. For example, if $\de_v^\dyn b\de_w$ is a rational section of $\bbS''_{w,v}$, the $\bullet$ action would give
\begin{equation}\label{eq:bullet}
\de_v^\dyn b\de_w\bullet af_x={}^{v^\dyn (wx)^{-1}}b\cdot {}^{v^\dyn}a f_{wx}. 
\end{equation}

Similarly, from \eqref{eq:periodicactiondyn} the action of of $\bbS''$ on $\bbM^\dyn$ can be described in the same way. Rational sections of $\bbM^\dyn_y$ are written as finite sums of  $a f_y^\dyn$, and given a rational section $\de_v^\dyn b\de_w$ of $\bbS''$, the action is written as 
\begin{equation}\label{eq:bulletdyn}
\de_v^\dyn b\de_w\bullet^\dyn af_y^\dyn=    {}^{(y^{-1})^\dyn}b\cdot{}^wa f^\dyn_{{vy}}. 
\end{equation}
If one takes a rational section $\de_w b\de_v^{\dyn}$ of $\bbS'$, the action would be written as 
\begin{equation}
\de_w b\de_v^\dyn\bullet^\dyn af_y^\dyn={}^w a\cdot {}^{((vy)^{-1})^\dyn w}bf_{vy}^\dyn. 
\end{equation}

\subsection{The  DL operators}\label{subsec:DL1}

Let $z\in \catA, \la\in \catA^\vee$. For any simple root $\al$, define the  Demazure-Lusztig (DL) operators with dynamical parameters by 
\begin{align}
\label{eq:Tzla}  T_{\al}^{z,\la}&=\de_\al^\dyn\frac{\theta(\hbar)\theta(z_\al+\la_{\al^\vee})}{\theta(z_\al)\theta(\hbar+\la_{\al^\vee})}+\de_\al^{\dyn}\frac{\theta(\hbar-z_\al)\theta(-\la_{\al^\vee})}{\theta(z_\al)\theta(\hbar+\la_{\al^\vee})}\de_\al , \\
\label{eq:Tzladyn} 
T_\al^{z,\la, \dyn}&=\de_\al\frac{\theta(\hbar)\theta(\la_{\al^\vee}+z_\al)}{\theta(\la_{\al^\vee})\theta(\hbar-z_{\al})}+\de_\al\frac{\theta(\hbar+\la_{\al^\vee})\theta(-z_{\al})}{\theta(\la_{\al^\vee})\theta(\hbar-z_{\al})}\de_\al^{\dyn}.
\end{align}

\begin{theorem}\label{thm:Trational}
The operators $T^{z,\la}_\al$  and $T^{z,\la,\dyn}_\al$  are  rational sections of $\bbS''_{}$ and $\bbS'_{}$, respectively. 
\end{theorem}

\begin{proof}  We only prove the conclusion for $T^{z,\la}_\al$ as an example. By Lemma \ref{lem:monoidal_equivalence} and \S\ref{subsec:Sproductrational}, it suffices to show that its variant form (move $\de_\al^\dyn$ to the right of the coefficients)
\begin{equation}\label{eq:T-in-bbS}
\frac{\theta(\hbar)\theta(z_\al-\la_{\al^\vee})}{\theta(z_\al)\theta(\hbar-\la_{\al^\vee})}\de_\al^{\dyn}+\frac{\theta(\hbar-z_\al)\theta(\la_{\al^\vee})}{\theta(z_\al)\theta(\hbar-\la_{\al^\vee})}\de_\al \de_\al^{\dyn}\end{equation}
 is a rational section of $\bbS$. In other words,  we show that the coefficient of $\de_\al^\dyn$ is a rational section of $\bbS_{e,s_\al}$ and that of $\de_\al\de_\al^\dyn$ is a rational section of $\bbS_{\al,\al}$.  Using Lemma \ref{lem:basicS}.\eqref{item:basicS1}  below,  
\begin{align*}
\bbS_{e,s_\al}|_z&=\calO(z_\al \al^\vee-\hbar\al^\vee), & \bbS_{e,s_\al}|_\la&=\calO(\la_{\al^\vee}\al).\\
\bbS_{s_\al,s_\al}|_z&=\calO(-\hbar\al^\vee), &\bbS_{s_\al,s_\al}|_\la&=\calO(\hbar\al). 
\end{align*}

Let us consider the fraction $\frac{\theta(\hbar)\theta(z_\al-\la_{\al^\vee})}{\theta(z_\al)\theta(\hbar-\la_{\al^\vee})}$ in \eqref{eq:T-in-bbS}. When restricting $\bbS_{e,s_\al}$ to $z$, one obtains a line bundle over $\catA^\vee$ and  $\la$ becomes the only variable. Since $\la_{\al^\vee}$ has a zero at $z_\al$ and a pole at $\hbar$, so by Proposition \ref{prop:E_zero_pole} it is a rational section of $\calO(z_\al-\hbar)$ on $E$.  Pulling back to $\catA^\vee$ via the map $\chi_{\al^\vee}$, we see that the fraction is a rational section of $\bbS_{e,s_\al}|_z=\calO((z_\al-\hbar)\al^\vee)$. Now restrict to $\la$, then $z\in \catA$ is the only variable, and this fraction has a zero at $\la_{\al^\vee}$ and a pole at $0$ on $E$. Pulling back to $\catA$, we see that it is a rational section of $\bbS_{e,s_\al}|_\la$. By the See-Saw Lemma, we see that the fraction is a rational section of $\bbS_{e,s_\al}$. 

For simplicity, we call the above analysis the `zero-poles' property. 
For the second fraction $\frac{\theta(\hbar-z_\al)\theta(\la_{\al^\vee})}{\theta(z_\al)\theta(\hbar-\la_{\al^\vee})}$, we can see that the `zeros-poles' for $\la$ is $-\hbar\al^\vee$, and the `zeros-poles' for $z$ is $\hbar\al$, so it is a rational section of $\bbS_{s_\al,s_\al}$. 
\end{proof}

\begin{remark} In
Lemma \ref{lem:i*T} below, we explain that \eqref{eq:Tzladyn} is identified with the elliptic DL operator for the Langlands dual system. More properties about  these  operators  are proved in \S~ \ref{sec:DL} below.
\end{remark}

\subsection{Main theorem} \label{subsec:map}
It is proved in \cite{RW20} that both operators for any $\alpha\in\Sigma$ are self-inverse, and that the braid relations are satisfied (see also \cite[Proposition~4.11]{ZZ22} for the present form). 
Given a reduced sequence of $v\in W$, by using the twisted product in \S~\ref{subsec:Sproductrational},  
 one define $T_v^{z,\la}$ and $T_v^{z,\la,\dyn}$, which are rational sections of $\bbS''$ and $\bbS'$, respectively. 

\begin{definition}\label{def:bbT}
We define the Demazure-Lusztig (DL) coefficients $a_{v,w}^{z,\la}$ (resp. $a_{v,w}^{z,\la,\dyn}$) as  rational sections of $\bbS_{w,v}''$  (resp. $\bbS'_{w,v}$) via the following formula
\begin{align}\label{eq:Tlinearcomb}
T_v^{z,\la}&=\sum_{w\le v}\de_v^\dyn a_{v,w}^{z,\la}\de_w, & T_v^{z,\la,\dyn}&=\sum_{w\le v}\de_v a_{v,w}^{z,\la,\dyn}\de_w^\dyn.
\end{align}
Define the rational maps 
\begin{align*}
\bbT^{z, \la}=\sum_vT_v^{z,\la}=\sum_{v,w} \fp^\dyn_v a^{z,\la}_{v,w}\fri_{w^{-1}}: \bbM=\bigoplus_{w}\bbM_{w^{-1}}\dashrightarrow \bigoplus_{v}\bbM_{v}^\dyn=\bbM^\dyn.
\end{align*}
\begin{align*}
\bbT^{ z,  \la, \dyn}=\sum_vT_v^{z,\la,\dyn}&=\sum_{v,w}\fp_v a^{z,\la, \dyn}_{v,u}\fri_{u^{-1}}^\dyn: \bbM^\dyn=\bigoplus_u\bbM_{u^{-1}}^\dyn\dashrightarrow \bigoplus_v \bbM_v= \bbM.
\end{align*}
\end{definition}
Here recall we have embeddings and projections
\[\xymatrix{\bbM\ar@/^.5pc/[r]^-{\fp_w}&\ar@/^.5pc/[l]^-{\fri_w} \bbM_w},\quad  \xymatrix{\bbM^\dyn\ar@/^.5pc/[r]^-{\fp^\dyn_w}&\ar@/^.5pc/[l]^-{\fri^\dyn_w} \bbM^\dyn_w} .\]
That is, $\bbT^{z, \la}$ is defined with $ a^{z,\la}_{v,w}$ as its $(v,w)$-th matrix coefficient. 
By the definition, we have the isomorphism
\[
\bbS''_{w,v}=v^\dynst\bbL\otimes (w^{-1})^*\bbL^{-1}\cong \sHom(\bbM_{w^{-1}}, \bbM_v^\dyn),
\]
which shows that a rational section $\de_v^\dyn a\de_w$ of $\bbS''_{w,v}$ is sent to $\fp_v^\dyn a \fri_{w^{-1}}$ as a rational section of $\sHom(\bbM,\bbM^\dyn)$. 
Since  $ a^{z,\la}_{v,w}$ is a rational section of $\bbS_{w,v}''$, so it  defines an rational map of line bundles \[\bbM_{w^{-1}}\dashrightarrow\bbM^\dyn_v.\]
Similar for $\bbT^{ z,  \la, \dyn}$, which uses 
\[
\bbS'_{w,v}=w^*\bbL\otimes (v^{-1})^\dynst\bbL^{-1}\cong \sHom(\bbM_{v^{-1}}^\dyn, \bbM_w),
\]
where a rational section $\de_wa\de_v^\dyn$ of $\bbS'_{w,v}$ is sent to $\fp_w a\fri^\dyn_{v^{-1}}$ as a rational section of $\sHom(\bbM^\dyn, \bbM)$. 

We are now ready to state the first main result of the present paper. 
\begin{theorem}
\label{thm:inverse_maps}
The following rational maps of vector bundles on $\catA\times\catA^\vee$ are inverses to each other
\begin{align*}
\bbT^{z,\la}&:\bbM\dashrightarrow\bbM^\dyn, & \bbT^{z,\la,\dyn}: \bbM^\dyn\dashrightarrow \bbM.
\end{align*}
\end{theorem}

Recall that (Remark~\ref{rmk:ell_coh}) $\bbM$ and $\bbM^\dyn$ are models of equivariant elliptic cohomologies of $T^*G/B$ and $T^*G^\vee/B^\vee$ respectively. Hence, the content of the theorem is two-fold. The matrix with DL-coefficients induce explicit rational isomorphism between the two cohomology spaces. Moreover, the matrix with DL-coefficients is inverse to the matrix with Langlands dual DL-coefficients. See Theorem~\ref{thm:invmat} for a reformulation in these terms. The proofs of both the reformulation Theorem~\ref{thm:invmat} as well as its equivalence with Theorem~\ref{thm:inverse_maps} are given in Section \S~\ref{sec:ell_class}. Elliptic classes and their Poincar\'e duality, to be developed in the next few sections, make the proofs easier. Moreover, in Theorem~\ref{thm:main}, the behaviour of $\bbT^{z,\la}$ and $ \bbT^{z,\la,\dyn}$ on elliptic classes are specified.

\section{Basic properties of the Poincar\'e line bundle}\label{sec:S_property}

In this section, for the convenience of the readers, we collect some basic properties and the restrictions of  the Poincar\'e line bundle, its shift $\bbL$, and also the twisted group algebra $\bbS$. Some of the properties are used in the proof of Theorem~\ref{thm:Trational} above, and again in later sections. 

\subsection{Weyl group actions on $\calP$}
We first fix some notations.
For the bundle $\calO(z)$ on $\catA^\vee$, denote by $\tilde\calO(z)$ its pull back to $\catA\times \catA^\vee$. Similarly, denote by $\tilde\calO(\la)$ the pull back of $\calO(\la)$ on $\catA$ to $\catA\times \catA^\vee$. The following lemma 
  collects some elementary facts about line bundles, whose proofs are straightforward.

\begin{lemma}\label{lem:basic-bundle}
\begin{enumerate}
\item \label{item:basic-bundle1}For any $ z,z'\in \catA, \la, \la'\in \catA^\vee,  $ we have
\[
\tilde\calO(z)|_{z'}=\calO(z), ~ \tilde\calO(z)|_\la=\calO_{\catA},~\tilde\calO(\la)|_{\la'}=\calO(\la), ~ \tilde\calO(\la)|_z=\calO_{\catA^\vee}.
\]	
\item \label{item:basic-bundle2}For any $v\in W$, we have
\begin{align*}
v^*\tilde\calO(z)&=\tilde\calO(z), & v^{\dyn*}\tilde\calO(z)&=(v^{\dyn*}\calO(z))^\sim=\tilde\calO(v^{-1}(z)), \\
v^{\dyn*}\tilde\calO(\la)&=\tilde\calO(\la), & v^{*}\tilde\calO(\la)&=(v^{*}\calO(\la))^\sim=\tilde\calO(v^{-1}(\la)).
\end{align*}
\item \label{item:basic-bundle3}Let $\phi:\catA\to \catA$, $\phi^\dyn:\catA^\vee\to \catA^\vee$ be maps of varieties.  For any sheaf $\calF$ on $\catA\times \catA^\vee$, we have 
\[
(\phi^*\calF)|_{\la}=\phi^*(\calF|_\la), ~(\phi^*\calF)|_z=\calF|_{\phi(z)}, ~(\phi^{\dyn*}\calF)|_\la=\calF|_{\phi^\dyn(\la)}, ~(\phi^{\dyn*}\calF)|_z=\phi^{\dyn*}(\calF|_z).
\]
\end{enumerate}
\end{lemma}
Recall the convention of Poincar\'e line bundle in \S~\ref{subsec:rho-shift}, which has the property that
\[
\calP|_z=\calO(z-\hbar\rho^\vee), ~\calP|_\la=\calO(\la).
\]
The maps $w^{-1}$ and $w^{\dyn}$ from $\catA\times \catA^\vee$ to itself are related by the following property:
\begin{lemma}For any $w\in W$, we have
 $(w^{-1})^*\calP=w^{\dyn*}\calP\otimes \tilde \calO(w^{-1}\hbar\rho^\vee-\hbar\rho^\vee).$
\end{lemma}
\begin{proof}
By using Lemma \ref{lem:basic-bundle}, we get 
\begin{align*}
(w^{\dyn*}\calP)|_{\hbar\rho^\vee}&=w^{\dyn*}(\calP|_{\hbar\rho^\vee})=w^{\dyn*}\calO_{\catA^\vee}=\calO_{\catA^\vee}, \\
 (w^{\dyn*}\calP)|_\la&=\calP|_{w(\la)}=\calO(w(\la)),\\
((w^{-1})^*\calP)|_{\hbar\rho^\vee}&=\calP|_{w^{-1}(\hbar\rho^\vee)}=\calO(w^{-1}(\hbar\rho^\vee)-\hbar\rho^\vee)=\left(w^\dynst \calP\otimes \tilde\calO(w^{-1}(\hbar\rho^\vee)-\hbar\rho^\vee)\right)|_{\hbar\rho^\vee},\\ ((w^{-1})^*\calP)|_\la&=(w^{-1})^*(\calP|_\la)=(w^{-1})^*\calO(\la)=\calO(w(\la))=\left(w^\dynst \calP\otimes \tilde\calO(w^{-1}(\hbar\rho^\vee)-\hbar\rho^\vee)\right)|_{\la}.
\end{align*}
Therefore, by the See-Saw Lemma, 
\[
(w^{-1})^*\calP=w^{\dyn*}\calP\otimes \tilde \calO(w^{-1}(\hbar\rho^\vee)-\hbar\rho^\vee).
\]
\end{proof}

\subsection{Weyl group actions on $\bbL$}

By the definition of $\bbL$ in Definition \ref{def:bbL} and Lemma \ref{lem:basic-bundle}, it is easy to see that 
\begin{align}\label{eq:Lres}
\bbL|_\la&=\calP|_{\la+\rho\hbar}=\calO(\la+\rho \hbar), & \bbL|_{z }&=\calO(z-\hbar\rho^\vee).
\end{align}
We define a new action of $W$ on $\catA^\vee$ by 
\[w_\bullet^\dyn \la=w(\la+\rho\hbar)-\rho\hbar, ~\la\in \catA^\vee.\]
We then have 
\begin{lemma} For any $w\in W$, we have isomorphisms
\[
w_\bullet^*\calO(\la)=w^*\calO(\la), \quad w_\bullet^{\dyn*}\calO(z)=w^{\dyn*}\calO(z), \quad \forall z\in \catA, \la\in \catA^\vee.
\]
\end{lemma}

\begin{proof}They follow from the fact that degree 0 line bundles are invariant under translations. 
\end{proof}

The following lemma shows that $\bbL$ behaves well in terms of the compatibility of the two $\bullet$-actions. 
\begin{lemma}\label{lem:bullettransfer}
For any $w\in W$, we have a canonical isomorphism
\[
(w_\bullet^{-1})^*\bbL\cong w_\bullet^{\dyn*}\bbL,
\]
which is compatible with composition of group elements. 
\end{lemma}
\begin{proof}
This follows direct from the definition of $w_\bullet^*$, the definition of $\bbL$, as well as the canonical functorial isomorphism \eqref{eqn:Pincare_dot}.
\end{proof}
\begin{remark}
We may also use Lemma \ref{lem:basic-bundle} to compute the restrictions:
\begin{align*}
((w_\bullet^{-1})^{*}\bbL)|_{\la}&=(w_\bullet^{-1})^{*}(\bbL|_{\la})\overset{\eqref{eq:Lres}}=(w^{-1}_\bullet)^{*}\calO(\lambda+\hbar\rho)=(w^{-1})^*\calO(\la+\hbar\rho)=\calO(w(\lambda+\hbar\rho)),\\
(w_\bullet^{-1})^{*}\bbL|_{{\hbar\rho^\vee}}&=\bbL|_{w^{-1}_\bullet\hbar\rho^\vee}=\bbL|_{{\hbar\rho^\vee}}\overset{\eqref{eq:Lres}}=\calO_{\catA^\vee},\\
(w_\bullet^{\dyn*}\bbL)|_{\la}&=\bbL|_{w_\bullet\lambda}\overset{\eqref{eq:Lres}}=\calO(w_\bullet\lambda+\hbar\rho)=\calO(w(\lambda+\hbar\rho)),\\
(w_\bullet^{\dyn*}\bbL)|_{{\hbar\rho^\vee}}&=w_\bullet^{\dyn*}(\bbL|_{\hbar\rho^\vee})\overset{\eqref{eq:Lres}}=w_\bullet^{\dyn*}\calO_{\catA^\vee}=\calO_{\catA^\vee}.
\end{align*}
This also shows the existence of an isomorphism
$w^{\dyn*}_\bullet\bbL\cong (w^{-1}_\bullet)^*\bbL.$ However, Lemma~\ref{lem:bullettransfer} above provides a canonical isomorphism. 
\end{remark}

The following lemma explains how $\bbL$ behaves with the maps on $\catA\times \catA^\vee$. 
\begin{lemma}\label{lem:act}\begin{enumerate}
\item $
w^{\dyn*}\bbL\cong\bbL\otimes w^{\dyn*}\calP\otimes \calP^{-1}.$
\item $w^*\bbL\cong(w^{-1}_\bullet)^{\dyn*}\bbL\otimes\tilde\calO(w\hbar\rho^\vee-\hbar\rho^\vee)\cong \bbL\otimes (w_\bullet^{-1})^{\dyn*}\calP\otimes \calP^{-1} \otimes\tilde\calO(w\hbar\rho^\vee-\hbar\rho^\vee).
$
\end{enumerate}
\end{lemma}

\begin{proof}(1). From Lemma \ref{lem:basic-bundle}, we have
\begin{align*}
(w^{\dyn*}\bbL)|_\la&=(w^{\dyn*}\vor^*\calP)|_\la=\calP|_{w^{\dyn}(\la)+\rho\hbar}=\calP|_{\la+\rho\hbar+w^{\dyn}(\la)-\la}=(\bbL\otimes w^{\dyn*}\calP\otimes \calP^{-1})|_\la,\\
(w^\dynst\bbL)|_{\hbar\rho^\vee}&=w^\dynst(\bbL|_{\hbar\rho^\vee})=w^\dynst\calO_{\catA^\vee}=\calO_{\catA^\vee}=(\bbL\otimes w^{\dyn*}\calP\otimes\calP^{-1})|_{\hbar\rho^\vee}.
\end{align*}
By the See-Saw Lemma, we know that $w^\dynst\bbL\cong\bbL\otimes w^\dynst\calP\otimes \calP^{-1}$.

(2).  Similarly, we have 
\begin{align*}(w^*\bbL)|_{\lambda}&= w^*(\bbL|_\la)= w^*_\bullet(\bbL|_\la)= (w^*_\bullet\bbL)|_\la\overset{\text{Lem. } \ref{lem:bullettransfer}}=((w_\bullet^{-1})^{\dyn*}\bbL)|_\la,\\
(w^*\bbL)|_{\hbar\rho^\vee}&=\bbL|_{w(\hbar\rho^\vee)}=\vor^*(\calP|_{w(\hbar\rho^\vee)})\overset{\sharp}=\calP|_{w(\hbar\rho^\vee)}=\calO(w(\hbar\rho^\vee)-\hbar\rho^\vee),\\
((w_\bullet^{-1})^{\dyn*}\bbL)|_{\hbar\rho^\vee}&=(w_\bullet^{-1})^{\dyn*}(\bbL|_{\hbar\rho^\vee})=(w_\bullet^{-1})^{\dyn*}\calO_{\catA^\vee}=\calO_{\catA^\vee}.
\end{align*}
Here $\sharp$ follows from the invariance of translations  of degree zero line bundles. 
Therefore, the See-Saw Lemma gives
\[
w^*\bbL\cong(w_\bullet^{-1})^{\dyn*}\bbL\otimes \tilde \calO(w(\hbar\rho^\vee)-\hbar\rho^\vee).
\]
\end{proof}

\begin{example}\label{rem:bullet} For each simple root $\al$ and $\la\in \catA^\vee$, Lemma \ref{lem:act} gives
\begin{align*}(s_\al^{\dyn *}\bbL)|_\la&=\bbL|_{s_\al(\la)}=\calO(s_\al(\la)+\rho\hbar)=\calO(\la+\rho\hbar)\otimes \calO(-\la_{\al^\vee}\al)=\bbL|_\la\otimes \calO(-\la_{\al^\vee}\al),\\
(s_\al^{\dyn *}\bbL)|_{\hbar\rho^\vee}&=s_\al^{\dyn*}(\bbL|_{\hbar\rho^\vee})=s_\al^{\dyn*}\calO_{\catA^\vee}=\calO_{\catA^\vee},\\
(s_\al^*\bbL)|_\la&=s_\al^*\calO(\la+\rho\hbar)=\calO(s_\al(\la+\hbar\rho))=\bbL|_\la\otimes \calO(-\la_{\al^\vee}\al-\hbar\al),\\
(s_\al^*\bbL)|_{\hbar\rho^\vee}&=\bbL|_{s_\al(\hbar\rho^\vee)}=\calO(-\hbar\al^\vee). 
\end{align*}
\end{example}

\begin{lemma}\label{lem:salphacomp}For each simple root $\al$, we have
\[
s_\al^*s_\al^{\dyn*}\bbL\cong\bbL\otimes \tilde\calO(-\hbar\al)\otimes \tilde \calO(\hbar\al^\vee)\cong\bbL\otimes (\calO(-\hbar\al)\boxtimes \calO(\hbar\al^\vee)).
\]
\end{lemma}

\begin{proof}We have
\begin{align*}
(s_\al^*s_\al^{\dyn*}\bbL)|_\la&=s_\al^*(s_\al^{\dyn*}\bbL|_\la)=s_\al^*(\bbL|_{s_\al(\la)})=s_\al^*\calO(s_\al(\la)+\rho\hbar)\\
&=\calO(\la+s_\al(\rho\hbar))=\calO(\la+\rho\hbar-\hbar\al),\\
(\bbL\otimes \tilde\calO(-\hbar\al)\otimes \tilde \calO(\hbar\al^\vee))|_\la&=\calO(\la+\rho\hbar)\otimes \calO(-\hbar\al)=\calO(\la+\rho\hbar-\hbar\al),\\
(s_\al^*s_\al^{\dyn*}\bbL)|_{\hbar\rho^\vee}&=(s_\al^{\dyn*}\bbL)|_{s_\al(\hbar\rho^\vee)}=s_\al^{\dyn*}(\bbL|_{s_\al(\hbar\rho^\vee)})=s_\al^{\dyn*}\calO(s_\al(\hbar\rho^\vee)-\hbar\rho^\vee)\\
&=s_\al^{\dyn*}\calO(-\hbar\al^\vee)=\calO(\hbar\al^\vee),\\
(\bbL\otimes \tilde\calO(-\hbar\al)\otimes \tilde\calO(\hbar\al^\vee))|_{\hbar\rho^\vee}&=\bbL|_{\hbar\rho^\vee}\otimes \tilde\calO(-\hbar\al)|_{\hbar\rho^\vee}\otimes \tilde\calO(\hbar\al^\vee)|_{\hbar\rho^\vee}\\
&=\calO(\hbar\rho^\vee-\hbar\rho^\vee)\otimes \calO_{\catA}\otimes \calO(\hbar\al^\vee)=\calO(\hbar\al^\vee).
\end{align*}
The conclusion then follows from the See-Saw Lemma.
\end{proof}

\subsection{Restrictions of $\bbS$}
The following Lemma will be used when we determine zeros and poles of rational sections. 
\begin{lemma}\label{lem:basicS}
\begin{enumerate}
\item \label{item:basicS1}
We have the following restrictions: 
\begin{align*}
\bbS_{w,v}|_\la&=\calO(\la-w_\bullet v^{-1}(\la)),&\quad  \bbS_{w,v}|_{z}&=\calO(z-v_\bullet w^{-1}(z)),\\
\bbS'_{w,v}|_\la&=\calO(w_\bullet^{-1}\la- v^{-1}(\la)),&\quad  \bbS'_{w,v}|_{z}&=\calO(w(z)-v_\bullet z),\\
\bbS''_{w,v}|_\la&=\calO(v(\la)-w_\bullet \la),&\quad  \bbS''_{w,v}|_{z}&=\calO(v_\bullet^{-1}z- w^{-1}(z)).
\end{align*}
\item \label{item:basicS2}If $w=v=s_\al$, we have
\[
\bbS_{\al, \al}=\bbS_{s_\al, s_\al}=\calO(\hbar\al)\boxtimes \calO(-\hbar\al^\vee).
\]
In particular, $\bbS_{w_0,w_0}=\calG\otimes \calH^{-1}$ where $\calG:=\tilde\calO(2\hbar\rho)$, $\calH:=\tilde\calO(2\hbar\rho^\vee)$. 

\end{enumerate}
\end{lemma}

\begin{proof}(1). They follow from  \eqref{eq:Lres} and  Lemma \ref{lem:basic-bundle}. We prove the restrictions to $\la$ as examples.
\begin{align*}
\bbS_{w,v}|_\la&=(\bbL\otimes (w^{-1})^*(v^{-1})^\dynst\bbL^{-1})|_\la=\calO(\la+\rho\hbar)\otimes (w^{-1})^*(\bbL^{-1}|_{v^{-1}(\la)})\\
&=\calO(\la+\rho\hbar)\otimes (w^{-1})^*\calO(-v^{-1}(\la)-\rho\hbar)=\calO(\la+\rho\hbar-wv^{-1}(\la)-w(\rho\hbar))\\
&=\calO(\la-w_\bullet v^{-1}(\la)). \\
\bbS'_{w,v}|_\la&=(w^*\bbS_{w,v})|_\la=w^*(\bbS_{w,v}|_\la)=w^*\calO(\la-w_\bullet v^{-1}(\la))=\calO(w^{-1}_\bullet \la-v^{-1}(\la)),\\
\bbS''_{w,v}|_\la&=(v^\dynst \bbS_{w,v})|_\la=\bbS_{w,v}|_{v(\la)}=\calO(v(\la)-w_\bullet \la)).
\end{align*}
(2). They follow from Part (1), and the identities $s_\al(\rho)=\rho-\al$ and  $w_0(\rho)=-\rho.$
\end{proof}

\section{Duality functors and Poincar\'e pairings} \label{sec:dual}
In this section we introduce an anti-involution of the elliptic twisted group algebra and the Poincar\'e pairings of periodic modules. 

\subsection{Duality functors}
Denote 
\[
\theta_\Pi(\hbar\pm z)=\prod_{\al>0}\theta(\hbar\pm z_\al), \quad \theta_\Pi(\hbar\pm \la)=\prod_{\al>0}\theta(\hbar\pm\la_{\al^\vee}), \quad \theta_\Pi(z)=\prod_{\al>0}\theta(z_\al), \quad \theta_\Pi(\la)=\prod_{\al>0}\theta(\la_{\al^\vee}). 
\]
Note that ${}^w\theta_\Pi(z)=\epsilon_w\theta_\Pi(z)$ and ${}^{w^{\dyn}}\theta_\Pi(\la)=\epsilon_w\theta_\Pi(\la)$. 
Define
\begin{equation}\label{eqn:h_g}
\bfg=\frac{\theta_\Pi(\hbar-z)}{\theta_\Pi(z)},\quad  \bfh=\frac{\theta_\Pi(\hbar-\la)}{\theta_\Pi(\la)}. 
\end{equation}
Recall from Lemma \ref{lem:basicS} that \[\calG=\tilde\calO(2\hbar\rho),\, \ \calH=\tilde\calO(2\hbar\rho^\vee).\]
Note that  $w^*\calH=\calH, w^{\dyn*}\calG=\calG$.

\begin{lemma}\label{lem:GHrational}The function $\bfg=\frac{\theta_\Pi(\hbar-z)}{\theta_\Pi(z)}$ is a rational section of the line bundle $\calG=\tilde\calO(2\hbar\rho )$, and $\bfh=\frac{\theta_\Pi(\hbar-\la)}{\theta_\Pi(\la)}$ is a rational section of the line bundle $\calH=\tilde \calO(2\hbar\rho^\vee)$. Moreover, $\bfg$ is invariant under $w^{\dyn*}$ and $\bfh$ is invariant under $w^*$.
\end{lemma}
\begin{proof}Let $\al>0$. We first look at the function  $\frac{\theta(\hbar-z_\al)}{\theta(z_\al)}$ with variable $z_\al\in E$. It is a rational section of $\calO(\hbar)$ over $E$. Pulling back to $\catA$ via $\chi_\al:\catA\to E$, we see that $\frac{\theta(\hbar-z_\al)}{\theta(z_\al)}$ is a rational section of $\calO(\hbar\al)$ over $\catA$, which can also be viewed as a rational section of $\tilde\calO(\hbar\al)$ over $\catA\times \catA^\vee$. Therefore, the fraction $\frac{\theta_\Pi(\hbar-z)}{\theta_\Pi(z)}$ is a rational section of 
\[
\bigotimes_{\al>0}\tilde\calO(\hbar\al)=\tilde\calO(2\hbar\rho).
\]
The claim for $\bfh$ is proved similarly. The last part follows from Lemma \ref{lem:basic-bundle}. 
\end{proof}

We have the following duality functors 
\[\sHom(-,\calH^{-1}):\Coh(\catA\times \catA^\vee)\to \Coh(\catA\times \catA^\vee), \, \ \sHom(-,\calG):\Coh(\catA\times \catA^\vee)\to \Coh(\catA\times \catA^\vee). \]
The functors given in this form will not be used. 
Instead, below we give simple expressions of these functors on the objects of interests. 
\subsection{Dual algebras and modules}

Define the map
\[D_\la:E\times \catA\times \catA^\vee\to E\times \catA\times \catA^\vee, (\hbar,z, \la)\mapsto (\hbar, z,-\la),\]
 and similarly define $D_z, D_\hbar: E\times \catA\times \catA^\vee\to E\times \catA\times \catA^\vee$ inverting $z$ and $\hbar$ respectively.  
We have the following  ({\it non-commutative}) diagram
\[
\xymatrix{E\times \catA\times\catA^\vee\ar[rd]^-{D_z}\ar[rr]^-{D_\la}&&E\times \catA\times \catA^\vee\\
&E\times \catA\times \catA^\vee\ar[ur]^{D_\hbar}&}.
\]
These maps induce auto-functors on the categories $\Coh^{W\times W^{\dyn}}(\catA\times \catA^\vee)$, $\Coh^{W^{\dyn}}(\catA\times \catA^\vee)$, $\Coh^{W}(\catA\times \catA^\vee)$, and $\Coh(\catA\times \catA^\vee)$.
By definition, it is easy to see the following. 
\begin{lemma}\label{lem:D1}
We have
\begin{enumerate}
 \item  $D_z^*, D_\la^*, D_\hbar^*$ commute with each other, and  commute with $w^*$  and $w^{\dyn*}$;
 \item we  have
 $(D_\la^*)^2=(D_z^*)^2=(D_\hbar^*)^2=\id$. 
 \end{enumerate}
\end{lemma}

\begin{definition}We define 
\begin{align}\bbL(-\la)&=D_\la^*\bbL,& \bbL(-z)&=D_{z}^*\bbL, & \bbS(-\la)&=D_\la^*\bbS,& \bbS(-z)&=D_z^*\bbS.
\end{align}
\end{definition}
We prove some basic properties of these functors. In particular, we are going to see that $D_\la^*\bbL=\sHom(\bbL,\calG)$ and $D_z^*\bbL=\sHom(\bbL,\calH^{-1})$, so $D_\la^*$ and $D_z^*$ are actually the duality functors when applied to $\bbL$.

\begin{lemma}\label{lem:D2}We have the following canonical isomorphisms.
\begin{enumerate}
\item $D_\la^*\calG=\calG, D_\la^*\calH=\calH^{-1}$.
\item $D_z^*\calG=\calG^{-1}$, $D_z^*\calH=\calH$.
\item $D_\hbar^*\calG=\calG^{-1},$ $D_\hbar^*\calH=\calH^{-1}$. 
\item $D_\la^*\bbL=\bbL^{-1}\otimes \calG=\sHom(\bbL,\calG)$,
\item $D_z^*\bbL=\bbL^{-1}\otimes \calH^{-1}=\sHom(\bbL,\calH^{-1})$,
\item $D_\hbar^*\bbL=\bbL\otimes \calG^{-1}\otimes \calH$. 
\item Applied on $\bbL$, $\calH$, and $\calG$, the composition of  two of the maps in $\{D_z^*, D_\la^*, D_\hbar^*\}$ agrees with the third one.  
\end{enumerate}
\end{lemma}
\begin{proof}

(1)-(3). They follow from definitions. 

(4)-(6). They follow from the property of Poincar\'e line bundle together with definitions of  $\bbL$, $\calH$, and $\calG$.
 
(7). We only prove $
D_\la^*\bbL=D_z^*D_\hbar^*\bbL$ as an example. By Parts (1)-(6), we have 
\begin{align*}
D_z^*D_\hbar^*\bbL&=D_z^*(\bbL\otimes \calG^{-1}\otimes \calH)=\bbL^{-1}\otimes \calH^{-1}\otimes \calG^{-1}\otimes \calH=\bbL^{-1}\otimes \calG^{-1}=D_\la^*\bbL.
\end{align*}
\end{proof}

\begin{lemma}\label{lem:D-S}
 We have the following canonical isomorphisms. 
\begin{enumerate}
\item $\bbS(-\la)_{w,v}=\bbL(-\la)\otimes (w^{-1})^*(v^{-1})^{\dyn*}\bbL(-\la)^{-1}.$
\item $\bbS(-z)_{w,v}=\bbL(-z)\otimes (w^{-1})^*(v^{-1})^{\dyn*}\bbL(-z)^{-1}. $
\item $D_\la^*\bbS_{w,v}=\bbS(-\la)_{w,v}=\bbS^{-1}_{w,v}\otimes \calG\otimes (w^{-1})^*\calG^{-1}.$
\item $ D_z^*\bbS_{w,v}=\bbS(-z)_{w,v}=\bbS^{-1}_{w,v}\otimes \calH^{-1}\otimes (v^{-1})^\dynst \calH. $
\item $D_\hbar^*\bbS_{w,v}=\bbS_{w,v}\otimes (w^{-1})^*\calG\otimes \calG^{-1}\otimes (v^{-1})^\dynst \calH^{-1}\otimes \calH.$
\item $D_\la^*\bbS(-z)_{w,v}=\bbS(-z)_{w,v}^{-1}\otimes (w^{-1})^*\calG\otimes \calG^{-1}$. 
\item $D_z^*\bbS(-\la)_{w,v}=\bbS(-\la)_{w,v}^{-1}\otimes \calH\otimes (v^{-1})^\dynst \calH^{-1}$. 
\end{enumerate}
\end{lemma}
\begin{proof}(1). Applying $D_\la^*$ and $D_z^*$ on the identity from the definition of $\bbS_{w,v}$ in Definition \ref{def:bbS}:
\[
\bbS_{w,v}=\bbL\otimes (w^{-1})^*(v^{-1})^\dynst\bbL^{-1},
\]
one will obtain the two identities.

(2)-(6): They all follow from Lemma \ref{lem:D2}. 
\end{proof}

\begin{lemma}
The functors $D_\lambda^*$ and $ D_z^*$ are  monoidal functors  in the tensor monoidal category $\Coh^{W\times W^\dyn}(\catA\times \catA^\vee)$. In particular, $\bbS(-\la)$ and $\bbS(-z)$ are algebra objects acting on the modules  $\bbL(-\la)$ and $\bbL(-z)$, respectively.
\end{lemma}
\begin{proof}They follow from the fact that $D_\la$ and $D_z$ both commute with the Weyl group actions.
\end{proof}
Hence, for any two morphisms $f:\calF_1\to\calF_2$ and $g:\calF_1'\to \calF_2'$, we have  $D^*_\lambda(f\star g)=D^*_\lambda(f)\star D^*_\lambda(g)$.
We can similarly define $\bbS(-\la)'$, $\bbS(-\la)''$, $\bbS(-z)'$, and $\bbS(-z)''$.  
\begin{remark}
Similar to  \S~\ref{subsec:rat_sec}, one can write the twisted product of rational sections of $\bbS(-\la)$ and $\bbS(-z)$. 
Moreover, the actions of $\bbS(-z)$ and $\bbS(-\la)$ on the corresponding modules $\bbM(-\la)$, $\bbM(-\la)^\dyn$, $ \bbM(-z)$, and $\bbM(-z)^\dyn$ can be described in the same way as in \S~\ref{subsec:rat_sec}. 
\end{remark}

It is easy to see that $\bbS(-\la)_{w,v}$ and $\bbS(-z)_{w,v}$ share similar properties as $\bbS_{w,v}$ in Lemma \ref{lem:basicS}. Notably, we have 
\begin{align*}
\bbS(-z)_{w_1w_2,v_1v_2}&=\bbS(-z)_{w_1,v_1}\otimes  (w_1^{-1})^*(v_1^{-1})^{\dyn*}\bbS(-z)_{w_2,v_2},\\
 \bbS(-\la)_{w,v}^{-1}&= (w^{-1})^*(v^{-1})^{\dyn*}\bbS(-\la)_{w^{-1}, v^{-1}},\\
\bbS(-\la)_{\al, \al}&=\calO(\hbar\al)\boxtimes \calO(\hbar\al^\vee),\\
\bbS(-z)_{w_1w_2,v_1v_2}&=\bbS(-z)_{w_1,v_1}\otimes  (w_1^{-1})^*(v_1^{-1})^{\dyn*}\bbS(-z)_{w_2,v_2},\\
\bbS(-z)_{w,v}^{-1}&= (w^{-1})^*(v^{-1})^{\dyn*}\bbS(-z)_{w^{-1}, v^{-1}},\\
\bbS(-z)_{\al, \al}&=\calO(-\hbar\al)\boxtimes \calO(-\hbar\al^\vee).
\end{align*}

\subsection{The degree inversion}\label{subsec:deg_inv}
We define the degree inversion functor, to be used in \S~\ref{subsec:antiinvolution} below.   Define 
\[\fD:\Coh^{W\times W^{\dyn}}(\catA\times \catA^\vee)\to \Coh^{W\times W^{\dyn}}(\catA\times \catA^\vee)\]
\[\calF\mapsto \fD(\calF) \text{ so that } \fD(\calF)_{w,v}=\calF_{w^{-1},v^{-1}}.\]
Note that 
$\fD^2=\id$. 
It is easy to see we have a canonical isomorphism
\begin{align}\label{eq:inv_deg}
\fD(\bbS''_{w,v})\star'\fD(\bbS''_{x,y})\cong\fD(\bbS''_{wx,vy}). 
\end{align}
Similar identity holds if one replaces $\bbS''$ by $\bbS'$ and  $\star'$ by $\star''$. 

The conclusion goes through word-for-word for variant forms of $\bbS''$, i.e., the $\bbS''(-\la)$ and $\bbS''(-z)$.

\subsection{The anti-involutions}\label{subsec:antiinvolution}
From Lemma~\ref{lem:D-S}.(2),  we have the following canonical isomorphisms
\[
\bbS(-\la)''_{w^{-1},v^{-1}}=(v^{-1})^\dynst \bbS(-\la)_{w^{-1},v^{-1}}= w^*\bbS_{w, v}\otimes w^*\calG^{-1}\otimes \calG=\bbS_{w,v}'\otimes w^*\calG^{-1}\otimes \calG.\]
The rational section  $\bfg$ defines a rational section
$\frac{\bfg}{{}^{w^{-1}}\bfg}$ of $w^*\calG^{-1}\otimes \calG$.
We define the rational map 
\begin{equation}
\label{eq:iotaDla}
{\iota}_{\la}:\bbS_{w,v}'\dashrightarrow \bbS(-\la)_{w^{-1},v^{-1}}''=\fD(\bbS(-\la)_{w^{-1}, v^{-1}}), \end{equation} 
by multiplication by $\frac{\bfg}{{}^{w^{-1}}\bfg}$. 
Taking the  sum over all $w,v\in W$, we obtain a rational map of graded coherent sheaves
\begin{equation}\label{eq:iota_la}
{\iota}_{\la}:\bbS'\dashrightarrow  \fD(\bbS(-\la)''), \quad \de_wa\de_v^\dyn\mapsto \de_{v^{-1}}^\dyn a\cdot \frac{\bfg}{{}^{w^{-1}\bfg}}\de_{w^{-1}}.\end{equation}

\begin{prop}\label{prop:anti}The rational map $\iota_{\la}$ defines an anti-homomorphism of sheaves of algebras. That is, we have the following commutative diagram
\[
\xymatrix{\bbS'_{w_1,v_1}\star' \bbS'_{w_2,v_2}\ar[d]_-{\sim}\ar@{-->}[rr]^-{\sigma\circ (\iota_{\la}\times \iota_{\la})}&&\fD(\bbS(-\la)_{w_2^{-1},v_2^{-1}}'')\star' \fD(\bbS(-\la)_{w_1^{-1},v_1^{-1}}'')\ar[d]^-\sim_-{\eqref{eq:inv_deg}}\\
\bbS'_{w_1w_2,v_1v_2}\ar@{-->}[rr]_-{\iota_{\la}}&&\fD(\bbS(-\la)_{(w_1w_2)^{-1},(v_1v_2)^{-1}}'')}
\]
Here $\sigma$ means switching the two components.
\end{prop}
\begin{proof}
The  top map sends a rational section $(\de_{w_1}a_1\de_{v_1}^\dyn, \de_{w_2}a_2\de_{v_2}^\dyn)$ to 
\[
\de_{v_2^{-1}}^\dyn a_2\cdot \frac{\bfg}{{}^{w_2^{-1}}\bfg}\de_{w_2^{-1}}\de_{v_1^{-1}}^\dyn a_1\cdot \frac{\bfg}{{}^{w_1^{-1}}\bfg}\de_{w_1^{-1}}=\de_{(v_1v_2)^{-1}}^\dyn{}^{v_1^\dyn}a_2\cdot {}^{w_2^{-1}}a_1\cdot \frac{\bfg}{{}^{w_2^{-1}}\bfg}\cdot {}^{w_2^{-1}}(\frac{\bfg}{{}^{w_1^{-1}}\bfg})\de_{(w_1w_2)^{-1}}.\]
On the other hand, we have 
\[
\iota_\la(\de_{w_1}a_1\de_{v_1}^\dyn\de_{w_2}a_2\de_{v_2}^\dyn)=\iota_\la (\de_{w_1w_2}{}^{w_2^{-1}}a_1\cdot {}^{v_1^\dyn}a_2\de_{v_1v_2}^\dyn)=\de_{(v_1v_2)^{-1}}^\dyn{}^{v_1^\dyn}a_2\cdot {}^{w_2^{-1}}a_1\cdot \frac{\bfg}{{}^{(w_1w_2)^{-1}}\bfg}\de_{(w_1w_2)^{-1}}. 
\]
Therefore, the two agree.
\end{proof}
Similarly, the rational section $\bfh$ of $\calH$ 
and the canonical isomorphisms \[\bbS(-z)_{w^{-1}, v^{-1}}''=(v^{-1})^\dynst \bbS(-z)_{w^{-1},v^{-1}}=w^*\bbS_{w,v}\otimes (v^{-1})^\dynst \calH^{-1}\otimes \calH=\bbS'_{w,v}\otimes (v^{-1})^\dynst  \calH^{-1}\otimes \calH \] define an anti-homomorphism by multiplying by $\frac{\bfh}{{}^{v^\dyn}\bfh}$:
\begin{equation}\label{eq:iota_z}
\iota_{z}: \bbS'_{w, v}\dashrightarrow \bbS(-z)''_{w^{-1},v^{-1}}, \quad \de_w a\de_v^\dyn\mapsto \de_{v^{-1}}^\dyn a\cdot \frac{\bfh}{{}^{v^\dyn}\bfh}\de_{w^{-1}}. 
\end{equation}
The statement corresponding to Proposition \ref{prop:anti} holds as well, that is, 
\[
\iota_z(\de_{w_1}a_1\de_{v_1}^\dyn\de_{w_2}a_2\de_{v_2}^\dyn)=\iota_z(\de_{w_2}a_2\de_{v_2}^\dyn)\iota_z(\de_{w_1}a_1\de_{v_1}^\dyn).
\]

\begin{remark}Via the equivalences of categories in Lemma  \ref{lem:monoidal_equivalence}, we can replace the domains and codomains of $\iota_\la$ and $\iota_z$ by the other versions of $\bbS, \bbS(-\la)$ and $\bbS(-z)$. For instance, if one considers $\iota_\la:\bbS_{w,v}(-\la)\dashrightarrow \bbS_{w^{-1}, v^{-1}}'$, in terms of rational sections, one would get
\[
\iota_\la(a\de_w\de_v^\dyn)=\iota_\la(\de_w {}^{w^{-1}}a\de_v^\dyn)=\de_{v^{-1}}^\dyn{}^{w^{-1}}a\cdot \frac{\bfg}{{}^{w^{-1}}\bfg}\de_{w^{-1}}=\de_{w^{-1}}{}^{(v_1^{-1})^\dyn}a\cdot \frac{{}^w\bfg}{\bfg}\de_{v^{-1}}^\dyn.
\]
\end{remark}

\subsection{Poincar\'e pairings on periodic modules}
Recall that \begin{align*}
\bbM&=\bigoplus_{x\in W}\bbM_x=\bigoplus_{x\in W}x^*\bbL, &\bbM^\dyn&=\bigoplus_{y\in W}\bbM^\dyn_y=\bigoplus_{y\in W}y^{\dyn*}\bbL.
\end{align*}
Applying $D^*_\la$ and $D^*_z$ give 
\begin{align*}
\bbM(-\la)&:=\bigoplus_{x\in W}\bbM(-\la)_x=\bigoplus_{x\in W}x^*\bbL(-\la), &\bbM(-\la)^\dyn&:=\bigoplus_{w\in W}\bbM(-\la)^\dyn_y=\bigoplus_{y\in W}y^{\dyn*}\bbL(-\la),\\
\bbM(-z)&:=\bigoplus_{x\in W}\bbM(-z)_x=\bigoplus_{x\in W}x^*\bbL(-z), &\bbM(-z)^\dyn&:=\bigoplus_{y\in W}\bbM(-z)_y^\dyn=\bigoplus_{y\in W}y^{\dyn*}\bbL(-z).
\end{align*}
Similar as in \S~\ref{subsec:act_periodic}, $\bbM(-\la)$ and $\bbM(-\la)^\dyn$ (resp. $\bbM(-z)$ and $\bbM(-z)^\dyn$) are modules over $\bbS(-\la)$, $\bbS(-\la)'$, and $\bbS(-\la)''$ (resp. $\bbS(-z)$, $ \bbS(-z)'$, and $\bbS(-z)''$), with actions denoted by $\bullet$\ and $\bullet^\dyn$.

We define the following pairing
\begin{equation}
\lag\_, \_\rag_{\la}: \bbM\otimes \bbM(-\la)\dashrightarrow \calO, ~\bbM_x\otimes \bbM(-\la)_x=x^*\calG\overset{\frac{1}{{}^{x^{-1}}\bfg}}\dashrightarrow \calO.
\end{equation}
\begin{equation}\label{eq:D-la-pair}
\lag af_{x_1}, bf_{x_2}\rag_{\la}\mapsto \de_{x_1,x_2}^{\Kr}a\cdot b\cdot \frac{1}{{}^{x^{-1}}\bfg}.
\end{equation}

\begin{prop}[Adjunction]\label{prop:adj}
For any $v,w,x\in W$, the following diagram commutes 
\[
\xymatrix{
v^\dynst\big((\bbS_{w,v}'\star' \bbM_x)\otimes \bbM(-\la)_{wx}\big) \ar@{-->}[d]^-{
(p\star'af_x,bf_{wx})
\mapsto (af_x,\iota_\la(p)\star'' bf_{wx})
}\ar[rr]^-{v^\dynst(\bullet\otimes
\id)}&&v^\dynst\big(\bbM_{wx}\otimes \bbM(-\la)_{wx}\big)\ar@{-->}[rr]^-{v^\dynst\lag-,-\rag_\la}&&\calO\ar@{=}[d]\\
\bbM_x\otimes \left(\bbS''(-\la)_{w^{-1}, v^{-1}}\star'' \bbM(-\la)_{wx}\right)\ar[rr]_-{\id\otimes \bullet}& &\bbM_x\otimes \bbM(-\la)_x\ar@{-->}[rr]_-{\lag-,-\rag_\la} && \calO}.
\]
Here the left vertical map is given by the composite
\[
\xymatrix{v^\dynst\big((\bbS_{w,v}'\star' \bbM_x)\otimes \bbM(-\la)_{wx}\big)\ar[r]^-\sim &v^\dynst x^*\bbS'_{w,v}\otimes \bbM_x\otimes v^\dynst \bbM(-\la)_{wx}\ar@{-->}[d]^-{\iota_\la\otimes \id\otimes \id}\\
\bbM_x\otimes \left(\bbS''(-\la)_{w^{-1}, v^{-1}}\star'' \bbM(-\la)_{wx}\right)& x^*v^\dynst \bbS(-\la)''_{w^{-1}, v^{-1}}\otimes \bbM_x\otimes v^\dynst \bbM(-\la)_{wx}\ar[l]^-\sim}.
\]
In particular, for any rational section $p$ of $\bbS_{w,v}'$, and $f$ and $f'$ rational sections of the corresponding modules, we have 
\begin{align}
\label{eq:adjD-la}{}^{(v^{-1})^{\dyn}}\lag p\bullet f,f'\rag_{\la}&=\lag f,\iota_{\la}(p)\bullet f'\rag_{\la}.
\end{align}
\end{prop}

\begin{proof} One can check the commutativity by using the formulas of $\iota_\la$ in \eqref{eq:iota_la} and the pairing $\lag\_, \_\rag_\la$ in \eqref{eq:D-la-pair}. 
\end{proof}

\begin{remark}
Similarly, we also have the following pairings:
\begin{align*}
\lag\_, \_\rag_{ \la}^\dyn:\bbM^\dyn\otimes \bbM(-\la)^\dyn\dashrightarrow \calO,& \quad 
\lag\_, \_\rag_{ \la}^\dyn:\bbM_y^\dyn\otimes \bbM(-\la)^\dyn_y=\calG\overset{\frac{1}{{}^{}\bfg}}\dashrightarrow \calO.\\
\lag\_, \_\rag_{ z}: \bbM\otimes \bbM(-z)\dashrightarrow \calO, &\quad 
\lag\_, \_\rag_{ z}: \bbM_x\otimes \bbM(-z)_x=\calH^{-1}\overset\bfh\dashrightarrow  \calO.\\
\lag\_, \_\rag_{ z}^\dyn: \bbM^\dyn\otimes \bbM(-z)^\dyn
\dashrightarrow \calO,&\quad \lag\_,\_\rag_z^\dyn: \bbM_y\otimes \bbM(-\la)_y=y^\dynst\calH^{-1}\overset{{}^{(y^{-1})^\dyn}\bfh}\dashrightarrow  \calO.
\end{align*}
For instance, in terms of rational sections, the last one can be written as 
\begin{equation}\label{eq:D-z-dyn-pair}
\lag af_{y_1}^\dyn , bf_{y_2}^\dyn\rag_{ z}^\dyn\mapsto \de^{\Kr}_{y_1,y_2}a\cdot b\cdot {}^{(y_1^{-1})^\dyn}\bfh.
\end{equation}
They satisfy the following adjointness: given $p$ a rational section of $\bbS_{w,v}$ (or $\bbS'_{w,v}$ or $\bbS''_{w,v}$),
\begin{align}
\label{eq:adjD-la-dyn}\lag p\bullet^\dyn f^\dyn, (f')^\dyn\rag_{ \la}^\dyn&={}^w\lag f^\dyn, \iota_{ \la}(p)\bullet^\dyn (f')^\dyn\rag_{ \la},\\
\label{eq:adjD-z}\lag p\bullet f,f'\rag_{ z}&={}^{v^{\dyn}}\lag f,\iota_{ z}(p)\bullet f'\rag_{ z},\\
\label{eq:adjD-z-dyn}\lag p\bullet^\dyn f^\dyn, (f')^\dyn\rag_{ z}^\dyn&={}^w\lag f^\dyn, \iota_{ z}(p)\bullet^\dyn (f')\rag^\dyn_{ z}.
\end{align}
\end{remark}

\section{Langlands duality}\label{sec:Langlands}
In this section we consider the Langlands dual system. In Proposition \ref{prop:i*S} and Lemma \ref{lem:i*T} below, we identify   the DL operators of the Langlands dual system with the operators  $T^{z,\la,\dyn}_\al$ defined in Section \S~\ref{sec:DL}. 

\subsection{Langlands dual root system}
Let $\calA=\catA^\vee$ and hence we have a natural isomorphism of abelian varieties  $\calA^\vee\cong\catA$ that preserves group structures. Notice that  $\calA=X^*\otimes E$, and $\calA^\vee:=X_*\otimes E$, with the lattices $X^*$ and $X_*$ coming from the root system and hence an action of its Weyl group $W$. The actions are denoted by $w^\vee$ on $\calA$ and $w^{\vee\dyn}$ on $\calA^\vee$, respectively.

Fixing $\hbar':=-\hbar\in E$, the construction of \S~\ref{subsec:rho-shift}
 defines a line bundle $\bbL^\vee$ on $\calA\times\calA^\vee$.
For concreteness and to avoid possible confusions, we spell out the line bundle  $\bbL^\vee$\footnote{In the present paper we avoid the usual dual of vector bundles and reserve the notation $\vee$ for Langlands duality, which is consistent with the dual of abelian varieties.}. 
The construction of \S~\ref{subsec:rootsys} using the lattices $X^*$ and $X_*$ defines coordinates $(z^\vee,\lambda^\vee)\in \calA\times\calA^\vee$.
Clearly we have an identification $(z^\vee,\lambda^\vee)\in \calA\times\calA^\vee$ with $(\lambda, z)\in \catA^\vee\times \catA$. Moreover, $(\rho^\vee)^\vee$ is identified with $\rho$. Then similar as in \eqref{eq:Lres}, we have:
\begin{equation}\label{eq:Lvee-res}
\bbL^\vee|_{\la^\vee}=\calO(\la^\vee+\hbar'\rho^\vee)=\calO(\la^\vee-\hbar\rho^\vee), ~\bbL^\vee|_{z^\vee}=\calO(z^\vee-\hbar'\rho)=\calO(z^\vee+\hbar\rho).
\end{equation}
As a consequence, we obtain the following. 
\begin{prop}\label{prop:i-Lvee}
Let $i:\catA\times\catA^\vee\to \calA\times\calA^\vee$ be the isomorphism $\lambda\mapsto z^\vee$ and $z\mapsto \lambda^\vee$. Then there is a canonical isomorphism \[i^*\bbL^\vee=\bbL.\] 
\end{prop}
\begin{proof}This follows from the definition of Poincar\'e line bundles and the following identifications under $i$
\[\hbar'\rho^\vee=-\hbar\rho^\vee\hbox{ and }-\hbar'(\rho^\vee)^\vee=\hbar\rho.\]
\end{proof}
It is easy to see that
\begin{align}\label{eq:dualweyl}
w^*i^*&=i^*(w^{\vee})^{\dyn*}, & v^{\dyn*}i^*&=i^*(v^\vee)^*.
\end{align}
Similar as Definition \ref{def:bbS}, we define
\[
\bbS^\vee_{w,v}=\bbL^\vee\otimes ((w^\vee)^{-1})^*((v^\vee)^{-1})^{\dyn*}(\bbL^\vee)^{-1}.
\]
\begin{prop}\label{prop:i*S}We have 
\[
i^*\bbS^\vee_{w,v}=\bbS_{v,w},\quad  i^*\bbS^{\vee'}_{w,v}=\bbS''_{v,w}, \quad i^*\bbS^{\vee''}_{w,v}=\bbS'_{v,w}. 
\]
\end{prop}
\begin{proof}The first one follows from \eqref{eq:dualweyl} and Proposition \ref{prop:i-Lvee}. For the second one, we have
\[
i^*(\bbS^\vee_{w,v})'=i^*{w^\vee}^*\bbS^\vee_{w,v}=w^{\dyn*}i^*\bbS^\vee_{w,v}=w^{\dyn*}\bbS_{w,v}=\bbS''_{v,w}.
\]
The third identity can be proved similarly.
\end{proof}
\begin{example}
We have
\begin{align*}
(s_\al^{\vee*}\bbL^\vee)|_{z^\vee}&=\bbL^\vee|_{s_\al^\vee(z^\vee)}=\calO(s_\al^\vee(z^\vee)+\hbar\rho)=\calO(z^\vee+\hbar\rho-z^\vee_{\al^\vee}\al)=\bbL^\vee|_{z^\vee}\otimes \calO(-z^\vee_{\al^\vee}\al),\\
 ~(s_\al^{\vee*}\bbL^\vee)|_{\la^\vee}&=s_\al^{\vee*}(\bbL^\vee|_{\la^\vee})=s_\al^{\vee*}\calO(\la^\vee-\hbar\rho^\vee)=\calO(s_\al^{\vee\dyn}(\la^\vee-\hbar\rho^\vee))\\
&=\calO(\la^\vee-\hbar\rho^\vee-\la^\vee_{\al}\al^\vee+\hbar\al^\vee)=\bbL^\vee|_{\la^\vee}\otimes \calO(\hbar\al^\vee-\la^\vee_\al \al^\vee).\\
(s_\al^{\vee})^{\dyn*}\bbL^\vee|_{z^\vee}&=(s_{\al}^\vee)^{\dyn*}(\bbL^\vee|_{z^\vee})=(s_\al^\vee)^{\dyn*}\calO(z^\vee+\hbar\rho)=\calO(s_\al^\vee(z^\vee+\hbar\rho))=\calO(z^\vee+\hbar\rho-z^\vee_{\al^\vee}\al-\hbar\al), \\
(s_\al^\vee)^{\dyn*}\bbL^\vee|_{\la^\vee}&=\bbL^\vee|_{s_\al^\vee(\la^\vee)}=\calO(s_\al^\vee(\la^\vee)-\hbar\rho^\vee)=\calO(\la^\vee-\la^\vee_\al\al^\vee-\hbar\rho^\vee).
\end{align*}
Similar as Lemma \ref{lem:basicS}, we have 
\[
\bbS^\vee_{\al, \al}=\bbL^\vee\otimes (s_\al^*s_\al^\dynst \bbL^{\vee})^{-1}=\calO(\hbar\al)\boxtimes \calO(-\hbar\al^\vee).
\]
One can then verify that $i^*\bbS_{\al,\al}^\vee=\bbS_{\al,\al}$ because from Lemma \ref{lem:basicS}, we have $\bbS_{\al,\al}=\calO(\hbar\al)\boxtimes \calO(-\hbar\al^\vee)$.
\end{example}

\subsection{ DL operators of  the Langlands dual system}
We define the dynamical DL operators for the Langlands dual system (recall from \S\ref{sec:Langlands} that $\hbar'=-\hbar$):
\begin{align*}
T^{z^\vee,\la^\vee}_{\al^\vee}&=\de_\al^{\dyn\vee}\frac{\theta(\hbar)\theta(z^\vee_{\al^\vee}+\la^\vee_{\al})}{\theta(z^\vee_{\al^\vee})\theta(\hbar-\la^\vee_{\al})}+\de_\al^{\dyn\vee}\frac{\theta(\hbar+z^\vee_{\al^\vee})\theta(-\la^\vee_{\al})}{\theta(z^\vee_{\al^\vee})\theta(\hbar-\la^\vee_{\al})}\de_\al^{\vee}. \\
T^{z^\vee,\la^\vee,\dyn}_{\al^\vee}&=\de_\al^\vee\frac{\theta(\hbar)\theta(\la^\vee_{\al}+z^\vee_{\al^\vee})}{\theta(\la^\vee_{\al})\theta(\hbar+z^\vee_{\al^\vee})}+\de_\al^\vee\frac{\theta(\hbar-\la^\vee_{\al})\theta(-z^\vee_{\al^\vee})}{\theta(\la^\vee_{\al})\theta(\hbar+z^\vee_{\al^\vee})} \de_\al^{\dyn\vee}. 
\end{align*}
By Theorem \ref{thm:Trational}, 
they are rational sections of  $\bbS^{\vee''}$ and $\bbS^{\vee'}$, respectively. 

The following lemma explains how we obtain the  operators $T_\al^{z,\la,\dyn}$ from $T^{z^\vee,\la^\vee}_{\al^\vee}$, that is, $T^{z,\la,\dyn}_\al$ is really the corresponding ` $T^{z,\la}_\al$ ' in the Langlands dual system. 
\begin{prop}\label{lem:i*T} We have 
\begin{equation}
i^* T_{\al^\vee}^{z^\vee, \la^\vee, \dyn}=T_\al^{z,\la},\quad  i^*T_{\al^\vee}^{z^\vee, \la^\vee}=T_\al^{z,\la,\dyn}. 
\end{equation}
\end{prop}
\begin{proof}They follow from Proposition \ref{prop:i*S} and direct computation using the definition of $i^*$. 
\end{proof}

\section{More on elliptic DL operators}\label{sec:DL}

In this section,  for the convenience of the readers and for later use, we collect some variant forms of elliptic  DL operators with dynamical parameters, and show their adjoint properties. 

\subsection{More DL operators}
\label{subsec:DL_operator}In addition to the elliptic DL operators defined in \S~\ref{subsec:DL1}, using $\iota_{ \la}, \iota_{ z}$ defined in \eqref{eq:iota_la} and \eqref{eq:iota_z}, we define 
\begin{align*}
T_{\al}^{z,-\la}=\iota_{ \la}(T_\al^{z,\la}), & ~T_\al^{z, -\la,\dyn}=\iota_{ \la}(T^{z,\la,\dyn}_{\al}), \\
T_\al^{-z, \la}=\iota_{ z}(T^{z,\la}_\al), &~T_\al^{-z, \la,\dyn}=
\iota_{ z}(T_\al^{z,\la, \dyn}), 
\end{align*}
explicitly given by
 \begin{align}
\label{eq:Tz-la} T_{\al}^{z,-\la}&=\de_\al^\dyn\frac{\theta(\hbar)\theta(z_\al-\la_{\al^\vee})}{\theta(z_\al)\theta(\hbar-\la_{\al^\vee})}+\de_\al^\dyn\frac{\theta(\hbar-z_\al)\theta(\la_{\al^\vee})}{\theta(z_\al)\theta(\hbar-\la_{\al^\vee})}\de_\al , \\
\label{eq:Tz-ladyn}T_{\al}^{z,-\la, \dyn}&=\de_\al\frac{\theta(\hbar)\theta(z_\al-\la_{\al^\vee})}{\theta(\la_{\al^\vee})\theta(\hbar-z_\al)}+\de_\al\frac{\theta(\hbar-\la_{\al^\vee})\theta(z_\al)}{\theta(\la_{\al^\vee})\theta(\hbar-z_\al)}\de_\al^{\dyn}, \\
\label{eq:T-zla} T_{\al}^{-z,\la}&=\de_\al^\dyn\frac{\theta(\hbar)\theta(z_\al-\la_{\al^\vee})}{\theta(-z_\al)\theta(\hbar+\la_{\al^\vee})}+\de_\al^{\dyn}\frac{\theta(\hbar+z_\al)\theta(\la_{\al^\vee})}{\theta(z_\al)\theta(\hbar+\la_{\al^\vee})}\de_\al, \\
\label{eq:T-zladyn}T_{\al}^{-z,\la, \dyn}&=\de_\al\frac{\theta(\hbar)\theta(\la_{\al^\vee}-z_\al)}{\theta(\la_{\al^\vee})\theta(\hbar+z_\al)}+\de_\al\frac{\theta(\hbar+\la_{\al^\vee})\theta(z_\al)}{\theta(\la_{\al^\vee})\theta(\hbar+z_\al)} \de_\al^{\dyn}. 
\end{align}
Note that the operators $T^{\pm z,\pm \la}_\al$ have constant degree $s_\al^\dyn$ at $W^\dyn$, and $T^{\pm z,\pm \la, \dyn}_\al$ have constant degree $s_\al$ at $W$. They are the dynamical DL operators for $G/B$ and for its Langlands dual $G^\vee/B^\vee$, respectively.

\begin{remark} \label{rem:various_T}
Notice that $T^{z,-\la}_\al$ and $T^{z,-\la,\dyn}_\al$ are defined to be $\iota_\lambda$ applied to $T^{z,\la}_\al$ and $T^{z,\la,\dyn}_\al$ respectively. This is in general not the same as applying $D_\la^*$ to $T^{z,\la}_\al$ and $T^{z,\la,\dyn}_\al$, while the latter simply substitute $-\la$ in place of $\la$ in \eqref{eq:Tzla} and \eqref{eq:Tzladyn} respectively. Indeed, one can simply check that  $T^{z,-\la}_\al= D_\la^*T^{z,\la}_\al$ but $T^{z,-\la,\dyn}_\al\neq D_\la^*T^{z,\la, \dyn}_\al$. Similarly, $T^{-z,\la,\dyn}_\al= D_z^*T^{z,\la,\dyn}_\al$ but $T^{-z,\la}_\al\neq  D_z^*T^{z,\la}_\al$.
\end{remark}

\begin{remark}
The formula of DL operator with dynamical parameters first appeared in \cite{RW20}. Indeed, one obtains $-T_\al^{z,-\la}$ from  the formula in \cite[Theorem 1.3]{RW20} by pluging-in $-\ln h=\hbar$, $\ln h^{\al^\vee}=\la_{\al^\vee}$, and $c_1^{coh}(\calL_\al)=z_\al$ with $\calL_\al=T\times_B\bbC_{-\al}$.
The same formula is also obtained from Felder's elliptic R-matrices with dynamical parameters, which originated from solutions to the dynamical Yang-Baxter equation \cite[Remark 4.10]{ZZ22}.
\end{remark}

Similar to Theorem~\ref{thm:Trational}, we have the following.
\begin{theorem}\label{thm:Trational2}
\begin{enumerate}
\item The operators $T^{z,-\la}_\al$  and $T^{z,-\la,\dyn}_\al $ are rational sections of $\bbS(-\la)''_{}$ and $\bbS(-\la)'_{}$, respectively.
\item The  operators $T^{-z,\la}_\al$  and $T^{-z,\la,\dyn}_\al $ are rational sections of $\bbS(-z)''_{}$ and $\bbS(-z)'_{}$, respectively.
\end{enumerate}
\end{theorem}



\subsection{Adjunctions}\label{subsec:braid}
We use $?$ as a place-holder to stand for one of the indices $({\pm z,\pm\la}), ({\pm z,\pm\la},\dyn)$.
It is proved in \cite{RW20} that $(T^?\al)^2=1$ for any $\alpha\in\Sigma$ and that the braid relations are satisfied (see also \cite[Proposition~4.11]{ZZ22} for the present form). 
Given a reduced sequence of $v\in W$, by using the twisted product in \S~\ref{subsec:Sproductrational},  one define $T^{?}_v$ as the product of corresponding $T^?_\al$. 
Expanding using the direct sums  $\bbS'=\bigoplus_{v,w}\bbS'_{w,v}$ and $\bbS''=\bigoplus_{v,w}\bbS''_{w,v}$, we define the following coefficients via
\begin{align}
T^{\pm z,\pm\la}_{v}&=\sum_{w\le v}\de_v^\dyn a^{\pm z,\pm \la}_{v,w}\de_w,& T^{\pm z,\pm\la,\dyn}_{v}&=\sum_{w\le v}\de_v a^{\pm z,\pm \la, \dyn}_{v,w}\de_w^\dyn.
\end{align}
Note that $T^{\pm z, \pm \la}_v$ have constant degree $v\in W^\dyn$ and $T^{\pm z, \pm \la, \dyn}_v$ have constant degree $v\in W$, respectively. 
Standard calculation using  the twisted product gives \begin{align}\label{eq:aw0}
a^{z,\la}_{w_0,w_0}&=\frac{\bfg}{{}^{w_0^\dyn}\bfh},& a^{z,-\la}_{w_0,w_0}&=\frac{\bfg}{\bfh},& a^{-z,\la}_{w_0,w_0}&=\frac{{}^{w_0}\bfg}{{}^{w_0^\dyn}\bfh}.\\
\label{eq:aw0dyn} a^{z,\la, \dyn}_{w_0,w_0}&=\frac{{}^{w_0^\dyn}\bfh}{\bfg},&
a^{z,-\la, \dyn}_{w_0,w_0}&=\frac{\bfh}{\bfg}, &a^{-z,\la, \dyn}_{w_0,w_0}&=\frac{{}^{w_0^\dyn}\bfh}{{}^{w_0}\bfg}.
\end{align}
We compute $a^{z,\la}_{w_0,w_0}$ as an example. Let $w_0=s_{i_1}\cdots s_{i_k}$. We have 
\begin{align*}
\de_{w_0}^\dyn a^{z,\la}_{w_0,w_0}\de_{w_0}=\de_{\al_{i_1}}^\dyn \frac{\theta(\hbar-z_{\al_{i_1}})\theta(-\la_{\al^\vee_{i_1}})}{\theta(z_{\al_{i_1}})\theta(\hbar+\la_{\al^\vee_{i_1}})}\de_{\al_{i_1}}\cdots \de_{\al_{i_k}}^\dyn \frac{\theta(\hbar-z_{\al_{i_k}})\theta(-\la_{\al^\vee_{i_k}})}{\theta(z_{\al_{i_k}})\theta(\hbar+\la_{\al^\vee_{i_k}})}\de_{\al_{i_k}}.
\end{align*}
By moving all $\de_\al^\dyn$ to the left and $\de_\al$ to the right, we see that the coefficient in the middle is 
\[
\prod_{\al>0}  \frac{\theta(\hbar-z_{\al})\theta(-\la_{\al^\vee})}{\theta(z_{\al})\theta(\hbar+\la_{\al^\vee})}=\frac{\bfg}{{}^{w_0^\dyn}\bfh}.
\]

Furthermore, by using the anti-involutions in \S\ref{subsec:antiinvolution}, we see that 
\begin{align*}\iota_\la(T_v^{z,\la})&=T_{v^{-1}}^{z,-\la}, &\iota_\la(T_v^{z,\la, \dyn})&=T_{v^{-1}}^{z,-\la, \dyn},\\
\iota_z(T_v^{z,\la})&=T_{v^{-1}}^{-z,\la}, &\iota_z(T_v^{z,\la, \dyn})&=T_{v^{-1}}^{-z,\la, \dyn}.
\end{align*}
\begin{lemma}\label{lem:a-iota}
We have the following identity
\begin{align*}
{}^w\bfg \cdot a^{z,\pm\la}_{v,w}&=\bfg\cdot{}^{(v^{-1})^\dyn w}a^{z, \mp\la}_{v^{-1}, w^{-1}}.
\end{align*}
\end{lemma}

\begin{proof} We have
\begin{align*}
\sum_w\de_v^\dyn a^{z,\pm \la, }_{v,w}\de_w&=T^{z,\pm \la, }_v=\iota_{\la}(T^{z,\mp \la, }_{v^{-1}})=\iota_{\la}(\sum_w\de_{v^{-1}}^\dyn a^{z,\mp \la, }_{v^{-1}, w^{-1}}\de_{w^{-1}})\\
&\overset{\eqref{eq:iota_la}}=\sum_w\de_v^\dyn\frac{\bfg}{{}^w\bfg}\cdot  {}^{(v^{-1})^\dyn w }a^{z,\mp \la, }_{v^{-1}, w^{-1}}\de_w.
\end{align*}
Comparing the coefficients of $\de_v^\dyn ?\de_w$, we obtain the conclusion. 
\end{proof}
\begin{remark}\label{rem:acoeffrelation}We can also verify 
\begin{align*}
\bfg \cdot a^{z,\pm \la, \dyn}_{v,w}&={}^{v^{-1}}\bfg\cdot {}^{w^\dyn v^{-1}}a^{z,\mp\la, \dyn}_{v^{-1}, w^{-1}},
&{}^{(v^{-1})^\dyn }\bfh\cdot a^{\pm z,\la}_{v,w}&=\bfh\cdot {}^{(v^{-1})^\dyn w}a^{\mp z, \la}_{v^{-1}, w^{-1}}, \\
 \bfh\cdot a_{v,w}^{\pm z, \la, \dyn}&={}^{w^\dyn}\bfh \cdot {}^{v^{-1}w^\dyn}a_{v^{-1}, w^{-1}}^{\mp z, \la, \dyn}.&&
\end{align*}
\end{remark}

Following from Proposition \ref{prop:adj}, we have 
\begin{corollary}\label{cor:Tadj}
For any rational sections $f,g, f^\dyn, g^\dyn$ of $\bbM$,  $\bbM(-\la)$, $\bbM^\dyn$ and $\bbM(-z)^\dyn$, correspondingly, we have
\begin{align}\label{eq:adjT-la}
\lag T^{z,\la}_{v^{-1}}\bullet f,g\rag_{ \la}&={}^{(v^{-1})^\dyn}\lag f, T^{z,-\la}_v\bullet g\rag_{ \la}, \\
 \lag T^{z,\la, \dyn}_{v^{-1}}\bullet^\dyn f^\dyn,g^\dyn\rag_{ z}&={}^{v^{-1}}\lag f^\dyn, T^{-z, \la,\dyn}_v\bullet^\dyn g^\dyn\rag_{ z}.\end{align}
\end{corollary}

\section{Elliptic classes}\label{sec:ell_class}
In this section, we define  the elliptic  classes for $T^*G/B$ and its Langlands dual as rational sections of the periodic modules. We also prove our main results (Theorem \ref{thm:inverse_maps} and Theorem \ref{thm:main}). Our definition of elliptic classes relies on a pre-fixed rational section $\fc$ of $\bbL.$

\subsection{Main theorem on elliptic classes}

We fix a  rational section $\fc$ of $\bbL$.  Recall the rational sections $\bfh$ and $\bfg$ from \eqref{eqn:h_g}. 
\begin{lemma}\label{lem:rationalsec}The four functions 
\[
\frac{\bfg}{\fc}, \quad \frac{1}{\bfh \fc}, \quad \frac{{}^{w_0^\dyn}\bfh\cdot {}^{w_0^\dyn \fc}}{\bfg}, \quad \frac{ \bfg\cdot {}^{w_0}\fc}{{}^{w_0^\dyn}\bfh},
\]are rational sections of the following line bundles:
\[
\bbM(-\la)_e, \quad\bbM(-z)_e^\dyn, \quad \bbM_{w_0}, \quad \bbM_{w_0}^\dyn, \quad. 
\]
\end{lemma}
\begin{proof}Lemma \ref{lem:D2} give  $\bbL(-\la)\cong \bbL^{-1}\otimes \calG$ and $\bbL(-z)\cong\bbL^{-1}\otimes \calH^{-1}$, which have the rational sections $\frac{\bfg}{\fc}$ and $\frac{1}{\fc \bfh}$, respectively.

Note that $w_0^*\calG=\calG^{-1}$, $w_0^\dynst \calH=\calH^{-1}$ and $w_0^*\calH=\calH$. We have
\begin{align*}
\bbM_{w_0}&=w_0^*\bbL\overset{\text{Def. }\ref{def:bbS}}=w_0^*\bbS_{w_0, w_0}\otimes w_0^\dynst \bbL\overset{\text{Lem. } \ref{lem:basicS}}=w_0^*(\calG\otimes \calH^{-1})\otimes w_0^\dynst \bbL=\calG^{-1}\otimes w_0^\dynst \calH\otimes w_0^\dynst \bbL,
\end{align*}
so $\frac{{}^{w_0^\dyn}\bfh\cdot {}^{w_0^\dyn \fc}}{\bfg}$ is a rational section. 

Lastly, we have 
\begin{align*}
\bbM_{w_0}^\dyn=w_0^\dynst\bbL\overset{\text{Def. }\ref{def:bbS}}=w_0^\dynst \bbS_{w_0, w_0}\otimes w_0^*\bbL\overset{\text{Lem. } \ref{lem:basicS}}=w_0^\dynst (\calG\otimes \calH^{-1})\otimes w_0^*\bbL=\calG\otimes w_0^\dynst \calH^{-1}\otimes w_0^*\bbL,
\end{align*}
so $\frac{ \bfg\cdot {}^{w_0}\fc}{{}^{w_0^\dyn}\bfh}$ is a rational section. 
\end{proof}

\begin{definition}\label{def:ellclass}
We  define the following elliptic classes
\begin{align}\label{eq:oppEll}
\Ell_{v}^{z,\la}&=T^{z,\la}_{vw_0}\bullet \frac{{}^{w_0^\dyn}\bfh\cdot {}^{w_0^\dyn}\fc}{\bfg}f_{w_0},&\Ell_{v}^{z,\la,\dyn}&=T^{z,\la,\dyn}_{vw_0}\bullet ^\dyn\frac{{}^{w_0}\fc\cdot \bfg}{{}^{w_0^\dyn}\bfh} f_{w_0}^\dyn. 
\end{align}
\end{definition}
By Lemma~\ref{lem:rationalsec}, and the definition of the actions $\bullet$ and $\bullet^\dyn$ in \S~\ref{subsec:act_periodic},   they  are rational sections of $\bbM$ and $\bbM^\dyn$ respectively. Here recall that $f_{w_0}$ and $f_{w_0}^\dyn$ are bookkeeping devices introduced in \S~\ref{subsec:Sproductrational}, indicating that the coefficients are considered as rational sections of $\bbM_{w_0}$ and $\bbM_{w_0}^\dyn$ respectively. 

Recall that under the Langlands duality $i$ (Lemma \ref{lem:i*T}) the classes $\Ell_{v}^{z,\la,\dyn}$ are identified with the elliptic classes of the Langlands dual system. One also note that $\Ell_{v}^{z,\la}$ and $\Ell_{v}^{z,\la,\dyn}$ live in $\oplus_{w\ge v}\bbM_w$ and  $\oplus_{w\ge v}\bbM_w^\dyn$ respectively.

To study the elliptic classes, recall the rational maps $\bbT^{z,\la}$ and $\bbT^{z,\la,\dyn}$ from Definition~\ref{def:bbT}. 
The following is the main theorem of this section, regarding elliptic classes.
\begin{theorem}\label{thm:main}
We have
\begin{align*}
\bbT^{z,\la}(\Ell^{z,\la}_{u})&={}^{u^\dyn}\fc f_{u^{-1}}^\dyn,&\bbT^{z,\la,\dyn}(\Ell_{u}^{z,\la,\dyn})&={}^{u}\fc f_{u^{-1}}. 
\end{align*}
\end{theorem} 

We give the proof in \S~\ref{subsec:trans}, which relies on Poincar\'e duality of elliptic classes below.

Recall that (Remark~\ref{rmk:ell_coh}) $\bbM$ and $\bbM^\dyn$ are models of equivariant elliptic cohomologies of $T^*G/B$ and $T^*G^\vee/B^\vee$ respectively. The theorem, rouphly speaking, asserts that the rational isomorphisms in Theorem~\ref{thm:inverse_maps} swamps elliptic classes with the fixed-point classes of the Langlands dual root datum. 

\subsection{Poincar\'e duality of elliptic classes}
Similar to Definition~\ref{def:ellclass}, 
we also define the following.
\begin{definition}\label{def:ellclass2}
We define the elliptic  classes as rational sections of $\bbM(-\la)$ and $\bbM^\dyn(-z)$, as follows:
\begin{align*}
\bfE_{v}^{z,-\la}&=T^{z,-\la}_{v}\bullet \frac{\bfg}{\fc} f_e,&
\bfE_{v}^{-z,\la,\dyn}&=T^{-z,\la,\dyn}_{v}\bullet^\dyn \frac{1}{\bfh \fc} f_e^\dyn.
\end{align*}
\end{definition}
By Lemma \ref{lem:rationalsec}, we know they are well defined and are support on $\oplus_{w\le v}\bbM_w$ and $\oplus_{w\le v}\bbM_w^\dyn$, respectively. 
The superscript $(z,-\la)$ (resp. $(-z,\la,\dyn)$) is to indicate that they are rational sections of $\bbM(-\la)$ (resp. $\bbM(-z)^\dyn$). 

We are going to show that via the Poincar\'e pairing, they are dual to the corresponding elliptic  classes $\calE_{u}^{z,\la}$ and $\calE_{u}^{z,\la,\dyn}$, respectively , and therefore are referred to as the opposite elliptic classes.

By using the definition, together with the expression of $T^?_v$ in \eqref{eq:Tlinearcomb}, one can write down $\bfE^?_{v}$  as linear combinations of $f_?$ or $f^\dyn_?$. We compute $\bfE^{z,-\la}_{v}$ as an example (see also \eqref{eq:res1} and $\eqref{eq:res2}$ below):
\begin{align}
\bfE^{z,-\la}_{v}&=T_{v}^{{z,-\la}}\bullet \frac{\bfg}{\fc}f_e\overset{\eqref{eq:Tlinearcomb}}=\sum_{w}\de_{v}^\dyn a^{z,-\la}_{v,w}\de_{w}\bullet \frac{\bfg}{\fc} f_e\overset{\eqref{eq:bullet}}=\sum_w {}^{v^\dyn w^{-1}}a^{z,-\la}_{v, w}\cdot \frac{\bfg}{{}^{v^\dyn}\fc} f_{w}\\
\label{eq:bfE-z-la}&\overset{\text{Lem.}~\ref{lem:a-iota}}=\sum_w\frac{{}^{w^{-1}}\bfg}{{}^{v^\dyn}\fc }\cdot a^{z,\la}_{v^{-1},w^{-1}}\cdot f_{w}. 
\end{align}
The rational sections in front of $f_{w}$ are called restriction coefficients.

\begin{theorem}[Poincar\'e Duality]\label{thm:Poin}
We have the following dualities
\begin{align}
\lag\bfE_{v}^{z, -\la},\Ell_{u}^{ z, \la}\rag_{\la}&=\de_{v,u}, &
\lag\bfE_{v}^{ -z, \la,\dyn},\Ell_{u}^{z, \la,\dyn}\rag_{z}^\dyn=\de_{v,u}.
\end{align}
\end{theorem}
\begin{proof}Note that $T^{z,-\la}_{v}$ is homogenous of degree $v^\dyn$, so we have
\begin{align*}
\lag \bfE^{z,-\la}_v, \Ell^{z,\la}_u\rag_{\la}&\overset{\text{Def.}~\ref{def:ellclass}}=\lag T^{z,-\la}_{v}\bullet \frac{\bfg}{\fc}f_e, T^{z,\la}_{uw_0}\bullet \frac{{}^{w_0^\dyn}\bfh\cdot {}^{w_0^\dyn }\fc}{\bfg}f_{w_0}\rag_{\la}\overset{\eqref{eq:adjT-la}}={}^{v^\dyn}\lag \frac{\bfg}{\fc} f_e, T^{z,\la}_{v^{-1}}T^{z,\la}_{uw_0}\bullet \frac{{}^{w_0^\dyn}\bfh\cdot {}^{w_0^\dyn }\fc}{\bfg}f_{w_0}\rag_{\la}.
\end{align*}
By the definition of the pairing $\lag\_,\_\rag_{\la}$ in \eqref{eq:D-la-pair}, it suffices to look at  the term of degree $e$ in  $T^{z,\la}_{v^{-1}}T^{z,\la}_{uw_0}\bullet \frac{{}^{w_0^\dyn}\bfh\cdot {}^{w_0^\dyn }\fc}{\bfg}f_{w_0}$, and by the definition of the $\bullet$-action in \eqref{eq:bullet}, it suffices to look at the  term of degree ${w_0}\in W$ (i.e., $\de_{w_0}$) in $T^{z,\la}_{v^{-1}}T^{z,\la}_{uw_0}=T^{z,\la}_{v^{-1}uw_0}$, which is equal to $0$ unless $u=v$. If $u=v$, then by \eqref{eq:aw0}, the term involving $\de_{w_0}^\dyn?\de_{w_0}$ is $\de_{w_0}^\dyn \frac{\bfg}{{}^{w_0^\dyn}\bfh}\de_{w_0}$, so in this case, we have 
\begin{align*}
\lag \bfE^{z,-\la}_v,\Ell^{z,\la}_v\rag_{\la}&={}^{v^\dyn}\lag \frac{\bfg}{\fc} f_e, \de_{w_0}^\dyn\frac{\bfg}{{}^{w_0^\dyn}\bfh}\de_{w_0}\bullet  \frac{{}^{w_0^\dyn}\bfh\cdot {}^{w_0^\dyn }\fc}{\bfg}f_{w_0}\rag_{\la}\overset{\eqref{eq:D-la-pair}}={}^{v^\dyn}(\frac{\bfg}{\fc}\cdot\fc \cdot \frac{1}{\bfg})=1.
\end{align*}

The second duality is proved similarly.
\end{proof}

\subsection{Auxilliary classes} \label{subsec:trans}

\begin{lemma}
The rational map $\bbT^{z,\la}$ is invertible.
\end{lemma}
\begin{proof}
The matrix $ (a^{z,\la}_{v,w})_{w,v\in W}$ is lower triangular  with nonzero (so invertible) diagonals. 
\end{proof}
Therefore, there exists for each $w,v\in W$ a $b_{w,v}^{z,\la}$, a rational section of 
\begin{equation}\label{eq:b}\sHom(\bbM_v^\dyn,\bbM_{w^{-1}})\cong v^{\dyn*}\bbS_{w,v}^{-1}=(w^{-1})^*\bbS_{w^{-1},v^{-1}}=\bbS'_{w^{-1}, v^{-1}}, 
\end{equation}
so that 
\[
\sum_w\fp^\dyn_v a^{z,\la}_{v,w}\fri_{w^{-1}}\circ\fp_{w^{-1}}b^{z,\la}_{w,u}\fri^\dyn_u=\de_{v,u}
\]
as a rational section of 
\[
\sHom(\bbM_u^\dyn,\bbM_{w^{-1}})\otimes \sHom(\bbM_{w^{-1}},\bbM_v^\dyn)\overset{\sim}\longrightarrow \sHom(\bbM_u^\dyn,\bbM_v^\dyn).
\]

Similar as $\bbT^{z,\la}$, the rational map $\bbT^{z,\la,\dyn}$ is also invertible. 
So there exists the inverse matrix $(b^{z,\la,\dyn}_{w,u})_{w,v\in W}$ with  $ b^{z,\la,\dyn}_{w,u}$ a rational section of 
\[u^* \bbS^{-1}_{u,w}\overset{\text{Lem. }\ref{lem:tensor}}=(w^{-1})^\dynst\bbS_{u^{-1}, w^{-1}}=\sHom(\bbM_u,\bbM^\dyn_{w^{-1}}).\]

\begin{remark}\label{rem:b}
Repeating the above,  the matrices $(b_{w,v}^{\pm z,\pm \la,\dyn})_{w,v}$ are defined as inverses of $(a_{w,v}^{\pm z,\pm \la, \dyn})$ (or without the superscript $\dyn$). Moreover,  variant forms of $\bbT^{z,\la}$ are defined similarly. See \S~\ref{subsec:invmat} below for more details.
\end{remark}

We define some auxiliary classes. 
\begin{align}\label{eq:T*}
  (T_u^{z,\la})^*&=\sum_{x\ge u}b^{z,\la}_{x^{-1}, u^{-1}}\cdot    {}^{u^\dyn}\fc  f_{x},&
(T_u^{z,\la, \dyn})^*&=\sum_{y\ge u}b^{z,\la, \dyn}_{y^{-1},u^{-1}}\cdot  {}^{u}\fc f^\dyn_{y}.
\end{align}

\begin{lemma}The classes  are rational sections of the modules $\bbM$ and $\bbM^\dyn$, respectively.
\end{lemma}
\begin{proof}
By definition (see \eqref{eq:b} and Remark \ref{rem:b}), we know $b^{z,\la}_{x ^{-1} ,u^{-1} }$ is a rational section of $\bbS'_{x, u}=x^*\bbS_{x, u}$. So $b^{z,\la}_{x^{-1} ,u^{-1} }\cdot {}^{u^\dyn}\fc$ is a rational section of 
\[
x^*\bbS_{x, u}\otimes (u^{-1} )^\dynst\bbL=x^*((x^{-1} )^*(u^{-1} )^\dynst \bbL^{-1}\otimes \bbL)\otimes (u^{-1} )^\dynst \bbL=x^*\bbL. 
\]
This shows that $(T^{z,\la}_u)^*$ is a rational section of $\bbM$. One can prove similar conclusion for $(T^{z,\la,\dyn}_u)^*$. 
\end{proof}

\begin{lemma}\label{lem:bfE-T*}We have the following dualities
\begin{align*}
\lag\bfE_v^{z, -\la}, (T^{z, \la}_u)^*\rag_{\la}&=\de_{v,u}. &
\lag\bfE_v^{-z, \la, \dyn}, (T^{ z, \la, \dyn}_u)^*\rag^\dyn_{z}&=\de_{v,u}.
\end{align*}
\end{lemma}

\begin{proof}We prove the first one as an example:
\begin{align*}
\lag \bfE_v^{z,-\la}, (T_u^{z,\la})^*\rag_{\la}&\overset{\eqref{eq:bfE-z-la}}= \lag\sum_w\frac{{}^{w^{-1} }\bfg}{{}^{v^\dyn}\fc}\cdot a^{z,\la}_{v^{-1} ,w^{-1} } \cdot f_{w}, \sum_{x}b^{z,\la}_{x^{-1} ,u^{-1} }\cdot{}^{u^\dyn}\fc  f_{x}\rag_{\la}\\
&\overset{\eqref{eq:D-la-pair}}=\sum_w\frac{{}^{w^{-1} } \bfg} {{}^{v^\dyn}\fc}\cdot a^{z,\la}_{v^{-1} ,w^{-1} }\cdot b^{z,\la}_{w^{-1} , u^{-1} }\cdot {}^{u^\dyn}\fc \cdot\frac{1}{{}^{w^{-1} }\bfg}=\de_{v,u}.\\
\end{align*}

\end{proof}

Combining Lemma \ref{lem:bfE-T*} and Theorem \ref{thm:Poin}, we get the following.
\begin{corollary}\label{cor:opp=*}We have 
\[
\Ell^{ z, \la}_v=(T^{ z, \la}_v)^*, \quad \Ell^{ z, \la,\dyn}_v=(T^{z, \la,\dyn}_v)^*. 
\]
\end{corollary}

With Corollary~\ref{cor:opp=*}, we are ready to prove the main result on elliptic classes Theorem~\ref{thm:main}.
\begin{proof}[Proof of Theorem~\ref{thm:main}]
We have
\begin{align*}
\bbT^{z,\la}(\Ell^{z,\la}_{u})&\overset{\text{ Def.}~\ref{def:bbT}}=(\sum_v T_{v^{-1} }^{z,\la})((T_u^{z,\la})^*)\overset{\text{ Cor.}~\ref{cor:opp=*}}=\sum_{v,w}\fp_{v^{-1} }^\dyn a^{z,\la}_{v^{-1} ,w^{-1} }\fri_{w}(\sum_xb^{z,\la}_{x^{-1} ,u^{-1} }\cdot  {}^{u^\dyn}\fc f_{x})\\
&=\sum_v\sum_wa^{z,\la}_{v^{-1} ,w^{-1} }b^{z,\la}_{w^{-1} ,u^{-1} }\cdot{}^{u^\dyn}\fc f_{v^{-1} }=\sum_v\de_{v,u}\cdot {}^{(u^{-1})^\dyn}\fc f_{v^{-1} }={}^{(u^{-1})^\dyn}\fc f_{u^{-1} }.
\end{align*}
The other equality is proved similarly. 
\end{proof}

\subsection{Proof of Theorem \ref{thm:inverse_maps}}
\label{subsec:proofs}

The following is a reformulation of  Theorem~\ref{thm:inverse_maps}.
\begin{theorem}\label{thm:invmat}
 The matrix $(a^{z,\la}_{v,w})_{w,v\in W}$ is the inverse to the matrix $(a^{z,\la,\dyn}_{v^{-1}, w^{-1}})_{w,v\in W}$. 
\end{theorem}
With the matrix coefficients $b^{z,\la}_{v,w}$ introduced in \S~\ref{subsec:trans}, the statement is equivalent to \[a^{z,\la,\dyn}_{v^{-1}, w^{-1}}=b^{z,\la}_{v,w}\] for any $v,w\in W$. 
\begin{proof}
We prove this by induction along the Bruhat order of $v$. Keep in mind that $(a_{v,w}^{z,\la})_{w,v}$ is lower triangular.

If $v=e$,  then $T^{z,\la,\dyn}_e=1$. Also we know that $a^{z,\la}_{e,e}=1$, which implies that $b^{z,\la} _{e,e}=1$. So the conclusion holds for $v=e$.
In the inductive process, for simplicity we write $T_\al^{z,\la}=\de_\al^\dyn p_\al\de_\al+\de_\al^\dyn q_\al$.
Note that we have in the same notation
\[
T^{z,\la, \dyn}_{\al}=\de_\al\frac{1}{p_\al}\de_\al^\dyn-\de_\al\frac{q_\al}{p_\al}.
\]
Assume that $a^{z,\la, \dyn}_{v^{-1},w^{-1}}=b^{z,\la} _{v,w}$ for any $w$, we want to show that $b^{z,\la} _{vs_\al, w}=a^{z,\la,\dyn}_{s_\al v^{-1}, w^{-1}}$ for any $w$. 

We first derive recursive formula for the $b$-coefficients. We have
\[T^{z,\la}_\al\bullet (T^{z,\la}_{w^{-1} })^*\overset{\text{Cor.}~\ref{cor:opp=*}}=T_\al^{z,\la}\bullet \Ell_{w^{-1} }^{z,\la}\overset{\text{ Def.}~\ref{def:ellclass}}=\Ell_{s_\al w^{-1} }^{z,\la}=(T^{z,\la}_{s_\al w^{-1} })^*.\]
Plug-in the definition $(T_{w^{-1} }^{z,\la})^*=\sum_{v\ge u}b_{v ,w }^{z,\la}\cdot {}^{(w^{-1} )^\dyn}\fc f_{v^{-1}}$,  we obtain 
\[
{}^{s_\al^\dyn}b^{z,\la} _{v ,w s_\al }=b^{z,\la} _{v s_\al, w }\cdot {}^{v}p_\al+b^{z,\la} _{v ,w}\cdot {}^{v }q_\al.
\]
This gives
\begin{equation}\label{eq:brecur}
b^{z,\la} _{vs_\al, w}={}^{s_\al^\dyn}b^{z,\la} _{v,ws_\al}\cdot \frac{1}{{}^v p_\al}-b^{z,\la} _{v,w}\cdot {}^{v}(\frac{q_\al}{ p_\al}). 
\end{equation}


We consider the $a$-coefficients. We have for each $w,v\in W$ a canonical isomorphism 
\[
(\bbS''_{w^{-1}, v^{-1}})^{-1}= w^*\bbL\otimes (v^{-1})^\dynst\bbL^{-1}= w^*\bbS_{w,v}=\bbS_{w,v}',
\] and hence we define 
\[
\overline{\{\cdot \}}:~\bbS'_{w,v}= (\bbS''_{w^{-1}, v^{-1}})^{-1}, \quad \de_w a\de_v^\dyn\mapsto\overline{\de_w a\de_v^\dyn}:= \de_{v^{-1}}^\dyn a\de_{w^{-1}}.
\]
Notice that both $\bbS'$ and $(\bbS'')^{-1}$ are algebra objects, and 
it is easy to see that $\overline{\{\cdot \}}$  is an anti-involution. Indeed,  we have 
\begin{align*}
\overline{\de_{w_1}a_1\de_{v_{1}}^\dyn\cdot \de_{w_2}a_2\de_{v_2}^\dyn}&=\overline{\de_{w_1w_2}{}^{(w_2)^{-1}}a_1\cdot {}^{v_1^\dyn}a_2\de_{v_1v_2}^\dyn}\\
&=\de_{v_2^{-1}v_1^{-1}}^\dyn {}^{(w_2)^{-1}}a_1\cdot {}^{v_1^\dyn}a_2\de_{w_2^{-1}w_1^{-1}}\\
&=\de^\dyn_{v_2^{-1}}a_2\de_{w_2^{-1}}\de_{v_1^{-1}}^\dyn a_1\de_{w_1^{-1}}\\
&=\overline{\de_{w_2}a_2\de_{v_2}^{\dyn}}\cdot \overline{\de_{w_1}a_1\de_{v_1}^\dyn}.
\end{align*}
In particular, we have 
\[
 \overline{T^{z,\la,\dyn}_{w_1}\cdot T^{z,\la,\dyn}_{w_2}}=\overline {T^{z,\la,\dyn}_{w_2}}\cdot \overline{T^{z,\la,\dyn}_{w_1}}. 
\]

It is easy to see that 
\[\overline{T^{z,\la,\dyn}_\al}=\de^\dyn_\al\frac{1}{p_\al}\de_\al-\frac{q_\al}{p_\al}\de_\al.\]
We then have 
\begin{align*}\sum_w\de_{w}^\dyn a^{z,\la,\dyn}_{s_\al v^{-1}, w^{-1}}\de_{vs_\al}=
&\overline{\sum_{w}\de_{s_\al v^{-1}}a_{s_\al v^{-1}, w^{-1}}^{z,\la,\dyn}\de^\dyn_{w^{-1}}}\\
&=\overline{T_{s_\al  v^{-1}}^{z,\la,\dyn}}=\overline{T_{\al}^{z,\la,\dyn}T_{v^{-1}}^{z,\la,\dyn}}=\overline{T^{z,\la,\dyn}_{v^{-1}}}\cdot \overline{T^{z,\la,\dyn}_\al}\\
&=\overline{\sum_w\de_{v^{-1}} a^{z,\la,\dyn}_{v^{-1},w^{-1}}\de_{w^{-1}}^\dyn}\cdot (\de_\al^\dyn \frac{1}{p_\al}\de_\al-\frac{q_\al}{p_\al}\de_\al)\\
&=\sum_w\de_{w}^\dyn a^{z,\la,\dyn}_{v^{-1},w^{-1}}\de_{v}\cdot (\de_\al^\dyn \frac{1}{p_\al}\de_\al-\frac{q_\al}{p_\al}\de_\al)\\
&=\sum_w\de_{ws_\al}^\dyn {}^{s_\al^\dyn}a^{z,\la,\dyn}_{v^{-1},w^{-1}}\cdot\frac{1}{{}^vp_\al}\de_{vs_\al}-\sum_w\de_w^\dyn a^{z,\la,\dyn}_{v^{-1},w^{-1}}\cdot {}^v(\frac{q_\al}{p_\al})\de_{vs_\al}\\
&=\sum_w\de_{w}^\dyn {}^{s_\al^\dyn}a^{z,\la,\dyn}_{v^{-1},s_\al w^{-1}}\cdot\frac{1}{{}^vp_\al}\de_{vs_\al}-\sum_w\de_w^\dyn a^{z,\la,\dyn}_{v^{-1},w^{-1}}\cdot {}^v(\frac{q_\al}{p_\al})\de_{vs_\al}\\
&=\sum_w \de_w^\dyn\left(  {}^{s_\al^\dyn}a^{z,\la,\dyn}_{v^{-1},s_\al w^{-1}}\cdot \frac{1}{{}^vp_\al}-a^{z,\la,\dyn}_{v^{-1},w^{-1}}\cdot {}^v(\frac{q_\al}{p_\al})\right)\de_{vs_\al} .
\end{align*}
Therefore, 
\[a^{z,\la,\dyn}_{s_\al v^{-1}, w^{-1}}={}^{s_\al^\dyn}a^{z,\la,\dyn}_{v^{-1},s_\al w^{-1}}\cdot \frac{1}{{}^vp_\al}-a^{z,\la,\dyn}_{v^{-1},w^{-1}}\cdot {}^v(\frac{q_\al}{p_\al}).\]
Comparing with the recursive formula for $b^{z,\la}_{v,w}$ in \eqref{eq:brecur}, we see that 
 $b^{z,\la}_{v,w}=a^{z,\la,\dyn}_{v^{-1}, w^{-1}}$. 
 \end{proof}

We can now prove Theorem~\ref{thm:inverse_maps}.
\begin{proof}[Proof of Theorem \ref{thm:inverse_maps}] We only prove $\bbT^{z,\la}\circ \bbT^{z,\la,\dyn}=\id$.  
We have 
\begin{align*}
\bbT^{z,\la}\circ \bbT^{z,\la,\dyn}&=\sum_uT_u^{z,\la}\circ \sum_vT^{z,\la,\dyn}_{v^{-1}}\\
&=\sum_u\sum_{w_1}\fp_u^\dyn a^{z,\la}_{u,w_1}\fri_{w_1^{-1}}\circ \sum_{v}\sum_{w_2}\fp_{v^{-1}}a^{z,\la, \dyn}_{v^{-1},w_2^{-1}}\fri_{w_2}^\dyn\\
&=\sum_{u,w_2}\fp_u^\dyn\left(\sum_v a^{z,\la}_{u,v} a^{z,\la, \dyn}_{v^{-1},w_2^{-1}}\right)\fri_{w_2}^\dyn\\
&\overset{\text{ Thm. }  \ref{thm:invmat}}=\sum_{u,w_2}\fp_u^\dyn\left(\sum_v a^{z,\la}_{u,v} b^{z,\la}_{v, w_2}\right)\fri_{w_2}^\dyn\\
&=\sum_{u,w_2}\de_{u,w_2}\fp^\dyn_u\fri^\dyn_{w_2}=\sum_u\fp^\dyn_u\fri^\dyn_u=\id_{\bbM(\la)^\dyn}. 
\end{align*}
\end{proof}

\section{A formula of the transition matrix/restriction formula} \label{sec:restriction}
In this section, we provide a formula of the transition  matrix between the Weyl group elements and the dynamical DL operators. Using \eqref{eq:bfE-z-la} and  \eqref{eq:res1},  this  gives restriction formulas of the elliptic  classes. 

\subsection{The formula for the transition matrices} 
Let $T_\al$ be one of $T^{\pm z,\pm \la}_\al$ or their counterparts for the Langlands dual system. 
We write
\[T_\al=\de_\al^\dyn\fa(z_\al, \la_{\al^\vee})+\de_\al^\dyn \fb(z_\al, \la_{\al^\vee})\de_\al, \]
and write 
\[T_v=\sum_{w\le v }\de_v^\dyn a_{v,w}\de_w. 
\]
Specializing $T_\al$ to $T^{\pm z,\pm \la}_\al$, the $a_{v,w}$ above becomes $a^{\pm z,\pm \la}_{v,w}$. 
Similar for the Langlands dual system.

For any reduced sequence $v=s_{i_1}\cdots s_{i_l}$, denote $I=(i_1,...,i_l)$. Define $\beta_l^\vee=\al^\vee_{i_l}$ and $\beta_j^\vee=s_{i_l}\cdots s_{i_{j+1}}(\al^\vee _{i_j}), j=l-1,...,1$.  For any $J\subset[l]$, denote $J_0=\emptyset$ and $J_j=J\cap [j]$ for $j\le l-1$, $w_{I,J_j}=\prod_{k\in J_j}s_{i_k}\in W$, and $\gamma_{j}=w_{I,J_{j-1}}(\al_{i_j})$. Clearly the set $\{\gamma_1,...,\gamma_l\}$ depends on $I$ and $J$, but the sequence $(\be_1^\vee ,...,\be^\vee _l)$ depends only on $I$. Moreover, $\{\be_1^\vee ,\cdots ,\be_l^\vee \}=\Phi^\vee (v^{-1})$. 
\begin{theorem}\label{thm:res}
We have 
\[
a_{v,w}=\sum_{J\subset [l], w_{I,J}=w}\zeta_J, \text{ where }\zeta_J:=\prod_{j\in J}\fb(z_{\gamma_j}, \la_{\be_j^\vee})\prod_{j\not \in J}\fa(z_{\gamma_j}, \la_{\be_j^\vee}). 
\]
\end{theorem}
\begin{proof} By definition,  $a_{v,w}$ is the sum of terms, one for each $J\subset [l]$ with $w_{I,J}=w$. For each $J$, we have 
\[
\prod_{j\in J}\de_{\al_{i_j}}^\dyn \fb^\la(z_{\la_{i_j}}, \la_{\al^\vee_{i_j}})\de_{\al_{i_j}}\cdot \de_{\al_{i_j}}^\dyn \prod_{j\not \in J}\fa^\la(z_{\la_{i_j}}, \la_{\al^\vee_{i_j}}),
\]
where the product should follow the order $j=1,...,l$. We then move the $\de_\al^\dyn$ to the left and $\de_\al$ to the right, and obtain the conclusion. 
\end{proof}

\begin{remark} By combining the formula for the corresponding $a$-coefficients with \eqref{eq:res1}, \eqref{eq:res2} and their analogues, one obtains a  formula for the restriction coefficients of the elliptic classes. In a forthcoming paper \cite{LZZ}, we will derive a root polynomial formula for the $b$-coefficients, generalizing the Billey and Graham-Willems formulas~\cite{billey, ajs, graham, willems}. This will be a more advantageous formula for the restriction coefficients of the elliptic classes, for reasons which will be specified. 
\end{remark}
\subsection{The coefficients $a^{z,\la}_{v,w}$}\label{subsec:T_res} In the remaining part of this paper, we  focus on  $T_\al=T_{\al}^{z,\la}$ version. The other versions are similar. Recall that 
\[T_{\al}^{z,\la}=\de_\al^\dyn\frac{\theta(\hbar)\theta(z_\al+\la_{\al^\vee})}{\theta(z_\al)\theta(\hbar+\la_{\al^\vee})}+\de_\al^{\dyn}\frac{\theta(\hbar-z_\al)\theta(-\la_{\al^\vee})}{\theta(z_\al)\theta(\hbar+\la_{\al^\vee})}\de_\al.\]
From Theorem \ref{thm:res}, we have
\begin{equation}\label{eq:T_zeta}
\zeta_J=\frac{\prod_{j\in J}\theta(\hbar-z_{\gamma_j})\theta(-\la_{{\be_j^\vee}})\cdot \prod_{j\not\in J}\theta(\hbar)\theta(z_{\gamma_j}+\la_{\beta_{j}^\vee})}{
\prod_{j=1}^l\theta(z_{\gamma_j})\theta(\hbar+\la_{\be_j^\vee})}.
\end{equation}

\begin{example}We compute all coefficients for the type $A_2$ using Theorem \ref{thm:res}. In this case there are two simple roots $\al_1,\al_2$. For simplicity,  denote $\de_i=\de_{s_i}$, $\de_i^\dyn=\de_{s_i}^\dyn, z_{i+j}=z_{\al_i+\al_j}, \la_{i+j}=\la_{\al_i^\vee+\al_j^\vee}$, $T_{ij}^{z,\la}=T_{s_is_j}^{z,\la}$. Then we have:
\begin{align*}
T^{z,\la}_{1}&=\de_1^\dyn\frac{\theta(\hbar)\theta(z_1+\la_{1})}{\theta(z_1)\theta(\hbar+\la_{1})}+\de_1^{\dyn}\frac{\theta(\hbar-z_1)\theta(-\la_{1})}{\theta(z_1)\theta(\hbar+\la_{1})}\de_1.\\
T^{z,\la}_{2}&=\de_2^\dyn\frac{\theta(\hbar)\theta(z_2+\la_{2})}{\theta(z_2)\theta(\hbar+\la_{2})}+\de_2^{\dyn}\frac{\theta(\hbar-z_2)\theta(-\la_{2})}{\theta(z_2)\theta(\hbar+\la_{2})}\de_2.\\
T^{z,\la}_{12}&=\de_{12}^\dyn\frac{1}{\theta(\hbar+\la_{1+2})\theta(\hbar+\la_2)}\bigg(\frac{\theta(\hbar)^2\theta(z_1+\la_{1+2})\theta(z_2+\la_2)}{\theta(z_1)\theta(z_2)}+\frac{\theta(\hbar-z_1)\theta(-\la_{1+2})\theta(\hbar)\theta(z_{1+2}+\la_2)}{\theta(z_1)\theta(z_{1+2})}\de_1\\
&+\frac{\theta(\hbar)\theta(z_1+\la_{1+2})\theta(\hbar-z_2)\theta(-\la_2)}{\theta(z_1)\theta(z_2)}\de_2+\frac{\theta(\hbar-z_1)\theta(-\la_{1+2})\theta(\hbar-z_{1+2})\theta(-\la_2)}{\theta(z_1)\theta(z_{1+2})}\de_{12}\bigg).\\
T^{z,\la}_{21}&=\de_{21}^\dyn\frac{1}{\theta(\hbar+\la_{1+2})\theta(\hbar+\la_1)}\bigg(\frac{\theta(\hbar)^2\theta(z_2+\la_{1+2})\theta(z_1+\la_1)}{\theta(z_2)\theta(z_1)}+\frac{\theta(\hbar-z_2)\theta(-\la_{1+2})\theta(\hbar)\theta(z_{1+2}+\la_1)}{\theta(z_2)\theta(z_{1+2})}\de_2\\
&+\frac{\theta(\hbar)\theta(z_2+\la_{1+2})\theta(\hbar-z_1)\theta(-\la_1)}{\theta(z_2)\theta(z_1)}\de_1+\frac{\theta(\hbar-z_2)\theta(-\la_{1+2})\theta(\hbar-z_{1+2})\theta(-\la_1)}{\theta(z_2)\theta(z_{1+2})}\de_{21}\bigg).\\
T^{z,\la}_{121}&=\de_{121}^\dyn\frac{1}{\theta(\hbar+\la_1)\theta(\hbar+\la_2)\theta(\hbar+\la_{1+2})}\bigg(\frac{\theta(\hbar)^3\theta(z_1+\la_2)\theta(z_2+\la_{1+2})\theta(z_1+\la_1)}{\theta(z_1)^2\theta(z_2)}\\
&+\frac{\theta(\hbar)^2\theta(z_1+\la_2)\theta(z_2+\la_{1+2})\theta(\hbar-z_1)\theta(-\la_1)}{\theta(z_1)^2\theta(z_2)}\de_1\\
&+\frac{\theta(\hbar)^2\theta(\hbar-z_1)\theta(-\la_2)\theta(z_{1+2}+\la_{1+2})\theta(z_{-1}+\la_1)}{-\theta(z_1)^2\theta(z_{1+2})}\de_1\\
&+\frac{\theta(\hbar)\theta(\hbar-z_1)\theta(-\la_2)\theta(z_{1+2}+\la_{1+2})\theta(\hbar+z_1)\theta(-\la_1)}{-\theta(z_1)^2\theta(z_{1+2})}\\
&+\frac{\theta(\hbar)^2\theta(z_1+\la_2)\theta(\hbar-z_2)\theta(-\la_{1+2})\theta(z_{1+2}+\la_1)}{\theta(z_1)\theta(z_2)\theta(z_{1+2})}\de_2\\
&+\frac{\theta(\hbar)\theta(z_1+\la_2)\theta(\hbar-z_2)\theta(-\la_{1+2})\theta(-\la_1)\theta(\hbar-z_{1+2})}{\theta(z_1)\theta(z_2)\theta(z_{1+2})}\de_{21}\\
&+\frac{\theta(\hbar)\theta(\hbar-z_1)\theta(-\la_2)\theta(-\la_{1+2})\theta(\hbar-z_{1+2})\theta(z_2+\la_1)}{\theta(z_1)\theta(z_2)\theta(z_{1+2})}\de_{12}\\
&+\frac{\theta(\hbar-z_1)\theta(\hbar-z_2)\theta(\hbar-z_{1+2})\theta(-\la_1)\theta(-\la_2)\theta(-\la_{1+2})}{\theta(z_1)\theta(z_2)\theta(z_{1+2})}\de_{121}\bigg).
\end{align*}
In the expansion of $T_{121}^{z,\la}$, one can use Fay's Trisecant Identity to verify that the coefficients for $\de_1$ coincides with that for $\de_2$ if one switches $1$ and $2$, and also show that the coefficient for $\de_e=1$ is symmetric with respect to $1$ and $2$. This shows that $T_{121}^{z,\la}=T_{212}^{z,\la}$. That is, the braid relation for $A_2$ is satisfied. 
\end{example}

\begin{example}\label{ex:1}Let us consider  the operator $T^{z,\la}_\al$ in the  $A_3$ case with simple roots $\al_1,\al_2,\al_3$. For simplicity, denote $\al_{i\pm j}=\al_i\pm \al_j$, $\la_{i\pm j}=\la_{\al_i^\vee\pm \al_j^\vee}$, $z_{i\pm j}=z_{\al_i}\pm z_{\al_j}$. Note that $\al_{i+j}^\vee=\al_i^\vee+\al_j^\vee$.  Let $v=s_1s_2s_1s_3s_2$ and $w=s_1s_2$. So 
\[\be_1^\vee =\al_3^\vee , \be_2^\vee =\al^\vee _{1+2+3},\be^\vee _3=\al_{1+2}^\vee ,\be_4^\vee =\al_{2+3}^\vee ,  \be_5^\vee =\al_2^\vee . \]
There are three subsets $J$ of $[5]$ with $w_{I,J}=w$, namely, 
\[
J=\{1,1,-,-,-\}, ~J'=\{1,-,-,-,1\},~ J''=\{-,-,-,1,-,1\}.
\]
We compute $\zeta_{J'}$ as an example. In this case 
\[
\gamma_1=\al_1,\gamma_2=\al_{1+2}, \gamma_3=-\al_1,\gamma_4=\al_3,\gamma_5=\al_{1+2}.
\]
So 
\[
\zeta_{J'}=\frac{\theta(\hbar)^3\theta(z_{1+2}+\la_{1+2+3})\theta(z_{-1}+\la_{1+2})\theta(z_{3}+\la_{2+3})\theta(-\la_{3})\theta(-\la_{2})\theta(\hbar-z_{1})\theta(\hbar-z_{1+2})}{\theta(z_1)\theta(z_{1+2})\theta(z_{-1})\theta(z_{3})\theta(z_{1+2})\theta(\hbar+\la_{3})\theta(\hbar+\la_{1+2+3})\theta(\hbar+\la_{1+2})\theta(\hbar+\la_{2+3})\theta(\hbar+\la_{2})}.
\]
\end{example}

\subsection{Restrictions of line bundles} \label{subsec:restr_linebdl}
Recall from Theorem \ref{thm:Trational}, $T^{z,\la}_{v}$ is a rational section of $\bbS''$. So $a^{z,\la}_{v,w}$ is a rational section of $\bbS''_{w,v}=v^\dynst \bbS_{w,v}$. Using Lemma \ref{lem:basic-bundle} and  Lemma \ref{lem:basicS}, we have
\begin{align}\label{eq:bbS''res-la}
\bbS''_{w,v}|_\la&=(v^\dynst \bbS_{w,v})|_\la=\bbS_{w,v}|_{v(\la)}=\calO(v(\la)-w_\bullet \la)),\\
\label{eq:bbS''res-z}\bbS''_{w,v}|_z&=(v^\dynst \bbS_{w,v})|_z=v^\dynst \calO(z-v_\bullet w^{-1}(z))=\calO(v^{-1}_\bullet z-w^{-1}(z)). 
\end{align}
Let $v=s_{i_1}\cdots s_{i_k}$ and denote $\mu_j=s_{i_1}\cdots s_{i_{j-1}}\al_{i_j}$ and $\mu_j^\vee=s_{i_1}\cdots s_{i_{j-1}}\al^\vee_{i_j}$,  then it is easy to see that 
\begin{align*}
v(\la)-\la&=-\sum_{j}\la_{\al_{i_j}^\vee}\mu_j,& \quad v(z)-z&=-\sum_jz_{\al_{i_j}}\mu_j^\vee.\\
v(\hbar\rho)-\hbar\rho&=-\hbar\sum_j\mu_j,&v(\hbar\rho^\vee)-\hbar\rho^\vee&=-\hbar\sum_j\mu_j^\vee.
\end{align*}
These identities can be used in calculations.

\begin{example}
We continue with Example \ref{ex:1}, and illustrate Theorem \ref{thm:Trational}. 
We first compute $\bbS''_{w,v}|_\la$, and get 
\begin{align*}
\bbS''_{s_1s_2, s_1s_2s_1s_3s_2}|_z&=\calO((2\hbar-z_2)\al_1^\vee+(4\hbar-z_2-z_3)\al_2^\vee+(3\hbar-z_1-z_2-z_3)\al_3^\vee),\\
\bbS''_{s_1s_2,s_1s_2s_1s_3s_2}|_\la&=\calO(2\hbar-\la_3)\al_1+(\hbar-\la_1-\la_2-\la_3)\al_2+(-\la_2-\la_3)\al_3).
\end{align*}

Consider  $J=\{1,1,-,-,-\}$. We have
\[
\gamma_1=\al_1,\gamma_2=\al_{1+2},\gamma_3=\al_{2}, \gamma_4=\al_{1+2+3}, \gamma_5=\al_{-1-2}.  
\]
so 
\[
\zeta_{J}=\frac{\theta(\hbar)^3\theta(\hbar-z_1)\theta(\hbar-z_{1+2})\theta(-\la_{3})\theta(-\la_{1+2+3})\theta(z_2+\la_{1+2})\theta(z_{1+2+3}+\la_{2+3})\theta(z_{-1-2}+\la_{2})}{\theta(z_{1})\theta(z_{1+2})\theta(z_{2})\theta(z_{1+2+3})\theta(z_{-1-2})\theta(\hbar+\la_{2})\theta(\hbar+\la_{2+3})\theta(\hbar+\la_{1+2})\theta(\hbar+\la_{1+2+3})\theta(\hbar+\la_{3})}.\]
One can verify that this is a rational section of $\bbS''_{s_1s_2,s_1s_2s_1s_3s_2}$ by using the argument in the proof of Corollary \ref{cor:zero} below (or see Example \ref{ex:C2} below).

\end{example}

\section{Further properties and examples}\label{sec:examples}
The method in \S~\ref{subsec:restr_linebdl} has further applications. We work with notations from Theorem \ref{thm:res}, for the operators $T^{z,\la}_w$. We would like to thank Allen Knutson for helpful suggestions in rephrasing some of the results in this section.

\begin{corollary}\label{cor:zero}For any $\la\in \catA^\vee, z\in \catA$, and for any $J\subset [l]$ so that $w_{I,J}=w$,   we have the following two identities in $\catA^\vee$ and $\catA$, respectively:
\begin{align}\label{eq:zero}
v(\la)-w_\bullet (\la)&=\hbar\sum_{j\in J}\gamma_j-\sum_{j\not \in J} \la_{\beta_j^\vee}\gamma_j, &
v_\bullet^{-1}z-w^{-1}(z)&=\hbar\sum_{j=1}^l\beta_j^\vee-\sum_{j\not \in J}z_{\gamma_j}\beta_j^\vee. 
\end{align}
\end{corollary}
\begin{proof}
We use the expression \eqref{eq:T_zeta}, which is a rational section of $\bbS''_{w,v}$. First, fix $\lambda\in \catA^\vee$ so that  $\zeta_J$ only depends on $z\in \catA$. If $j\in J$, then $\frac{\theta(\hbar-z_{\gamma_j})}{\theta(z_{\gamma_j})}$ is a rational section of $\calO(\hbar \gamma_j)$. If $j\not \in J$, then $\frac{\theta(z_{\gamma_j}+\la_{\beta_j^\vee})}{\theta(z_{\gamma_j})}$ is a rational section of $\calO(-\la_{\beta_j^\vee}\gamma_j)$. Therefore,  $\zeta_J$ is a rational section of 
\[\calO(\hbar\sum_{j\in J}\gamma_j-\sum_{j\not \in J} \la_{\beta_j^\vee}\gamma_j). \]
Comparing with $\bbS''_{w,v}|_\la$  in \eqref{eq:bbS''res-z}, we obtain
\[\hbar\sum_{j\in J}\gamma_j-\sum_{j\not \in J} \la_{\beta_j^\vee}\gamma_j=v(\la)-w_\bullet (\la).\]

Similarly, fix $z\in\catA$. If $j\in J$, then $\frac{\theta(-\la_{\beta_{j}^\vee})}{\theta(\hbar+\la_{\beta_j^\vee})}$ is a rational section of $\calO(\hbar\be_j^\vee)$, and if $j\not\in J$, then $\frac{\theta(z_{\gamma_j}+\la_{\be_j^\vee})}{\theta(\hbar+\la_{\be_j^\vee})}$ is a rational section of $\calO((\hbar-z_{\gamma_j})\be_j^\vee)$. Therefore, $\zeta_J$ is a rational section of 
\[
\calO(\hbar\sum_{j\in J}\be_j^\vee+\sum_{j\not\in J}(\hbar-z_{\gamma_j})\beta_j^\vee)=\calO(\hbar \sum_{j=1}^l\beta_j^\vee-\sum_{j\not\in J}z_{\gamma_j}\be_j^\vee).
\]
Comparing with $\bbS''_{w,v}|_z$, we obtain 
\[
\hbar\sum_{j=1}^l\beta_j^\vee-\sum_{j\not\in J}z_{\gamma_j}\beta_j^\vee=v_\bullet^{-1}z-w^{-1}(z). 
\]
\end{proof}

\begin{example}\label{ex:C2}
We  illustrate the proof of  Corollary \ref{cor:zero} in the $C_2$ case.  The Langlands dual system is of type $B_2$. There are  two simple roots $\al_1,\al_2$ such that \[s_1(\al_2)=\al_2+2\al_1,~s_2(\al_1)=\al_1+\al_2,~ s_1(\al_2^\vee)=\al_2^\vee+\al_1^\vee,~ s_2(\al_1^\vee)=\al_1^\vee+2\al_2^\vee.\]
Let $v=s_{1}s_2s_1s_2, w=s_1$, then by computations  using \eqref{eq:bbS''res-la} and \eqref{eq:bbS''res-z}, we obtain
\begin{align*}
\bbS''_{w,v}|_\la&=\calO((\hbar-\la_1-2\la_2)\al_1+(-\la_1-2\la_2)\al_2),\\
\bbS''_{w,v}|_z&=\calO((3\hbar-z_1-z_2)\al_1^\vee+(4\hbar-2z_1-2z_2)\al_2^\vee).
\end{align*}

On the other hand, with $v=s_{1}s_2s_1s_2$, we have
\[
\be_1^\vee=\al_1^\vee ,\be_2^\vee=\al_1^\vee+\al_2^\vee ,\be_3^\vee =\al_{1}^\vee+2\al_2^\vee,\be_4^\vee =\al_2^\vee. 
\]
Consider $J=\{1,-,-,-\}$, so 
\[
\gamma_1=\al_1,\gamma_2=\al_2+2\al_1, \gamma_3=-\al_1,\gamma_4=\al_2+2\al_1. 
\]
Then \eqref{eq:T_zeta} gives
\begin{align*}
\zeta_J&=\frac{
\theta(\hbar-z_{\al_1})\theta(-\la_{\al_1^\vee})\theta(\hbar)^3\theta(z_{\al_2+2\al_1}+\la_{\al_1^\vee+\al_2^\vee})\theta(z_{-\al_1}+\la_{\al_1^\vee+2\al_2^\vee})\theta(z_{2\al_1+\al_2}+\la_{\al_2^\vee})}{
\theta(z_{\al_1})\theta(z_{2\al_{1}+\al_2})\theta(-z_{\al_1})\theta(z_{2\al_1+\al_2})\theta(\hbar+\la_{\al_1^\vee})\theta(\hbar+\la_{\al_1^\vee+\al_2^\vee})\theta(\hbar+\la_{\al_1^\vee+2\al_2^\vee})\theta(\hbar+\la_{\al_2^\vee})}\\
&=\frac{\theta(\hbar-z_{\al_1})}{\theta(z_{\al_1})}\cdot \frac{\theta(z_{2\al_1+\al_2}+\la_{\al_1^\vee+\al_2^\vee})}{\theta(z_{2\al_1+\al_2})}\cdot \frac{\theta(z_{-\al_1}+\la_{\al_1^\vee+2\al_2^\vee})}{\theta(-z_{\al_1})}\cdot \frac{\theta(z_{2\al_1+\al_2}+\la_{\al_2^\vee})}{\theta(z_{2\al_1+\al_2})}\\
&\cdot \frac{\theta(-\la_{\al_1^\vee})\theta(\hbar)^3}{\theta(\hbar+\la_{\al_1^\vee})\theta(\hbar+\la_{\al_1^\vee+\al_2^\vee})\theta(\hbar+\la_{\al_1^\vee+2\al_2^\vee})\theta(\hbar+\la_{\al_2^\vee})}.
\end{align*}
Fix $\la\in \catA^\vee$, and let $z\in \catA$ vary. The first fraction is a zero section of $\calO(\hbar)$ with $z_{\al_1}\in E$, so it is a rational section of $\calO(\hbar\al_1)=\chi_{\al_1}^*\calO(\hbar)$. Similar, the second, third, fourth fractions are rational sections of \[\calO(-\la_{\al_1^\vee+\al_2^\vee}(2\al_1+\al_2)),~ \calO(-\la_{\al_1^\vee+2\al_2^\vee}(-\al_1)),~ \calO(-\la_{\al_2^\vee}(2\al_1+\al_2)).\]
The last fraction is a constant since there is no $z$ variable involved. 
So $\zeta_J$ (viewed as on $\catA\times\{\la\}$)  is a rational section of 
\begin{align*}
&\calO(\hbar\al_1)\otimes \calO(-\la_{\al_1^\vee+\al_2^\vee}(2\al_1+\al_2))\otimes\calO(-\la_{\al_1^\vee+2\al_2^\vee}(-\al_1))\otimes  \calO(-\la_{\al_2^\vee}(2\al_1+\al_2))\\
&=\calO(\hbar \al_1-\la_{\al_1^\vee+\al_2^\vee}(2\al_1+\al_2)+\la_{\al_1^\vee+2\al_2^\vee}\al_1-\la_{\al_2}(2\al_1+\al_2))\\
&=\calO((\hbar-\la_{\al_1^\vee}-2\la_{\al_2^\vee})\al_1-(\la_{\al_1^\vee}+2\la_{\al_2^\vee})\al_2)=\bbS''_{w,v}|_\la.
\end{align*}

One can similarly fix $z$, and let $\la\in \catA^\vee$  vary. Rewriting $\zeta_J$, one can show that it is a rational section of $\bbS''_{w,v}|_z$. This verifies that $\zeta_J$ is a rational section of $\bbS''_{w,v}$. 
\end{example}

The following is an  application in root systems.

\begin{corollary}\label{cor:la} For any $\mu\in X^*, \mu^\vee\in X_*$ and $v\ge w$, with notations as in Theorem \ref{thm:res},  we have 
\begin{align}
\label{eq:la}v(\mu)-w(\mu)&=-\sum_{j\not\in J}\lag \mu, \be_j^\vee\rag \gamma_j, &
v^{-1}(\mu^\vee)-w^{-1}(\mu^\vee)&=-\sum_{j\not\in J}\lag \mu^\vee, \gamma_j\rag \be_j^\vee. 
\end{align}
In particular, $\sum_{j\not\in J}\gamma_j\otimes \be_j^\vee\in X^*\otimes X_*$ is independent of $I$ or $J$. 
\end{corollary}

\begin{proof}We prove the second identity of \eqref{eq:la} first. Let $z=\hbar\otimes \mu^\vee$, then  the second identity in Corollary \ref{cor:zero} can be written as
\[
\hbar\otimes (v^{-1}(\mu^\vee-\rho^\vee)+\rho^\vee-w^{-1}(\mu^\vee))=\hbar\otimes ( \sum_{j=1}^l\be_j^\vee-\sum_{j\not\in J}\lag \mu^\vee, \gamma_j\rag \be_j^\vee),
\]
which implies the following identity:
\[
v^{-1}(\mu^\vee-\rho^\vee)+\rho^\vee-w^{-1}(\mu^\vee)=\sum_{j=1}^l\be_j^\vee-\sum_{j\not\in J}\lag \mu^\vee, \gamma_j\rag \be_j^\vee.
\]
It is easy to see that $\rho^\vee-v^{-1}(\rho^\vee)=\sum_{j=1}^l\be_j^\vee$, so the identity  in \eqref{eq:la} holds. The first identity of \eqref{eq:la} can be proved similarly.
\end{proof}

\subsection{Logarithmic degree} Motivated by the previous corollary, for any $w\in W$,  we consider the following map
\[
\up_{w}:X^*\to X^*, \mu\mapsto \mu-w(\mu), 
\] 
Via the isomorphism $\Hom_\bbZ(X^*,X^*)\cong X^*\otimes X_*$, the map $\up_{w}$ is uniquely determined by an element of $X^*\otimes X_*$. That is, if  $\up_{w}=\sum_j\gamma_j\otimes \be_j^\vee\in  X^*\otimes X_*$, then 
\begin{equation}
\label{eq:up}
\up_{w}(\mu)=\sum_{j}\lag\mu, \be_j^\vee\rag\gamma_j, \quad \mu\in X^*. 
\end{equation}
Moreover,  the isomorphism $\Hom_\bbZ(X_*,X_*)\cong X^*\otimes X_*$ makes  $\up_{w}$ into a map 
\[
\up_{w}:X_*\to X_*, \quad \up_{w}(\mu^\vee)=\sum_j\lag\mu^\vee, \gamma_j\rag\be_j^\vee.
\]
Note that for any $\mu'\in X^*, \mu^\vee\in X_*$, we have $
\lag w(\mu'), \mu^\vee\rag=\lag \mu', w^{-1}(\mu^\vee)\rag.$ Therefore, 
\[\up_{w}(\mu^\vee)=\mu^\vee-w^{-1}(\mu^\vee). \]

We call this element $\up_{w}$  the {\it logarithmic degree} of the element $w\in W$. Clearly from \eqref{eq:up}, one can write $\up_{w}=\sum_{j}\gamma_j\otimes \be_j^\vee$  so that  $\gamma_j$ are positive roots and $\be_j^\vee$ are simple roots for all $j$. We will not distinguish  $\up_{w}\in X^*\otimes X_*$ from the element it defines in $\End_\bbZ(X^*)$ or $\End_\bbZ(X_*)$. 

\begin{remark}The notion $\Theta_w$ coincides with the map $\chi$ in \cite[\S~3.1]{KS23}. Moreover, $\Theta_w-\Theta_v$ controls the quasi-periods of the sections of $\bbS_{w,v}$, which corresponds to \cite[Proposition 2]{KS23}. 
\end{remark}
For convenience, in the following proposition we denote the Weyl group action on $X^*$ and $X_*$ by $w$ and $w^\vee$ respectively.
\begin{prop}
\begin{enumerate}
\item $\up_{e}=0$. 
\item $u(\up_{v}-\up_w)=\up_{uv}-\up_{uw}$, $u^\vee(\up_{v}-\up_w)=\up_{vu^{-1}}-\up_{wu^{-1}}$. 
\end{enumerate}
\end{prop}
\begin{proof}(1)  is  obvious. For (3), write $\up_{v}-\up_w=\sum_{j}\ga_{j}\otimes \be_{j}^\vee$, so  \[(\up_v-\up_w)(\mu)=w(\mu)-v(\mu)=\sum_{j}\lag\mu, \be_{j}^\vee\rag\ga_{j}.\]
  Then 
\[
u(\up_v-\up_w)=\sum_{j}u(\gamma_j)\otimes \be_j^\vee, 
\]
 \[
u(\up_v-\up_w)(\mu)=\sum_{j}\lag\mu, \be_j^\vee\rag u(\gamma_j)=u((\up_{v}-\up_w)(\mu))=u(w\mu)-v(\mu))=uv(\mu)-uw(\mu)=(\up_{uv}-\up_{uv})(\mu). 
\]
Therefore, $u(\up_{v}-\up_w)=\up_{uv}-\up_{ uw}$. The second part can be proved by using the identity $(\up_{v}-\up_w)(\mu^\vee)=w^{-1}(\mu^\vee)-v^{-1}(\mu^\vee).$
\end{proof}

\begin{remark}Suppose $w\le v$. From Corollary \ref{cor:zero} and Corollary \ref{cor:la}, one has
\[\sum_{\ga\in \Phi(w)}\ga=\sum_{j\in J}\gamma_j, \quad \sum_{\be^\vee\in \Phi^\vee(v^{-1})}\be^\vee=\sum_{j\in I}\be_j^\vee, \quad \up_v-\up_w=\sum_{j\not\in J}\gamma_j\otimes \be_j^\vee.\]  In other words, one can compute the RHSs of these identities without using the notations of Theorem \ref{thm:res}. Then via \eqref{eq:bbS''res-la} and \eqref{eq:bbS''res-z}, these three elements uniquely determine the line bundle $\bbS''_{w,v}$. 
\end{remark}
For simplicity, we write $\la\otimes \mu^\vee=\la \mu^\vee$ in $X^*\otimes X_*$.

\begin{example} In the $A_3$ case,  we compute $\up_{v}-\up_w$ with $v=s_1s_2s_3s_1s_2s_1$ and $w=s_1$.  We have  
\[
\up_v=\al_1\al_1^\vee+\al_2\al_{1+2}^\vee+\al_3\al_{1+2+3}^\vee+\al_1\al_2^\vee+\al_2\al_{2+3}^\vee+\al_1\al_3^\vee, \quad \up_w=\al_1\al_1^\vee,
\]
so
\[
\up_v-\up_w=\al_2\al_{1+2}^\vee+\al_3\al_{1+2+3}^\vee+\al_1\al_2^\vee+\al_2\al_{2+3}^\vee+\al_1\al_3^\vee. 
\]
\end{example}

\begin{example}
In the case of $C_2$, we compute $\up_v-\up_w$ with $v=s_1s_2s_1, w=e$. We have
\[
\up_v-\up_e=\al_1\al_1^\vee+\al_2(\al_1^\vee+\al_2^\vee)+\al_1(\al_1^\vee+2\al_2^\vee).
\]
\end{example}

\begin{example}
We consider the $A_2$ case. Let $v=s_1s_2s_1$. We compute $\up_{s_1s_2s_1}-\up_w$ for  all  $w\le s_1s_2s_1$. Then 
\begin{align*}
\up_{s_1s_2s_1}-\up_e&=\al_1\al_2^\vee+\al_2\al_{1+2}^\vee+\al_1\al_1^\vee=\al_{1+2}\al_{1+2}^\vee.\\
\up_{s_1s_2s_1}-\up_{s_1}&=\al_1\al_2^\vee+\al_2\al_{1+2}^\vee.\\
\up_{s_1s_2s_1}-\up_{s_2}&=\al_2\al_1^\vee+\al_{1}\al_{1+2}^\vee.\\
\up_{s_1s_2s_1}-\up_{s_1s_2}&=\al_2\al_1^\vee.\\
\up_{s_1s_2s_1}-\up_{s_2s_1}&=a_1\al_2^\vee.\\
\end{align*}
\end{example}

\appendix
\section{A complete list of  elliptic classes}\label{sec:complete}
In this section, we collect all elliptic classes that one could consider in our setting. This is for the sake of completeness, and for the convenience of the readers. The proofs of the properties in this section are similar to the corresponding ones in the previous section. 

\subsection{Complete list of elliptic classes}\label{subsec:complete_list}
Here is a complete list of elliptic classes one can consider, including the ones we defined before:
\begin{align*}
\bfE_v^{z,\la}&=T^{z,\la}_{v}\bullet \fc f_e,&\Ell_v^{z,\la}&=T^{z,\la}_{vw_0}\bullet \frac{{}^{w_0^\dyn}\bfh\cdot {}^{w_0^\dyn}\fc}{\bfg}f_{w_0},\\
\bfE_v^{z,\la,\dyn}&=T^{z,\la,\dyn}_{v}\bullet ^\dyn\fc f_e^\dyn,&\Ell_v^{z,\la,\dyn}&=T^{z,\la,\dyn}_{vw_0}\bullet^\dyn \frac{{}^{w_0}\fc\cdot \bfg}{{}^{w_0^\dyn}\bfh} f_{w_0}^\dyn,\\
\bfE_v^{z,-\la}&=T^{z,-\la}_{v}\bullet \frac{\bfg}{\fc} f_e,&\Ell_{v}^{z,-\la}&=T^{z,-\la}_{vw_0}\bullet \frac{\bfh}{{}^{w_0^\dyn}\fc}f_{w_0},\\
\bfE_v^{z,-\la,\dyn}&=T^{z,-\la, \dyn}_{v}\bullet^\dyn \frac{\bfg}{\fc} f_e^\dyn,&\Ell_v^{z,-\la,\dyn}&=T^{z,-\la,\dyn}_{vw_0}\bullet  ^\dyn \frac{\bfg\cdot {}^{w_0}\bfg}{\bfh\cdot  {}^{w_0}\fc}f_{w_0}^\dyn,\\
\bfE_v^{-z,\la}&=T^{-z,\la}_{v}\bullet \frac{1}{\bfh \fc} f_e,&\Ell_v^{-z,\la}&=T^{-z,\la}_{vw_0}\bullet \frac{1}{{}^{w_0}\bfg\cdot {}^{w_0^\dyn}\fc}f_{w_0},\\
\bfE_v^{-z,\la,\dyn}&=T^{-z,\la,\dyn}_{v}\bullet^\dyn \frac{1}{\bfh\fc} f_e^\dyn,&\Ell_v^{-z,\la,\dyn}&=T^{-z,\la,\dyn}_{vw_0}\bullet ^\dyn \frac{{}^{w_0}\bfg}{\bfh\cdot {}^{w_0^\dyn}\bfh\cdot {}^{w_0}\fc}f_{w_0}^\dyn,\\
(T^{z,\la}_{u})^*&=\sum_{x\ge u}b^{z,\la}_{x^{-1},u^{-1}}\cdot {{}^{u^\dyn}\fc}f_{x},&(T_u^{z,\la,\dyn})^*&=\sum_{y\ge u}b^{z,\la,\dyn}_{y^{-1},u^{-1}}\cdot {}^{u}\fc f_{y}^\dyn,\\
  (T_u^{z,-\la})^*&=\sum_{x\ge u}b^{z,-\la}_{x^{-1}, u^{-1}}\cdot   \frac{\bfg}{ {}^{u^\dyn}\fc } f_{x},&
(T^{z,-\la, \dyn}_u)^*&=\sum_{y\ge u}b^{z,-\la,\dyn}_{y^{-1},u^{-1}}\cdot \frac{{}^{u}\bfg}{{}^{u}\fc}f^\dyn_{y},\\
(T_u^{-z,\la})^*&=\sum_{x\ge u}b^{-z,\la}_{x^{-1},u^{-1}}\cdot \frac{1}{{}^{u^\dyn}\bfh\cdot {}^{u^\dyn}\fc}f_{x},&
(T_u^{-z,\la, \dyn})^*&=\sum_{y\ge u}b^{-z,\la, \dyn}_{y^{-1},u^{-1}}\cdot \frac{1}{\bfh\cdot {}^{u}\fc}f^\dyn_{y}.
\end{align*}
Using the definitions of the periodic modules, one can verify that the $\bfE^?_v$ and $\Ell_v^?$ are rational sections of the corresponding vector bundles. Similarly, one can use the definition of the $b$-coefficients to check that the $(T^?_u)^*$ are all rational sections of the corresponding vector bundles. Note also that in the definition of the classes $\Ell_v^?$, we apply $T^?_{{vw_0}}$ on certain  rational sections of degree $w_0$ or $w_0^\dyn$. These sections are chosen so as to obtain the Poincar\'e duality without any normalization factor. 

\subsection{Rational maps on the other periodic modules}\label{subsec:invmat} For the operators 
\[T^{z,-\la}_v,\quad  T^{-z,\la}_v, \quad T_v^{z,-\la, \dyn},\quad T_v^{-z,\la,\dyn},\]we can repeat the construction in \S~\ref{subsec:trans}, and define
\begin{align*}
\bbT^{z,-\la}&=\sum_v\ep_vT_v^{z,-\la}, & \bbT^{-z,\la}&=\sum_v\ep_vT_v^{-z,\la}, \\
\bbT^{ z, - \la, \dyn}&=\sum_{v}\ep_vT^{ z,  -\la, \dyn}, &\bbT^{ -z,  \la, \dyn}&=\sum_{v}\ep_vT^{ -z,  \la, \dyn}.
\end{align*}
Their domains and codomains are indicated in the diagram below.
Note that we add the sign $\ep_v=(-1)^{\ell(v)}$  in these four identities. Indeed, similar as Theorem \ref{thm:inverse_maps}, we have 
 \begin{align*}
\bbT^{z,-\la}(\Ell_u^{z,-\la})&=\frac{\ep_u \bfg}{{}^{u^\dyn}\fc}f_{u^{-1}}^\dyn, & \bbT^{z,-\la,\dyn}(\Ell_u^{z,-\la,\dyn})&=\frac{\ep_u \cdot {}^{u}\bfg}{{}^{u}\fc}f_{u^{-1}}.\\ \bbT^{-z,\la}(\Ell_u^{-z,\la})&=\frac{\ep_u }{{}^{u^\dyn}\bfh\cdot {}^{u^\dyn}\fc}f_{u^{-1}}^\dyn, &\bbT^{-z,\la,\dyn}(\Ell_u^{-z,\la,\dyn})&=\frac{\ep_u }{\bfh\cdot {}^{u}\fc}f_{u^{-1}}.
\end{align*}
The rational sections in front of the $f$'s are naturally induced by the prefixed rational sections $\fc, \bfg, \bfh$. 
Moreover, we have
\begin{enumerate}
\item The matrices $(a^{z,-\la}_{v,w})_{v,w\in W}$ is the inverse  of the matrix $(\ep_w\ep_v a^{z,-\la,\dyn}_{v^{-1}, w^{-1}})_{v,w\in W}$. Consequently, $\bbT^{z,-\la}$ is the inverse of $\bbT^{z,-\la,\dyn}$.
\item The matrices $(a^{-z,\la}_{v,w})_{v,w\in W}$ is the inverse  of the matrix $(\ep_w\ep_v a^{-z,\la,\dyn}_{v^{-1}, w^{-1}})_{v,w\in W}$.
Consequently, $\bbT^{-z,\la}$ is the inverse of $\bbT^{-z,\la,\dyn}$.
\end{enumerate}
Their relation with the periodic modules are summarized in the diagram below.

\subsection{A  diagram}\label{subsec:diagram}
Relations of the periodic modules and elliptic classes are  summarized in the following triangular prism:
\begin{equation}\label{eq:diagram}
\xymatrix{\bfE^{z,\la}_v,\Ell_v^{z,\la}=(T_v^{z,\la})^*\ar@{.>}[r]&\bbM\ar@{-->}@/^1pc/[ddd]^{\bbT^{z,\la}}\ar@{<->}[rr]^-{\lag\_,\_\rag_{\la}}\ar@{<->}[dr]^-{\lag\_,\_\rag_{z}}&&\bbM(-\la)\ar@{-->}@/^1pc/[ddd]^{ \bbT^{z,-\la}}&\bfE_v^{z,-\la}, \ar@{.>}[l]\Ell_v^{z,-\la}=(T^{z,-\la}_v)^*\\
&&\bbM(-z)\ar@{-->}@/^1pc/[ddd]^{\bbT^{-z,\la}} &&\bfE_v^{-z,\la}, \Ell_v^{-z,\la}=(T^{-z,\la}_v)^*\ar@{.>}[ll]\\
&&&&&\\
\bfE_v^{z,\la,\dyn},  \Ell_v^{z,\la,\dyn}=(T^{z,\la,\dyn})^*\ar@{.>}[r]&\bbM^\dyn\ar@{-->}@/^1pc/[uuu]^-{\bbT^{z,\la,\dyn}} \ar@{<->}[rr]\ar@{<->}[dr]&&\bbM(-\lambda)^\dyn\ar@{-->}@/^1pc/[uuu]^{\bbT^{z,-\la,\dyn}}&\bfE_{v}^{z,-\la,\dyn}, \Ell_v^{z,-\la,\dyn}=(T^{z,-\la,\dyn}_v)^*\ar@{.>}[l]\\
&&\bbM(-z)^\dyn\ar@{-->}@/^1pc/[uuu]^{\bbT^{-z,\la,\dyn}}&&\bfE_v^{-z,\la,\dyn},\Ell_v^{-z,\la,\dyn}=(T^{-z,\la,\dyn}_v)^*\ar@{.>}[ll]}
\end{equation}
The top face denotes the periodic modules for the  chosen root system and the bottom face denotes the modules for the Langlands dual system. The solid two-head arrows on the top  and bottom faces denote the four Poincar\'e Pairings $\lag\_,\_\rag_{\la}$, $\lag\_,\_\rag_{z}$, $\lag\_,\_\rag_{\la}^\dyn$, $\lag\_,\_\rag_{z}^\dyn$. The vertical dashed arrows denote the rational maps $\bbT$ induced by the dynamical DL operators, which are mutually inverses to each other (Theorem \ref{thm:inverse_maps}). The twelve elliptic  classes are  rational sections of the corresponding periodic modules, indicated by the dotted arrows. Moreover, the $\bfE$ classes are always dual to the $\calE$ classes, and the $\calE$ classes are always equal to the $T^*$ classes. 

\subsection{ 3d mirror symmetry}\label{sec:mirror}
The transition matrices in Theorem \ref{thm:invmat} (and also in \S~\ref{subsec:invmat}) are not only inverses to each other, they are also equal to each other after certain shifts by $w_0$ and normalization by the rational sections $\bfg, \bfh$. We establish such a statement in this section.

In addition to the relation between the $a$-coefficients in Remark \ref{rem:acoeffrelation}, Theorem \ref{thm:invmat} also gives the following result.
\begin{lemma}\label{lem:trans_mat}We have 
\begin{align*}
\bfg\cdot \ep_x\ep_ua^{z,-\la, \dyn}_{x, u}={}^{(uw_0)^\dyn}\bfh\cdot {}^{(uw_0)^\dyn x^{-1}}a^{z,-\la}_{uw_0, xw_0}, \quad \bfg  \cdot a_{x, u}^{z,\la,\dyn}={}^{u^\dyn}\bfh\cdot {}^{(uw_0)^\dyn x^{-1}} a^{z,\la}_{uw_0, xw_0}. 
\end{align*}
\end{lemma}
\begin{proof}
We have
\begin{align*}
\Ell_{u}^{z,-\la}&=T^{z, -\la}_{uw_0}\bullet \frac{\bfh}{{}^{w_0^\dyn}\fc} f_{w_0}\\
&=\sum_{w}\de_{uw_0}^\dyn a^{z,-\la}_{uw_0, w}\de_w\bullet \frac{\bfh}{{}^{w_0^\dyn}\fc}  f_{w_0}\\
&=\sum_w {}^{(uw_0)^\dyn}(\frac{\bfh}{{}^{w_0^\dyn}\fc} )\cdot {}^{(uw_0)^\dyn (ww_0)^{-1}}a^{z,-\la}_{uw_0, w}f_{ww_0}\\
&=\sum_x \frac{{}^{(uw_0)^\dyn}\bfh}{{}^{u^\dyn}\fc} \cdot {}^{(uw_0)^\dyn x^{-1}}a^{z,-\la}_{uw_0, xw_0}f_{x}. 
\end{align*}
Recall from Corollary \ref{cor:opp=*}, $\calE_{u}^{z,-\la}=(T_{u}^{z,-\la})^*$. Comparing with the formula for $(T^{z,-\la}_{u})^*$ in \S~\ref{subsec:complete_list},  we get 
\[
\frac{{}^{(uw_0)^\dyn}\bfh}{{}^{u^\dyn}\fc}\cdot {}^{(uw_0)^\dyn x^{-1}}a^{z,-\la}_{uw_0, xw_0}=b^{z,-\la}_{x^{-1},u^{-1}}\cdot \frac{\bfg}{{}^{u^\dyn}\fc},
\]
which gives
\[
\bfg\cdot b^{z,-\la}_{x^{-1},u^{-1}}={}^{(uw_0)^\dyn}\bfh\cdot {}^{(uw_0)^\dyn x^{-1}}a^{z,-\la}_{uw_0, xw_0}.
\]
By \S~\ref{subsec:invmat}, the matrix $(\ep_w\ep_v a^{z,-\la,\dyn}_{v^{-1}, w^{-1}})_{v,w\in W}$ is inverse to  $(a^{z,-\la}_{v,w})_{v,w\in W}$, so 
\[
\bfg\cdot \ep_x\ep_ua^{z,-\la, \dyn}_{x, u}=\bfg\cdot b^{z,-\la}_{x^{-1},u^{-1}}={}^{(uw_0)^\dyn}\bfh\cdot {}^{(uw_0)^\dyn x^{-1}}a^{z,-\la}_{uw_0, xw_0}.
\]
This proves the first identity. Applying $\Gamma_\la$ (see Remark \ref{rem:Gamma} below), and use the fact that 
\[\Gamma_\la ({}^{(uw_0)^\dyn}\bfh)={}^{u^\dyn}\bfh,\] we obtain the second identity. 
\end{proof} 

\subsection{Comparison with a 3d-mirror symmetry result of Rimanyi and Weber}
Similar to the calculation in \eqref{eq:bfE-z-la}, we have 
\begin{align}\label{eq:res1}
\bfE^{z,\la}_v&=\sum_w\frac{{}^{w^{-1}}\bfg}{\bfg}\cdot a^{z,-\la}_{v^{-1},w^{-1}}\cdot {}^{v^\dyn}\fc f_{w}=:\sum_{w}c^{z,\la}_{v,w}f_{w}.\\
\label{eq:res2}
\bfE_{v}^{z,\la, \dyn}&=\sum_{w}\frac{\bfg}{{}^{v}\bfg}a^{z,-\la, \dyn}_{v^{-1},w^{-1}} \cdot {}^{v}\fc f_{w}^\dyn=:\sum_wc^{z,\la,\dyn}_{v,w}f_{w}^\dyn.
\end{align}
From the first identity of Theorem \ref{lem:trans_mat}, we obtain
\[
{}^{v^{-1}} c^{z,\la,\dyn}_{v,w}\ep_v\ep_w \cdot {}^{v^{-1}w_0}\bfg={}^{(w^{-1}w_0)^\dyn}\bfh\cdot {}^{(w^{-1}w_0)^\dyn }c_{w_0w, w_0v}^{z,\la}
\]
Explicitly, we have
\begin{equation}
{}^{v^{-1}}c_{v,w}^{z,\la,\dyn}\cdot \prod_{\al>0}\frac{\theta(\hbar+z_{v^{-1}(\al)})}{\theta(z_\al)}=\prod_{\al>0}\frac{\theta(\hbar+\la_{w^{-1}(\al^\vee)})}{\theta(\la_{\al^\vee})}\cdot {}^{(w^{-1}w_0)^\dyn}c_{w_0w,w_0v}^{z,\la}. 
\end{equation}
This identity recovers \cite[Theorem 7]{RW22}, which provides a comparison of restriction coefficients of the elliptic classes (after certain normalization) and those of the Langlands dual system. Such comparison was predicted by the 3d mirror symmetry \cite[\S~1.3.1]{AO21}.  A large class of examples are given in \cite{RSVZ19, RSZV22}. 

\begin{remark}\label{rem:Gamma}
 Let $\Gamma_\la$ (resp. $\Gamma_z$) be the operation on rational sections by formally replacing $\la$ by $-\la$ (resp. $z$ by $-z$). Then by comparing the formula of the operators in \S~\ref{subsec:DL_operator} and use the calculation in \S~\ref{subsec:braid}, we easily obtain the following four identities.
 \begin{align}
\label{eq:aGammala}
a^{z,\la}_{v,w}&=\Gamma_\la(a^{z,-\la}_{v,w}),& ~a^{z,\la,\dyn}_{v,w}&=\Gamma_z(a^{-z, \la,\dyn}_{v,w}).\\
\label{eq:aGammaz}a^{z,\la,\dyn}_{v,w}&=\ep_v\ep_w\Gamma_\la(a^{z,-\la, \dyn}_{v,w}), & a^{z,\la}_{v,w}&=\ep_v\ep_w\Gamma_z(a^{-z,\la}_{v,w}).
\end{align}
 Looking at their inverses matrices, we can also obtain the last four identities.
\begin{align}
\label{eq:bGammala}b^{z,\la}_{v,w}&=\Gamma_\la(b^{z,-\la}_{v,w}),& ~b^{z,\la,\dyn}_{v,w}&=\Gamma_z(b^{-z, \la,\dyn}_{v,w}).\\
\label{eq:bGammaz}b^{z,\la,\dyn}_{v,w}&=\ep_w\ep_v\Gamma_\la(b^{z,-\la, \dyn}_{v,w}), & b^{z,\la}_{v,w}&=\ep_w\ep_v\Gamma_z(b^{-z,\la}_{v,w}).
\end{align}
Using these identities,  properties of one transition matrix implies the same property of another.  
\end{remark}

\newcommand{\arxiv}[1]
{\texttt{\href{http://arxiv.org/abs/#1}{arXiv:#1}}}
\newcommand{\doi}[1]
{\texttt{\href{http://dx.doi.org/#1}{doi:#1}}}
\renewcommand{\MR}[1]
{\href{http://www.ams.org/mathscinet-getitem?mr=#1}{MR#1}}


\begin{thebibliography}{00}
\bibitem[AO21]{AO21} M. Aganagic, A. Okounkov, {\em Elliptic stable envelope},  J. Amer. Math. Soc. 34 (2021), no. 1, 79–133.

\bibitem[AJS94]{ajs}
H. Andersen, J. Jantzen, W. Soergel, 
{\em Representations of quantum groups at a $p$th root of unity and of semisimple groups in characteristic $p$: independence of $p$},
{Ast\'erisque} (1994), no. 220, 321 pp.

\bibitem[Bi99]{billey}
S.~Billey, 
{\em Kostant polynomials and the cohomology ring for {$G/B$}}, 
{Duke J. Math.} 96 (1999), no. 1, 205–224.

\bibitem[BL03]{BL03} L. Borisov, A. Libgober, {\it Elliptic genera of singular varieties},  Duke Math. J. 116 (2003), no. 2, 319–351.

\bibitem[BR23]{BR23} T. M. Botta, R. Rimanyi, {\em Bow varieties: stable envelopes and threir 3d mirror symmetry}, Preprint, \arxiv{2308.07300}.


\bibitem[SGA3]{SGA3} 
M.~Demazure, A.~Grothendieck, 
{\em Sch\'emas en groupes I, II, III}, 
Lecture Notes in Math 151, 152, 153, Springer-Verlag, New York, 1970,
and new edition :  Documents Math\'ematiques 7, 8,  Soci\'et\'e Math\'ematique de France, 2003.  

\bibitem[FGA05]{FGA} B. Fantechi, L. G\"ottsche, L. Illusie, S. L. Kleiman, N. Nitsure, A. Vistoli,
{\em Fundamental Algebraic Geometry: Grothendieck's FGA Explained}, Mathematical surveys and monographs 123, American Mathematical Soc., 2005.

\bibitem[FRV18]{FRV18} G. Felder, R. Rim\'anyi, A. Varchenko, {\it Elliptic dynamical quantum groups and equivariant elliptic cohomology}, 
SIGMA Symmetry Integrability Geom. Methods Appl. 14 (2018), Paper No. 132, 41 pp.

\bibitem[Ga12]{Ga12} N. Ganter, {\em The elliptic Weyl character formula},  Compos. Math. 150 (2014), no. 7, 1196-1234.


\bibitem[GKV95]{GKV95} V.~Ginzburg, M.~Kapranov,  E.~Vasserot, {\em Elliptic algebras and equivariant elliptic cohomology},
Preprint, (1995). \arxiv{9505012}

\bibitem[GKV97]{GKV97} V.  Ginzburg, M. Kapranov,  E. Vasserot, {\it Residue construction of Hecke algebras} Adv. Math. 128 (1997), no. 1, 1–19. 

\bibitem[Gr02]{graham}
W. Graham, 
{\em Equivariant $K$-theory and Schubert varieties}, 
Preprint, 2002.

\bibitem[H20]{H20}  T. Hikita, {\em Elliptic canonical bases for toric hyper-K\"ahler manifolds}, \arxiv{2003.03573}.

\bibitem[KS22]{KS22} Y. Kononov, A. Smirnoy, {\em Pursuing quantum difference equations I: stable envelopes of subvarieties}, 
Lett. Math. Phys. 112 (2022), no. 4, Paper No. 69, 25 pp.

\bibitem[KS23]{KS23} Y. Kononov, A. Smirnov, {\em Pursuing quantum difference equations II: 3D mirror symmetry}, 
Int. Math. Res. Not. IMRN(2023), no. 15, 13290–13331.

\bibitem[LZZ]{LZZ} C. Lenart, C. Zhong, {\em The Root polynomials for elliptic classes}, in preparation. 

\bibitem[L99]{L} G. Lusztig, {\em Bases in equivariant $K$-theory. II}, Represent. Theory 3 (1999), 281--353.

\bibitem[P03]{P03} A.  Polishchuk, {\it Abelian varieties, theta functions and the Fourier transform} Cambridge Tracts in Mathematics, 153. Cambridge University Press, Cambridge, 2003. xvi+292 pp. ISBN: 0-521-80804-9.

 \bibitem[RSVZ19]{RSVZ19} R. Rim\'anyi, A.  Smirnov, A. Varchenko,  Z. Zhou,  {\it Three-dimensional mirror self-symmetry of the cotangent bundle of the full flag variety}, SIGMA Symmetry Integrability Geom. Methods Appl. 15 (2019), Paper No. 093, 22 pp.

\bibitem[RSZV22]{RSZV22} R. Rim\'anyi, A.  Smirnov,  Z. Zhou, A. Varchenko,  {\it Three-dimensional mirror symmetry and elliptic stable envelopes}, Int. Math. Res. Not. IMRN 2022, no. 13, 10016–10094.

\bibitem[RW20]{RW20} R. Rim\'anyi, A. Weber, {\it Elliptic classes of Schubert varieties via Bott-Samelson resolution}, J. of Topology, Vol. 13, Issue 3, September 2020, 1139-1182, DOI 10.1112/topo.12152

\bibitem[RW22]{RW22}R. Rim\'anyi, A. Weber, {\it Elliptic classes on Langlands dual flag varieties}, Comm. in Contemp. Math., Vol. 24, No. 1 (2022) 2150014, DOI: 10.1142/S0219199721500140

\bibitem[S80]{Sie} C.L.~Siegel, {\em Advanced Analytic Number Theory},  Second edition. Tata Institute of Fundamental Research Studies in Mathematics, 9. Tata Institute of Fundamental Research, Bombay, 1980. v+268 pp.

\bibitem[SZ22]{SZ22} A. Smirnov, Z. Zhou, {\em 3d mirror symmetry and quantum  K -theory of hypertoric varieties}, Adv. Math. 395 (2022), Paper No. 108081, 61 pp.





\bibitem[Wil04]{willems}
M.~Willems, 
{\em Cohomologie et {$K$}-th{\'e}orie {\'e}quivariantes des
  vari{\'e}t{\'e}s de {B}ott-{S}amelson et des vari{\'e}t{\'e}s de drapeaux}, 
{Bull. Soc. Math. France} 132 (2004), no. 4, 569–589.

\bibitem[ZZ22]{ZZ22}
G. Zhao and C. Zhong, {\em Representations of the elliptic affine Hecke algebras}, Adv. Math. 395 (2022), Paper No. 108077, 75 pp.


\bibitem[ZZ23]{ZZ23} G. Zhao and C. Zhong, {\em A Langlands duality of elliptic Hecke algebras}, Preprint, \arxiv{2310.13460}
\end{thebibliography}
\end{document}